\documentclass{article}

\usepackage{lineno,hyperref}
\usepackage{amsmath}
\usepackage{amssymb}
\usepackage{amsthm}
\modulolinenumbers[5]
\usepackage{enumitem}
\usepackage[all]{xy}
\SelectTips{cm}{12}

\newtheorem{lemma}{Lemma}
\newtheorem*{lemma*}{Lemma}
\newtheorem{prop}{Proposition}
\newtheorem{theorem}{Theorem}
\newtheorem*{theorem*}{Theorem}

\theoremstyle{definition}

\newcommand\dotminus{\mathbin{\dot{-}}}

\newcommand\dotoplus{\mathbin{\dot{\oplus}}}

\DeclareMathOperator\elem{E}

\DeclareMathOperator\stmap{st}
\DeclareMathOperator\glin{GL}

\DeclareMathOperator\stlin{St}
\DeclareMathOperator\gstlin{GSt}

\DeclareMathOperator\Cent{C}

\DeclareMathOperator\Ker{Ker}

\newcommand\e{{\mathrm e}}
\newcommand\R{{\mathcal R}}

\newcommand\leqt{\trianglelefteq}

\newcommand\op{{\mathrm{op}}}
\newcommand\id{{\mathrm{id}}}

\DeclareMathOperator\Image{Im}
\DeclareMathOperator\Spec{Spec}

\newcommand\fp{{\mathrm{fp}}}

\DeclareMathOperator\ev{ev}
\DeclareMathOperator\colim{colim}

\newcommand{\up}[2]{{^{#1}\!{#2}}}

\newcommand{\Set}{\mathbf{Set}}
\newcommand{\Group}{\mathbf{Grp}}

\newcommand{\Ring}{\mathbf{Ring}}

\DeclareMathOperator\Pro{Pro}
\DeclareMathOperator\Ind{Ind}
\DeclareMathOperator\Ex{Ex}

\DeclareMathOperator\Sub{Sub}
\newcommand{\Cat}{\mathbf{Cat}}

\title{
	Locally isotropic Steinberg groups I. \\
	Centrality of the
	\(
		\mathrm K_2
	\)-functor
}

\author{
	Egor Voronetsky\thanks{
		Research is supported by the Russian Science Foundation grant 19-71-30002.
	} \\
	Chebyshev Laboratory, \\
	St. Petersburg State University, \\
	14th Line V.O., 29B, \\
	Saint Petersburg 199178 Russia \\
}

\begin{document}
\maketitle

\begin{abstract}
	We begin to study Steinberg groups associated with a locally isotropic reductive group \(G\) over a arbitrary ring. We propose a construction of such a Steinberg group functor as a group object in a certain completion of the category of presheaves. We also show that it is a crossed module over \(G\) in a unique way, in particular, that the
	\(
		\mathrm K_2
	\)-functor is central. If \(G\) is globally isotropic in a suitable sense, then the Steinberg group functor exists as an ordinary group-valued functor and all such abstract Steinberg groups are crossed modules over the groups of points of \(G\).
\end{abstract}


\section{Introduction}

In recent paper \cite{isotropic-elem} we initiated the study of point groups of locally isotropic reductive groups, i.e. the groups
\(
	G(K)
\), where \(K\) is a unital commutative ring, \(G\) is a reductive group scheme over \(K\), and
\(
	G_{\mathfrak p}
\) is sufficiently isotropic for every prime ideal
\(
	\mathfrak p \in \Spec(K)
\). Unlike the previous results in this direction (see e.g. \cite{iso-centralizer, iso-perf, iso-elem-norm, iso-a1-inv, iso-norm-str, iso-survey}) we do not require the existence of proper parabolic subgroups of \(G\) itself, i.e. there are no global isotropic conditions. A typical example of such \(G\) is the automorphism group of a finitely generated projective module
\(
	P_K
\) such that the rank of \(P\) at every prime ideal is sufficiently large.

In that paper we proposed a general definition of the elementary subgroup
\(
	\elem_G(K) \leq G(K)
\) if the local isotropic rank of \(G\) is at least \(2\) and proved that it is both normal and perfect (excluding the case of rank \(2\) and small residue fields). Moreover, it turned out that the functor
\(
	\elem_G({-})
\) on the category of \(K\)-algebras is ``scheme generated'', i.e. there exists a scheme morphism
\(
	T \colon \mathbb A^N_K \to G
\) such that
\[
	\elem_G(R) = \bigl\langle T(R^N) \bigr\rangle
\]
for all unital commutative \(K\)-algebras \(R\).

The elementary subgroup is important for at least two reasons. On the one hand, it plays the central role in the group structure of
\(
	G(K)
\). Namely, under some assumptions (i.e. for
\(
	G = \mathrm{GL}_n
\),
\(
	n \geq 3
\)) it is a perfect normal subgroup, and if \(K\) is in addition finite dimensional, then
\(
	G(K) / \elem_G(K)
\) is solvable. Moreover, there exists a nice description of all subgroups of
\(
	G(K)
\) normalized by
\(
	\elem_G(K)
\) (including all normal subgroups) in terms of ideals of \(K\) and relative elementary subgroups. On the other hand,
\[
	\mathrm K_1^G(K) = G(K) / \elem_G(K)
\]
is the unstable (and non-split in general) analogue of the \(
	\mathrm K_1
\)-functor from algebraic
\(
	\mathrm K
\)-theory, an important homological invariant of the ring \(K\).

In the case
\(
	G = \mathrm{GL}_n
\),
\(
	n \geq 3
\) there exists a corresponding group ``at the level of
\(
	\mathrm K_2
\)'', namely, the linear Steinberg group
\(
	\stlin_G(K) = \stlin(n, K)
\). As a group it is given by explicit generators and relations, it is perfect, and there is a canonical homomorphism
\[
	\stmap \colon \stlin_G(K) \to G(K)
\]
with the image
\(
	\elem_G(K)
\). The kernel
\[
	\mathrm K_2^G(K) = \Ker(\stmap)
\]
is the unstable
\(
	\mathrm K_2
\)-functor. It turns out that
\(
	\mathrm K_2^G(K) \leq \stlin_G(K)
\) is central for
\(
	n \geq 4
\), moreover,
\(
	\stmap \colon \stlin_G(K) \to G(K)
\) is a crossed module in a unique way. If
\(
	n \geq 5
\), then
\(
	\stmap \colon \stlin_G(K) \to \elem_G(K)
\) is even the universal central extension, so
\(
	\mathrm K_2^G(K)
\) is the Schur multiplier of the elementary subgroup. In the case
\(
	n = 4
\) the Steinberg group is not necessarily centrally closed, but its Schur multiplier is known.

More generally, the Steinberg group
\(
	\stlin_G(K)
\) is easy to define for any Chevalley group scheme of rank at least \(3\), i.e. a split reductive group. It is perfect (essentially by definition) and a crossed module over
\(
	G(K)
\) by \cite{c-k2, d-k2, e-k2, lin-k2-vdk, unit-k2, thesis}. If the rank is sufficiently large, then
\(
	\stlin_G(K)
\) is centrally closed \cite{centr-closed}, and in any case its Schur multipliers is computed in \cite{schur-mult}.

For non-split isotropic reductive groups the only known results on centrality of
\(
	\mathrm K_2
\)-functor concern either classical groups \cite{lin-k2-tulenbaev, unit-k2, lin-k2, thesis} or groups over local rings \cite{boge, iso-congr-ker}.

In this paper we begin to study Steinberg groups corresponding to arbitrary locally isotropic reductive groups (more precisely, reductive groups of the local isotropic rank at least \(3\)). As in \cite{isotropic-elem} we first construct the values
\(
	\stlin_G
\) over some non-unital (but idempotent and even firm) \(K\)-algebras
\(
	K^{(s^\infty)} \in \Pro(\Set)
\) called co-localizations, where
\(
	s \in K
\) is sufficiently large in the sense of divisibility (so that
\(
	G_{K_s}
\) is globally isotropic). Note that these objects are not abstract rings, instead they are pro-sets with a ring structure. It is hard to work with groups given by generators and relations inside
\(
	\Pro(\Set)
\), so the Steinberg group
\(
	\stlin_G(K^{(s^\infty)})
\) is constructed as a group object in the exact-ind-completion
\(
	\Ex(\Ind(\Pro(\Set)))
\) of the category of pro-sets, see \S \ref{categorical-logic} for the definitions.

Moreover, we are interested in functorial properties of these constructions and we also heavily use the universal element method in proposition \ref{local-action}. Thus instead of \(\Set\) we use the category of presheaves
\[
	\mathbf P_K = \Cat(\Ring_K^\fp, \Set)
\]
on the category of affine \(K\)-schemes of finite presentation. Here the finitely presentation condition is needed to avoid set-theoretic difficulties, so all our categories are either ZFC-sets or proper classes consisting only of sets. For example, the base group object is
\(
	G \in \mathbf P_K
\) as a scheme (i.e. a representable presheaf), not
\(
	G(K)
\). The constructed co-local Steinberg groups are
\[
	\stlin_G(\R^{(s^\infty)})
	\in \Ex(\Ind(\Pro(\mathbf P_K))),
\]
where
\(
	\R = \mathbb A^1_K
\) is the scheme representing the forgetful functor
\[
	\Ring_K^\fp \to \Set,\, R \mapsto R.
\]

It turns out that the co-local Steinberg groups are crossed modules over
\(
	G(\mathcal R) = G
\), we prove this in the same way as in \cite{unit-k2, lin-k2, thesis} using ``root elimination'' (i.e. that the Steinberg group is independent on the choice of root subgroups) and the cosheaf property of
\(
	\stlin_G(\R^{(s^\infty)})
\) following \cite{cosheaves}. The whole Steinberg group functor
\(
	\stlin_G(\R)
\) is then constructed by gluing
\(
	\stlin_G(\R^{(s^\infty)})
\) via the cosheaf property, it is a crossed module essentially by definition. Also, it is independent on the elements \(s\), i.e. the Zariski covering, and the root subgroups of
\(
	G_{K_s}
\).

Unfortunately,
\[
	\stlin_G(\R) \in \Ex(\Ind(\Pro(\mathbf P_K)))
\]
cannot be considered as an ordinary group sheaf without further (still unproved) results. Namely, if \(G\) is globally isotropic in a certain strong sense (e.g.
\(
	G = \glin(n, A)
\) for
\(
	n \geq 4
\) and an Azumaya algebra \(A\) over \(K\)), then
\[
	\stlin_G(\R) \in \Ex(\Ind(\mathbf P_K))
\]
by an explicit construction with generators and relations. In this case we may apply the ``evaluation'' functor
\[
	\Ex(\Ind(\mathbf P_K)) \to \mathbf P_K
\]
to
\(
	\stlin_G(\R)
\) an obtain the required group functor
\(
	\stlin_G
\). We plan to prove that always
\[
	\stlin_G(\R) \in \Ex(\Ind(\mathbf P_K))
\]
(i.e. if \(G\) is only locally isotropic) in the future. We also plan to prove that the functor
\(
	\stlin_G
\) is ``scheme presented'' by morphisms from affine spaces in a suitable sense.

In this paper we usually assume that the local isotropic rank of \(G\) is at least \(3\). Is is well-known \cite{wendt} that in the rank \(2\) case the
\(
	\mathrm K_2
\)-functor may be not central even for Chevalley groups. On the other hand, if \(K\) is local, then the
\(
	\mathrm K_2
\)-functor is known to be central at least if \(G\) is simply connected, see \cite[theorem 1.3]{iso-congr-ker}. Thus we have the following open problem: is it true that the Steinberg group functor may be defined in the rank \(2\) case if \(K\) is semi-local (or, more generally, an LG-ring following \cite{gille}) and that it is a crossed module over \(G\) in a natural way? It is not clear even whether the crossed module structure is unique since already in the split case the Steinberg group may be not perfect.

The paper is organized as follows. In \S 2 we recall the definition of locally isotropic reductive groups following \cite{isotropic-elem}. The only difference is that now we require that the root subgroups determine a root grading of \(G\), i.e. there are corresponding Weyl elements. The next section 3 contains the necessary notions from category theory, namely, ind-, pro-, and exact completions, as well as the geometric logic inside infinitary positive categories. This machinery is useful to construct group objects by generators and relations and to prove various claims about them in the same way as in the ordinary group theory, see e.g. lemma \ref{gen-rel}.

We give the constructions of
\(
	\mathbf P_K
\), its completion
\[
	\mathbf U_K = \Ex(\Ind(\Pro(\mathbf P_K))),
\]
the ring object \(\R\), its co-localizations
\(
	\R^{(s^\infty)}
\) and formal localizations
\(
	\R_s^{\Ind}
\), and the Steinberg groups
\(
	\stlin_{G, T, \Phi}(\R^{(s^\infty)})
\),
\(
	\stlin_{G, T, \Phi}(\R_s^{\Ind})
\) in \S 4, where
\(
	(T, \Phi)
\) is an isotropic pinning of
\(
	G_{K_s}
\). The whole section \S 5 contains the proof of the root elimination results, namely, lemma \ref{elim-sur} and proposition \ref{elim-bij}. In \S 6 we prove that
\(
	\stlin_{G, T, \Phi}(\R^{(s^\infty)})
\) is perfect and a crossed module over
\(
	G(\R_s^{\Ind})
\) in the unique way (lemma \ref{st-perf} and proposition \ref{local-action}), also these groups satisfy the cosheaf property (lemma \ref{st-cosheaf}).

Finally, in section 7 we show that
\(
	\stlin_{G, T, \Phi}(\R^{(s^\infty)})
\) is independent of the choice of the isotropic pinning
\(
	(T, \Phi)
\) (lemma \ref{st-unique}) and these group objects may be glued together to the group object
\(
	\stlin_G(\R)
\) by the cosheaf property, it is both perfect and a crossed module over
\(
	G(\R)
\) by theorem \ref{global-st}. The ``ordinary'' centrality of the
\(
	\mathrm K_2
\)-functor (i.e. for the set-theoretic Steinberg group) is deduced in theorem \ref{abstr-xmod} under the additional assumption
\(
	\stlin_G(\R) \in \Ex(\Ind(\mathbf P_K))
\), this assumption obviously holds if \(G\) is already has a global isotropic pinning of rank at least \(3\).

\section{Isotropic reductive groups}

\subsection{Reductive group schemes}

In this paper all rings and algebras are commutative and associative, but not necessarily unital. For every unital ring \(K\) we identify \(K\)-schemes with the corresponding functors on the category of \(K\)-algebras and finitely presented affine \(K\)-schemes with the corresponding functors on the category of finitely presented \(K\)-algebras. All root systems in this paper are finite and crystallographic, though possibly non-reduced, i.e. we allow components of type
\(
	\mathsf{BC}_\ell
\).

Take a \textit{reductive group scheme} \(G\) over \(K\) in the sense of \cite[XIX, definition 2.7]{sga3}, i.e. \(G\) is a smooth affine group scheme and its geometric fibers
\(
	G_{\overline{\kappa(\mathfrak p)}}
\) are (connected) reductive groups over algebraically closures of the residue fields in the usual sense for all
\(
	\mathfrak p \in \Spec(K)
\). A \textit{pinning} \cite[XXIII, definition 1.1]{sga3} of \(G\) consists of
\begin{itemize}

	\item
	a maximal torus
	\(
		T \leq G
	\) with a chosen isomorphism
	\(
		T \cong \mathbb G_{\mathrm m}^\ell
	\);

	\item
	a root datum
	\(
		(
			\mathbb Z^\ell,
			\Phi,
			(\mathbb Z^\ell)^\vee,
			\Phi^{\vee}
		)
	\)
	such that \(\Phi\) and \(\Phi^*\) are the sets of roots and coroots of \(G\) with respect to \(T\) (or
	\(
		K = 0
	\)), in particular, roots and coroots are constant functions on
	\(
		\Spec(K)
	\);

	\item
	a basis
	\(
		\Delta \subseteq \Phi
	\);

	\item
	trivializing sections
	\(
		x_\alpha \in \mathfrak g_\alpha
	\) in the root spaces
	\(
		\mathfrak g_\alpha
		\leq \mathfrak g
		= \mathrm{Lie}(G)
	\) for
	\(
		\alpha \in \Delta
	\), so all root spaces are free \(K\)-modules of rank \(1\).

\end{itemize}
The group scheme \(G\) always has a pinning locally in the \'etale topology and any two pinned reductive group schemes with the same root datum are isomorphic. The \textit{type} of \(G\) at a point
\(
	\mathfrak p \in \Spec(K)
\) is the isomorphism class of the root datum of
\(
	G_{\overline{\kappa(\mathfrak p)}}
\), it is locally constant on
\(
	\Spec(K)
\).

If \(K\) is local (or even an LG-ring with connected spectrum by results from \cite{gille}), then \(G\) always has a minimal parabolic subgroup
\(
	P \leq G
\) and any two such subgroups are conjugate by the group
\(
	G(K)
\) of \(K\)-points \cite[XXVI, corollary 5.7(ii)]{sga3}. Moreover, \(G\) has a maximal split torus
\(
	T_0 \leq G
\), any two such tori are also conjugate by
\(
	G(K)
\), any such torus \(T_0\) is contained in the center of a Levi subgroup \(L\) of a minimal parabolic subgroup
\(
	P \leq G
\), and \(L\) is the scheme centralizer of \(T_0\) by \cite[XXVI, corollary 6.11 and proposition 6.16]{sga3}. In the LG case these results hold for maximal split tori in the scheme derived subgroup
\(
	[G, G]
\) instead of the whole \(G\), see \cite[lemma 4.6 and theorem 4.7]{gille}. If
\(
	P^{+}, P^{-} \leq G
\) are opposite parabolic subgroups, then the \textit{Gauss decomposition}
\[
	G(K) = U^{+}(K)\, U^{-}(K)\, U^{+}(K)\, L(K)
\]
holds by \cite[XXVI, corollary 5.2]{sga3} (or \cite[theorem 3.1(c)]{gille}), where
\(
	U^{\pm} \leq P^{\pm}
\) are the unipotent radicals and
\(
	L = P^{+} \cap P^{-}
\) is the common Levi group.

If \(K\) is local, then the set of non-zero weights of
\(
	\mathfrak g
\) with respect to a maximal split torus
\(
	T_0 \leq G
\) is indeed a root system \(\Phi\) (its elements are called \textit{relative roots}). For any
\(
	\alpha \in \Phi
\) there exists a unique smooth subgroup
\(
	U_\alpha \leq G
\) with connected fibers such that
\(
	U_\alpha \cap L = 1
\) and
\(
	\mathrm{Lie}(U_\alpha)
	= \bigoplus_{
		\substack{
			i \alpha \in \Phi
		\\
			i \in \mathbb N_+
		}
	}
		\mathfrak g_{i \alpha}
\) \cite[XXVI, 7.4.2]{sga3}, where
\(
	L = \Cent_G(T_0)
\). Clearly,
\(
	P^{\pm}
	= \prod_{\alpha \in \pm \Pi}
		U_\alpha
	\rtimes L
\) are opposite minimal parabolic subgroups with the common Levi subgroup \(L\) for some subset
\(
	\Pi \subseteq \Phi
\) of positive roots. Moreover, for any root
\(
	\alpha \in \Phi
\) there exists
\(
	w \in U_\alpha(K)\, U_{-\alpha}(K)\, U_\alpha(K)
\) such that
\(
	\up w L = L
\) and
\(
	\up w{U_\beta} = U_{s_\alpha(\beta)}
\), where
\(
	s_\alpha(\beta)
	= \beta
	- 2 \frac{\alpha \cdot \beta}{\alpha \cdot \alpha}
		\alpha
\) is the reflection of \(\beta\) in the hyperplane
\(
	\alpha^{\perp}
\). This follows from \cite[XXVI, 7.4.2]{sga3} and the Gauss decomposition for the reductive group scheme
\(
	L_\alpha
	= \langle U_{-\alpha}, L, U_\alpha \rangle
\).

\subsection{Isotropic pinnings}

Motivated by the above, we say that a \textit{isotropic pre-pinning} of \(G\) consists of
\begin{itemize}

	\item
	a split torus
	\(
		T \leq G
	\) with a chosen isomorphism
	\(
		T \cong \mathbb G_{\mathrm m}^\ell
	\);

	\item
	a root system
	\(
		\Phi \subseteq \mathbb Z^\ell
	\) (with respect to some dot product on
	\(
		\mathbb R^\ell
	\)) with a chosen base
	\(
		\Delta \subseteq \Phi
	\) such that \(\Phi\) is the set of non-zero weights of
	\(
		\mathfrak g
	\) with respect to \(T_0\) (or
	\(
		K = 0
	\)) and all root subspaces are free \(K\)-modules;

\end{itemize}

We say that an isotropic pre-pinning
\(
	(T, \Phi)
\) \textit{is contained} in an isotropic pre-pinning
\(
	(T', \Psi')
\) if
\(
	T \leq T'
\), the inclusion is given by a constant surjective homomorphism
\(
	u \colon \mathbb Z^{\ell'} \to \mathbb Z^\ell
\) of the corresponding abelian groups, and
\(
	u(\Phi' \cup \{0\}) = \Phi \cup \{0\}
\) (the last condition is vacuous if \(K\) is non-zero). Clearly, every isotropic pre-pinning is contained in a pinning \'etale locally, so this definition of isotropic pre-pinnings coincides with the definition of isotropic pinnings from \cite[\S 2]{isotropic-elem}. Actually, the condition that the root subspaces
\(
	\mathfrak g_\alpha
\) are free is not necessary, but this simplifies some arguments. The group
\(
	G(K)
\) of \(K\)-points acts on the set of isotropic pinnings by conjugation
\(
	(T, \Phi) \mapsto (\up g T, \up g \Phi)
\), where
\(
	\up g \Phi = \{\up g \alpha \mid \alpha \in \Phi\}
\) is the image of \(\Phi\) under the canonical isomorphism
\(
	T \cong \up g T
\) given by the conjugation.

The \textit{rank} of an isotropic pre-pinning
\(
	(T, \Phi)
\) is \(0\) if \(T\) commutes with some simple factor of
\(
	(G / \Cent(G))_{\overline{\kappa(\mathfrak p)}}
\) for
\(
	\mathfrak p \in \Spec(K)
\), otherwise the rank of
\(
	(T, \Phi)
\) is the smallest rank of components of \(\Phi\). Clearly, if
\(
	(T, \Phi) \subseteq (T', \Phi')
\), then the rank of
\(
	(T, \Phi)
\) is at most the rank of
\(
	(T', \Phi')
\).

For every pre-pinning
\(
	(T, \Phi)
\) there exist unique \textit{root subgroups}
\(
	U_\alpha \leq G
\) for
\(
	\alpha \in \Phi
\) \cite[\S 2]{isotropic-elem} with the following properties.
\begin{itemize}

	\item
	\(
		U_\alpha \leq G
	\) are preserved under base change, they are smooth closed subschemes.

	\item
	If
	\(
		(T, \Phi)
	\) is a pinning (formally, a part of a pinning), then
	\(
		U_\alpha
	\) are ordinary root subgroups.

	\item
	If
	\(
		(T, \Phi) \subseteq (T', \Phi')
	\) and
	\(
		\alpha \in \Phi
	\), then
	\[
		U_\alpha
		= \prod_{
			\substack{
				\beta \in \Phi'
			\\
				\Image(\beta) \in \mathbb N_+ \alpha
			}
		}
			U_\beta.
	\]

	\item
	The scheme centralizer
	\(
		L = \Cent_G(T)
	\) normalizes all
	\(
		U_\alpha
	\).

	\item
	If
	\(
		\alpha, \beta \in \Phi
	\) are not anti-parallel, then
	\[
		[U_\alpha, U_\beta]
		\subseteq \prod_{
			\substack{
				i \alpha + j \beta \in \Phi
			\\
				i, j \in \mathbb N_+
			}
		}
			U_{i \alpha + j \beta}.
	\]

	\item
	If
	\(
		\alpha, 2 \alpha \in \Phi
	\) (i.e. \(\alpha\) is ultrashort in a component of \(\Phi\) of type
	\(
		\mathsf{BC}_\ell
	\)), then
	\(
		U_{2 \alpha} \leq U_\alpha
	\).

	\item
	If
	\(
		\Pi \subseteq \Phi
	\) is a subset of positive roots, then the product map
	\[
		\prod_{\alpha \in \Pi \setminus 2 \Pi}
			U_{-\alpha}
		\times L
		\times \prod_{\alpha \in \Pi \setminus 2 \Pi}
			U_\alpha
		\to G
	\]
	is an open embedding. The group schemes
	\(
		U^{\pm}
		= \prod_{\alpha \in \Pi \setminus 2 \Pi}
			U_{\pm \alpha}
	\) are the unipotent radicals of the opposite parabolic subgroups
	\(
		P^{\pm} = U^{\pm} \rtimes L
	\) with the common Levi subgroup \(L\).

\end{itemize}

A subset
\(
	\Sigma \subseteq \Phi
\) is
\begin{itemize}

	\item
	\textit{closed} if
	\(
		(\Sigma + \Sigma) \cap \Phi \subseteq \Sigma
	\);

	\item
	\textit{unipotent} if it is closed and contained in an open half-space;

	\item
	\textit{saturated root subsystem} if it is an intersection of \(\Phi\) with a subspace (e.g. its span), so it is closed and a root system.

\end{itemize}

If
\(
	\Sigma \subseteq \Phi
\) is a unipotent subset, then the multiplication morphism
\[
	\prod_{\alpha \in \Sigma \setminus 2 \Sigma}
		U_\alpha
	\to G
\]
is a closed embedding for any linear order on
\(
	\Sigma \setminus 2 \Sigma
\). We denote its image by
\(
	U_\Sigma \leq G
\), it is a group subscheme.

We say that a pre-pinning
\(
	(T, \Phi)
\) is an \textit{isotropic pinning} if there are
\(
	w_\alpha
	\in U_\alpha(K)\, U_{-\alpha}(K)\, U_\alpha(K)
\) for all
\(
	\alpha \in \Phi
\) such that
\(
	\up {w_\alpha} T = T
\),
\(
	\up {w_\alpha} L = L
\), and
\(
	\up{w_\alpha}{U_\beta} = U_{s_\alpha(\beta)}
\). If \(K\) is local, then every isotropic pre-pinning is contained in a maximal one, all maximal isotropic pre-pinnings are conjugate by
\(
	G(K)
\), and all maximal isotropic pre-pinnings are actually isotropic pinnings.

If \(K\) is semi-local with connected spectrum, then the \textit{isotropic rank} of a reductive group scheme \(G\) over \(K\) is the common rank of its maximal isotropic pinnings. In general case the \textit{local isotropic rank} of \(G\) is the minimum of the isotropic ranks of
\(
	G_{\mathfrak m}
\) for all maximal ideals
\(
	\mathfrak m \leqt K
\). The local isotropic rank is non-decreasing under base changes.

\subsection{Root morphisms}

Recall \cite[\S 4]{twisted-forms} that a \textit{\(2\)-step nilpotent \(K\)-module} is a pair
\(
	(M, M_0)
\), where \(M\) is a group with the group operation \(\dotplus\) and nilpotent filtration
\(
	\dot 0 \leq M_0 \leq M
\) (i.e.
\(
	[M, M]^\cdot \leq M_0
\) and
\(
	[M, M_0]^\cdot = \dot 0
\)), \(M_0\) has a structure of a left \(K\)-module, and there is an action
\(
	({-}) \cdot ({=})
\) of the multiplicative monoid
\(
	K^\bullet
\) on \(M\) by endomorphisms from the right such that
\begin{itemize}

	\item
	\(
		m \cdot (k + k')
		= m \cdot k
			\dotplus k k' \tau(m)
			\dotplus m \cdot k'
	\) for some (unique)
	\(
		\tau(m) \in M_0
	\);

	\item
	\(
		[m \cdot k, m' \cdot k'] = k k' [m, m']
	\);

	\item
	\(
		m \cdot k = k^2 m
	\) if
	\(
		m \in M_0
	\).

\end{itemize}
The map \(\tau\) satisfies the following identities:
\(
	\tau(m) = 2 m
\) for
\(
	m \in M_0
\),
\(
	\tau(m \cdot k) = k^2 \tau(m)
\), and
\(
	\tau(m \dotplus m') = \tau(m) + [m, m'] + \tau(m')
\). Any \(K\)-module \(N\) may be considered as a \(2\)-step nilpotent module
\(
	(N, 0)
\) or
\(
	(N, N)
\).

For any \(K\)-modules \(M_0\) and \(M_1\) and any bilinear form
\(
	c \colon M_1 \times M_1 \to M_0
\) we may construct the \textit{split} \(2\)-step nilpotent \(K\)-module
\(
	(M, M_0)
\), where
\(
	M = M_0 \dotoplus M_1
\) (i.e.
\(
	M_0 \times M_1
\) as a set) with
\begin{align*}
	(m_0 \dotoplus m_1)
		\dotplus (m'_0 \dotoplus m'_1)
	&= (m_0 + c(m_1, m'_1) + m'_0, m_1 + m'_1),
\\
	(m_0 \dotoplus m_1) \cdot k
	&= k^2 m_0 \dotoplus k m_1,
\\
	\tau(m_0 \dotoplus m_1)
	&= (2 m_0 - c(m_1, m_1)) \dotoplus 0.
\end{align*}
Clearly, every \(2\)-step nilpotent \(K\)-module
\(
	(M, M_0)
\) with free \(K\)-module
\(
	M / M_0
\) splits, see also \cite[\S 4]{twisted-forms}.

Any split \(2\)-step nilpotent \(K\)-module
\(
	(M, M_0)
\) determines a whole functor
\(
	\mathbb W(M)
\) on the category of unital \(K\)-algebras
\(
	\Ring_K
\). Namely, if
\(
	M = M_0 \dotoplus M_1
\) for some \(2\)-co cycle
\(
	c \colon M_1 \times M_1 \to M_0
\), then
\[
	\mathbb W(M)(K')
	= (M_0 \otimes_K K')
		\dotoplus (M_1 \otimes_K K').
\]
This construction is independent on the choice of the splitting, i.e. it may be given in terms of abstract \(2\)-step modules, see \cite[\S 4]{twisted-forms}. Also, this is a generalization of the functor
\(
	\mathbb W(N) = N \otimes_K ({-})
\) for a \(K\)-module \(N\).

Now let \(G\) be a reductive group scheme over \(K\) with an isotropic pinning
\(
	(T, \Phi)
\). For all
\(
	\alpha \in \Phi
\) there are natural group isomorphisms
\(
	t_\alpha
	\colon \mathbb W(\mathfrak g_\alpha)
	\to U_\alpha
\) unless \(\alpha\) is ultrashort in a component of type
\(
	\mathsf{BC}_\ell
\), see \cite[\S 2]{isotropic-elem}. In the exceptional case there is a (necessarily split) \(2\)-step nilpotent \(K\)-module
\(
	(
		\mathfrak g_\alpha
			\dotoplus \mathfrak g_{2 \alpha},
		\mathfrak g_{2 \alpha}
	)
\) together with an isomorphism
\(
	\mathbb W(
		\mathfrak g_\alpha
		\dotoplus \mathfrak g_{2 \alpha}
	)
	\to U_\alpha
\). We denote the domain of
\(
	t_\alpha
\) by
\(
	P_\alpha
\) in all cases, this is a unipotent group scheme (i.e. isomorphic to an affine space as a scheme). The homomorphisms
\(
	t_\alpha \colon P_\alpha \to G
\) have the following properties.
\begin{itemize}

	\item
	\(
		t_{2 \alpha} = t_\alpha|_{P_{2 \alpha}}
	\) if \(\alpha\) is ultrashort in a component of type
	\(
		\mathsf{BC}_\ell
	\).

	\item
	\(
		[t_\alpha(x), t_\beta(y)]
		= \prod_{
			\substack{
				i \alpha + j \beta \in \Phi
			\\
				i, j \in \mathbb N_+
			}
		}
			t_{i \alpha + j \beta}\bigl(
				f_{\alpha \beta}
					^{i \alpha + j \beta}
					(x, y)
			\bigr)
	\) if \(\alpha\) and \(\beta\) are not anti-parallel, where
	\(
		f_{\alpha \beta}^\gamma
	\) are some scheme morphisms. Here we fix some order in the product and assume that
	\(
		f_{\alpha \beta}^{2 i \alpha + 2 j \beta}
		= \dot 0
	\) if both
	\(
		i \alpha + j \beta
	\) and
	\(
		2 i \alpha + 2 j \beta
	\) are roots (so
	\(
		f_{\alpha \beta}^\gamma
	\) are unique).

	\item
	There are unique actions of \(L\) on
	\(
		P_\alpha
	\) such that
	\(
		t_\alpha
	\) are \(L\)-equivariant. Moreover,
	\(
		f_{\alpha \beta}^\gamma
	\) are \(L\)-equivariant.

	\item
	Let
	\(
		({-}) \cdot_\alpha ({=})
	\) be the canonical right action of the multiplicative monoid
	\(
		K^\bullet
	\) (or the whole multiplicative monoid scheme) on
	\(
		P_\alpha
	\), i.e.
	\(
		m \cdot k = k m
	\) if \(\alpha\) is not ultrashort in a component of type
	\(
		\mathsf{BC}_\ell
	\) (so
	\(
		m \cdot_\alpha k = m \cdot_{2 \alpha} k^2
	\) if \(\alpha\) is ultrashort and
	\(
		m \in P_{2 \alpha}(K)
		= \mathfrak g_{2 \alpha}
	\)). Then
	\[
		f_{\alpha \beta}^{i \alpha + j \beta}(
			x \cdot_\alpha k,
			x' \cdot_\beta k'
		)
		= f_{\alpha \beta}^{i \alpha + j \beta}(x, x')
			\cdot_{i \alpha + j \beta} k^i {k'}^j.
	\]

\end{itemize}

\subsection{Weyl triples}

A \textit{Weyl triple} is a tuple
\(
	(a, b, c)
	\in P_\alpha(K)
		\times P_{-\alpha}(K)
		\times P_\alpha(K)
\) such that
\(
	w = t_\alpha(a)\, t_{-\alpha}(b)\, t_\alpha(c)
\) is a Weyl element, i.e.
\(
	\up w {U_\beta} = U_{s_\alpha(\beta)}
\) and
\(
	\up w L = L
\). Such an element induces isomorphisms
\(
	\up w {(-)} \colon P_\beta \to P_{s_\alpha(\beta)}
\) of group schemes. An element
\(
	c \in P_\alpha(K)
\) is called \textit{invertible} if it is the last element of a Weyl triple. The set of invertible elements
\(
	P_\alpha(K)^* \subseteq P_\alpha(K)
\) is non-empty by the definition of isotropic pinnings.

\begin{lemma} \label{root-units}
	Let
	\(
		(T, \Phi)
	\) be an isotropic pinning of a reductive group scheme \(G\) and
	\(
		\alpha \in \Phi
	\) be a root.
	\begin{enumerate}

		\item
		If
		\(
			(a, b, c)
		\) is a Weyl triple with the Weyl element \(w\), then
		\(
			(\dotminus c, \dotminus b, \dotminus a)
		\),
		\(
			(\up w c, a, b)
		\), and
		\(
			(b, c, \up {w^{-1}} a)
		\) are also Weyl triples. In particular,
		\(
			a \in P_\alpha(K)
		\) is invertible if and only if it is a component of a Weyl triple in any position. Also,
		\(
			\dotminus P_\alpha(K)^* = P_\alpha(K)^*
		\).

		\item
		The set of Weyl triples is closed under conjugations by Weyl elements. If
		\(
			w \in U_\alpha U_{-\alpha} U_\alpha
		\) is a Weyl element, then
		\(
			\up w {P_\beta(K)^*}
			= P_{s_\alpha(\beta)}(K)^*
		\).

		\item
		If
		\(
			\alpha \in \Phi
		\) is ultrashort in a component of type
		\(
			\mathsf{BC}_\ell
		\), then
		\(
			P_{2 \alpha}(K)^* \subseteq P_\alpha(K)^*
		\).

		\item
		If the rank of \(\Phi\) is at least \(2\), then any Weyl triple
		\(
			(a, b, c)
		\) is uniquely determined by any its component.

		\item
		If
		\(
			e \in P_\alpha(K)^*
		\), then
		\(
			f_{\alpha, s_\alpha(\beta)}
				^\beta
				(e, {-})
			\colon P_{s_\alpha(\beta)}
			\to P_\beta
		\) is an isomorphism of group schemes for every neighbor \(\beta\) of \(\alpha\), it induces a bijection between
		\(
			P_{s_\alpha(\beta)}(K)^*
		\) and
		\(
			P_\beta(K)^*
		\).

	\end{enumerate}
\end{lemma}
\begin{proof}
	The first three claims are easy, see e.g. \cite[proposition 2.2.6]{wiedemann}. The fourth one follows from \cite[proposition 2.2.22]{wiedemann} (where
	\(
		U_\alpha^\sharp
		= t_\alpha(P_\alpha(K)^*)
		\subseteq G(K)
	\)).

	Let us prove the last claim. Choose a (necessarily unique) Weyl element
	\[
		w
		= t_{-\alpha}(a)\, t_\alpha(e)\, t_{-\alpha}(c)
		\in P_{-\alpha}(K)\,
			P_\alpha(K)\,
			P_{-\alpha}(K)
	\]
	and evaluate
	\(
		\up w {t_{s_\alpha(\beta)}(p)}
	\) using the commutator formula. We obtain that
	\(
		f_{\alpha, s_\alpha(\beta)}^\beta(y, {-})
	\) is a scheme isomorphism identifying the sets of invertible elements. Evaluating both sides of
	\[
		[
			t_\alpha(e),
			t_{s_\alpha(\beta)}(p \dotplus q)
		]
		= [t_\alpha(e), t_{s_\alpha(\beta)}(p)]\,
		\up{t_{s_\alpha(\beta)}(p)}{
			[t_\alpha(e), t_{s_\alpha(\beta)}(q)]
		}
	\]
	using the commutator formula, one may also check that the morphism from the statement is a homomorphism.
\end{proof}

\section{Categorical logic} \label{categorical-logic}

\subsection{Ind- and pro-completions}

Recall that for any category
\(
	\mathbf C
\) its \textit{pro-completion}
\(
	\mathbf C \to \Pro(\mathbf C)
\) is the universal functor to a category with all (small) projective limits. In other words,
\begin{itemize}

	\item
	for any small filtered category
	\(
		\mathbf I
	\) and a contra-variant functor
	\(
		X \colon \mathbf I \to \Pro(\mathbf C)
	\) there exists
	\(
		\varprojlim_{i \in \mathbf I} X_i
		\in \Pro(\mathbf C)
	\);

	\item
	for any functor
	\(
		F \colon \mathbf C \to \mathbf D
	\) to a category with small projective limits there exist a functor
	\(
		G \colon \Pro(\mathbf C) \to \mathbf D
	\) preserving small projective limits and a natural isomorphism
	\(
		\sigma \colon F \to G|_{\mathbf C}
	\);

	\item
	if
	\(
		(G', \sigma')
	\) is another such pair, then there exists unique natural isomorphism
	\(
		\tau \colon G \to G'
	\) such that
	\(
		\tau|_{\mathbf C} \circ \sigma = \sigma'
	\).

\end{itemize}
The pro-completion always exists and is unique up to an equivalence. It may be constructed as the category of all \textit{inverse systems}, i.e. functors
\(
	X \colon \mathbf I_X^\op \to \mathbf C
\) for small filtered categories
\(
	\mathbf I_X
\), with
\[
	\Pro(\mathbf C)(X, Y)
	= \varprojlim_{j \in \mathbf I_Y}
		\varinjlim_{i \in \mathbf I_X}
			\mathbf C(X_i, Y_j).
\]

The \textit{ind-completion}
\(
	\Ind(\mathbf C) = \Pro(\mathbf C^\op)^\op
\) may be defined by the duality. Its objects are \textit{direct systems}, i.e. functors
\(
	X \colon \mathbf I_X \to \mathbf C
\) for small filtered categories
\(
	\mathbf I_X
\), with
\[
	\Ind(\mathbf C)(X, Y)
	= \varprojlim_{i \in \mathbf I_X}
		\varinjlim_{j \in \mathbf I_Y}
			\mathbf C(X_i, Y_j).
\]
The canonical functors
\(
	\mathbf C \to \Pro(\mathbf C)
\) and
\(
	\mathbf C \to \Ind(\mathbf C)
\) are fully faithful. If
\(
	\mathbf C \to \mathbf D
\) is a fully faithful functor, then the induced functors
\(
	\Pro(\mathbf C) \to \Pro(\mathbf D)
\) and
\(
	\Ind(\mathbf C) \to \Ind(\mathbf D)
\) are also fully faithful.

\subsection{Infinitary coherence}

A category
\(
	\mathbf C
\) is called \textit{infinitary positive} if the following properties hold.
\begin{itemize}

	\item
	\(
		\mathbf C
	\) has all finite limits.

	\item
	For every morphism
	\(
		f \in \mathbf C(X, Y)
	\) the kernel pair
	\(
		X \times_Y X \rightrightarrows X
	\) has a coequalizer.

	\item
	The class of regular epimorphisms (i.e. coequalizers of some pairs of morphisms) is stable under pullbacks.

	\item
	\(
		\mathbf C
	\) has all small coproducts
	\(
		\bigsqcup_{i \in I}
			X_i
	\).

	\item
	The embedding morphisms
	\(
		X \to X \sqcup Y
	\) and
	\(
		Y \to X \sqcup Y
	\) are monomorphisms.

	\item
	The intersection of \(X\) and \(Y\) as subobjects of
	\(
		X \sqcup Y
	\) is empty (i.e. the initial object \(\varnothing\)).

	\item
	The natural morphism
	\(
		\bigsqcup_{i \in I}
			(X_i \times Y)
		\to (
			\bigsqcup_{i \in I}
				X_i
		) \times Y
	\) is an isomorphism for all small families
	\(
		(X_i)_{i \in I}
	\) and all \(Y\).

\end{itemize}

If instead of all small families in coproducts we consider only finite ones, then
\(
	\mathbf C
\) is called \textit{positive}.

Every morphism
\(
	f \colon X \to Y
\) in a \textit{regular} category (i.e. a category satisfying the first three axioms of positive categories) has an \textit{image decomposition}
\(
	X \to \Image(f) \to Y
\) \cite[definition 5, axiom (2')]{carboni-vitale}, where the left morphism is a regular epimorphism and the right one is a monomorphism. Such image decomposition is unique up to unique isomorphism, it is functorial on \(f\) and stable under base change. A morphism is an isomorphism if and only if it is both a monomorphism and a regular epimorphism.

Moreover, in an positive category the classes of subobjects
\(
	\Sub(X)
\) are lattices with the bottom element
\(
	\varnothing \subseteq X
\), the top element
\(
	X \subseteq X
\), the binary unions
\(
	Y \cup Z = \Image(Y \sqcup Z \to X)
\), and the binary intersections
\(
	Y \cap Z = Y \times_X Z
\) of subobjects
\(
	Y, Z \subseteq X
\). For any
\(
	f \colon X \to Y
\) the base change maps
\(
	f^* \colon \Sub(Y) \to \Sub(X)
\) are lattice homomorphisms and the image maps
\(
	f_* \colon \Sub(X) \to \Sub(Y)
\) are \(\cup\)-semi-lattice homomorphisms (left adjoint to \(f^*\)). In the infinitary case the lattices
\(
	\Sub(X)
\) have all small unions and these unions are preserved under base changes and images.

A functor
\(
	F \colon \mathbf C \to \mathbf D
\) between regular categories is called \textit{exact} if it preserves all finite limits and regular epimorphisms (so it also preserves the images of morphisms). A functor between infinitary positive categories is called \textit{geometric} if it is exact and preserves all coproducts, its finitary counterpart for positive categories is called \textit{coherent}.

For example, \(\Set\) is infinitary positive, as well as all presheaf categories
\(
	\Cat(\mathbf D^\op, \Set)
\). For any positive category
\(
	\mathbf C
\) both
\(
	\Ind(\mathbf C)
\) and
\(
	\Pro(\mathbf C)
\) are positive by \cite[the list after example 1.12 and \S 1.8]{jacqmin-janelidze} and the embeddings
\(
	\mathbf C \to \Ind(\mathbf C)
\),
\(
	\mathbf C \to \Pro(\mathbf C)
\) are coherent. Namely, suppose that
\(
	f \colon X \to Y
\) is a \textit{level} morphism in
\(
	\Ind(\mathbf C)
\) or
\(
	\Pro(\mathbf C)
\), i.e. a natural transformation between direct or inverse systems considered as a morphism in the completion. Then its component-wise image decomposition
\(
	X \to \Image(f) \to Y
\) is the image decomposition in the completion. Finite limits and coproducts of level diagrams may be computed component-wise \cite[propositions 6.6.16 and 6.1.18]{kashiwara-schapira} and every finite acyclic diagram in the completion is isomorphic to a level one \cite[theorem 6.4.3]{kashiwara-schapira}, including the discrete diagrams for products and coproducts, as well as the equalizer diagrams
\(
	\bullet \rightrightarrows \bullet
\) and the fiber product diagrams
\(
	\bullet \to \bullet \leftarrow \bullet
\). Of course,
\(
	\Ind(\mathbf C)
\) is infinitary positive with
\[
	\coprod_{i \in I}
		X_i
	= \varinjlim_{J \subseteq I \text{ finite}}
		\coprod_{i \in J}
			X_i
\]
and
\[
	\bigcup_{i \in I}
		A_i
	= \varinjlim_{J \subseteq I \text{ finite}}
		\bigcup_{i \in J}
			A_i
\]
for subobjects
\(
	A_i \subseteq X
\). If
\(
	\mathbf C \to \mathbf D
\) is a coherent functor, then the induced functor
\(
	\Pro(\mathbf C) \to \Pro(\mathbf D)
\) between the pro-completions is also coherent, and
\(
	\Ind(\mathbf C) \to \Ind(\mathbf D)
\) is geometric.

\subsection{Geometric logic}

There is a convenient way to work within infinitary positive categories using geometric logic. Let \(L\) be a multi-sorted first order language, i.e. a class of \textit{sorts} \(D\), functional symbols
\(
	f \colon D_1 \times \ldots \times D_n \to D'
\) with given arities and sorts, and predicate symbols
\(
	p \colon D_1 \times \ldots \times D_n \to \mathbb B
\) with given arities and sorts including the equality symbols
\(
	{=_{D}} \colon D \times D \to \mathbb B
\) for each sort \(D\). A \textit{context} \(\Phi\) is a finite set of variables
\(
	x \colon D
\) with prescribed sorts, we also sometimes write
\(
	x^D
\) instead of
\(
	x \colon D
\). A \textit{term}
\(
	t \colon D
\) in a context \(\Phi\) is either a variable
\(
	x \colon D
\) from \(\Phi\) or an expression
\(
	f(t_1, \ldots, t_n)
\), where
\(
	f \colon D_1 \times \ldots \times D_n \to D
\) is a functional symbol and
\(
	t_i \colon D_i
\) are smaller terms in the same context. An \textit{elementary formula} in a context \(\Phi\) is an expression
\(
	p(t_1, \ldots, t_n)
\), where
\(
	p \colon D_1 \times \ldots \times D_n \to \mathbb B
\) is a predicate symbol (e.g. an equality) and
\(
	t_i \colon D_i
\) are terms in the same context. A \textit{formula} in a context \(\Phi\) is either an elementary formula, or the propositional constant \(\top\), or an expression
\(
	(\varphi \wedge \psi)
\) for simpler formulas \(\varphi\) and \(\psi\) in the same context, or an expression
\(
	\bigvee_{i \in I} \varphi_i
\) for simpler formulas
\(
	\varphi_i
\) indexed by a set \(I\) in the same context, or an expression
\(
	\exists x^D\, \varphi
\) for a formula \(\varphi\) in the context
\(
	\Phi \cup \{x^D\}
\) (assuming
\(
	x \notin \Phi
\)). In other words, formulas are represented by possibly infinite syntactic trees, we only require that there are no infinite branches. Finally, a \textit{sequent} is a formal expression
\(
	\Phi;\, \Gamma \vdash \varphi
\), where \(\Phi\) is a context, \(\Gamma\) is a finite set of formulae in the context \(\Phi\), and \(\varphi\) is a formula in the same context.

Suppose that
\(
	\mathbf C
\) is an infinitary positive category and \(L\) is a multi-sorted first order language. An \textit{interpretation} of \(L\) in
\(
	\mathbf C
\) maps every domain \(D\) to an object
\(
	[\![d]\!] \in \mathbf C
\), every functional symbol
\[
	f \colon D_1 \times \ldots \times D_n \to D'
\]
to a morphism
\[
	[\![f]\!]
	\colon [\![D_1]\!] \times \ldots \times [\![D_n]\!]
	\to [\![D']\!],
\]
and every predicate symbol
\[
	p \colon D_1 \times \ldots \times D_n \to \mathbb B
\]
to a subobject
\[
	[\![p]\!]
	\subseteq [\![D_1]\!] \times \ldots [\![D_n]\!]
\]
such that
\[
	[\![=_D]\!]
	\subseteq [\![D]\!] \times [\![D]\!]
\] is the diagonal for every sort \(D\). Such an interpretation naturally maps every context \(\Phi\) to
\(
	[\![\Phi]\!] = \prod_{x^D \in \Phi} [\![D]\!]
\), every first-order term
\(
	t \colon D
\) in a context \(\Phi\) to a morphism
\(
	[\![t]\!]_\Phi \colon [\![\Phi]\!] \to [\![D]\!]
\), and every formula \(\varphi\) in a context \(\Phi\) to a subobject
\(
	[\![\varphi]\!]_\Phi \subseteq [\![\Phi]\!]
\) in such a way that
\begin{align*}
	[\![t]\!]_{\Phi \sqcup \Phi}
	&= [\![t]\!]_\Phi
	\circ \bigl(
		\pi_\Phi
		\colon [\![\Phi]\!] \times [\![\Psi]\!]
		\to [\![\Psi]\!]
	\bigr);
\\
	[\![x]\!]_{x^D}
	&= \id_{[\![D]\!]};
\\
	[\![t|_{x := t'}]\!]_\Phi
	&= [\![t]\!]_{\Phi \sqcup \{x^D\}}
	\circ (\id_{[\![\Phi]\!]}, [\![t']\!]_\Phi);
\\
	[\![\varphi]\!]_{\Phi \sqcup \Phi}
	&= [\![\varphi]\!]_\Phi \times [\![\Psi]\!];
\\
	[\![\varphi|_{x := t}]\!]_\Phi
	&= \lim\bigl(
		[\![\varphi]\!]_{\Phi \sqcup \{x^D\}}
		\hookrightarrow [\![\Phi]\!] \times [\![D]\!]
		\xleftarrow{(\id, [\![t]\!]_\Phi)} [\![\Phi]\!]
	\bigr);
\\
	[\![\top]\!]_\Phi &= [\![\Phi]\!];
\\
	[\![\varphi \wedge \psi]\!]_\Phi
	&= [\![\varphi]\!]_\Phi \cap [\![\psi]\!]_\Phi;
\\
	\bigl[\!\!\!\:\bigl[
		\bigvee_{i \in I} \varphi_i
	\bigr]\!\!\!\:\bigr]_\Phi
	&= \bigcup_{i \in I} [\![\varphi_i]\!]_\Phi;
\\
	[\![\exists x^D\, \varphi]\!]_\Phi
	&= \Image\bigl(
		[\![\varphi]\!]_{\Phi \sqcup \{x^D\}}
		\hookrightarrow [\![\Phi]\!] \times [\![D]\!]
		\to [\![\Phi]\!]
	\bigr).
\end{align*}
A sequent
\(
	\Phi;\, \Gamma \vdash \phi
\) is \textit{true} in
\(
	\mathbf C
\) if
\(
	\bigcap_{\psi \in \Gamma} [\![\phi]\!]_\Phi
	\subseteq [\![\phi]\!]_\Phi
\). We usually omit the context in a the sequent if it is the minimal meaningful one. Also, we write
\(
	\vdash \varphi
\) just as \(\varphi\), this is useful when dealing with algebraic theories.

There is a natural way to derive true sequents from another true sequents. Namely, consider the following rules for a fixed language \(L\).
\begin{itemize}

	\item
	\(
		\displaystyle
		\frac{
			\Phi;\, \Gamma \vdash \varphi
		}{
			\Phi, x^D;\, \Gamma \vdash \varphi
		}
	\);\quad
	\(
		\displaystyle
		\frac{
			\Phi, x^D;\, \Gamma \vdash \varphi
		}{
			\Phi;\, \Gamma|_{x := t}
			\vdash \varphi|_{x := t}
		}
	\);

	\item
	\(
		\displaystyle
		\frac{}{
			\Phi;\, \varphi \vdash \varphi
		}
	\);\quad
	\(
		\displaystyle
		\frac{
			\Phi;\, \Gamma \vdash \varphi
		}{
			\Phi;\, \Gamma, \psi \vdash \varphi
		}
	\);\quad
	\(
		\displaystyle
		\frac{
			\Phi;\, \Gamma \vdash \varphi
		\quad
			\Phi;\, \varphi, \Delta \vdash \psi
		}{
			\Phi;\, \Gamma, \Delta \vdash \psi
		}
	\);

	\item
	\(
		\displaystyle
		\frac{}{
			\Phi, x^D;\, \Gamma \vdash x =_D x
		}
	\);\quad
	\(
		\displaystyle
		\frac{
			\Phi, x^D, y^D;\, \Gamma \vdash x =_D y
		\quad
			\Phi, x^D, y^D;\, \Gamma
			\vdash \varphi|_{y := x}
		}{
			\Phi, x^D, y^D;\, \Gamma \vdash \varphi
		}
	\);

	\item
	\(
		\displaystyle
		\frac{}{
			\Phi;\, \Gamma \vdash \top
		}
	\);\quad
	\(
		\displaystyle
		\frac{
			\Phi;\, \Gamma \vdash \varphi
		\quad
			\Phi;\, \Gamma \vdash \psi
		}{
			\Phi;\, \Gamma \vdash \varphi \wedge \psi
		}
	\);\quad
	\(
		\displaystyle
		\frac{
			\Phi;\, \Gamma \vdash \varphi \wedge \psi
		}{
			\Phi;\, \Gamma \vdash \varphi
		}
	\);\quad
	\(
		\displaystyle
		\frac{
			\Phi;\, \Gamma \vdash \varphi \wedge \psi
		}{
			\Phi;\, \Gamma \vdash \psi
		}
	\);

	\item
	\(
		\displaystyle
		\frac{
			\Phi; \Gamma \vdash \varphi_i
			\text{ for some } i \in I
		}{
			\Phi; \Gamma
			\vdash \bigvee_{i \in I} \varphi_i
		}
	\);\quad
	\(
		\displaystyle
		\frac{
			\Phi; \Gamma, \varphi_i \vdash \psi
			\text{ for all } i \in I
		}{
			\Phi; \Gamma, \bigvee_{i \in I} \varphi_i
			\vdash \psi
		}
	\);

	\item
	\(
		\displaystyle
		\frac{
			\Phi;\, \Gamma \vdash \varphi|_{x := t}
		}{
			\Phi;\, \Gamma
			\vdash \exists x^D\, \varphi
		}
	\);\quad
	\(
		\displaystyle
		\frac{
			\Phi, x^D;\, \Gamma, \varphi
			\vdash \psi
		}{
			\Phi;\, \Gamma, \exists x^D\, \varphi
			\vdash \psi
		}
	\).

\end{itemize}
A \textit{derivation graph} is a possibly infinite oriented graph, where edges are labeled by sequents. Vertices are either instances of the derivation rules, or ``assumption vertices'' with only one output edge, or the ``conclusion vertex'' with only one input edge. We require that there is a unique conclusion vertex and there are no backward infinite paths
\(
	\ldots \to \bullet \to \bullet
\), in particular, that the graph is acyclic. For every interpretation of \(L\) in an infinitary positive
\(
	\mathbf C
\) and every derivation graph if all assumptions hold in
\(
	\mathbf C
\), then the conclusion also holds in
\(
	\mathbf C
\) by \cite[proposition D1.3.2]{elephant}.

Similarly, if we consider only formulas with finitary disjunctions (i.e. working within coherent logic), then \(L\) may be interpreted in any positive category and the derivation preserves the class of true sequents. In particular, we may speak about group objects, ring objects, algebras over unital ring objects, and so on in any positive category (of course, algebraic theories may even be interpreted in any Cartesian category).

\subsection{Pretopoi}

For example, a binary relation
\(
	E \subseteq X \times X
\) in an infinitary positive category
\(
	\mathbf C
\) is called an \textit{equivalence relation} if it satisfies the axioms
\(
	E(x, x)
\);
\(
	E(x, y) \vdash E(y, x)
\); and
\(
	E(x, y), E(y, z)
	\vdash E(x, z)
\). Every binary relation \(R\) is contained in the smallest equivalence relation \cite[the proof of lemma A1.4.19]{elephant}
\[
	\bigl[\!\!\!\:\bigl[
		x = y
		\vee \bigvee_{n = 0}^\infty
			\exists x_1^X \ldots \exists x_n^X
				\bigl(
					R'(x, x_1)
					\wedge R'(x_1, x_2)
					\wedge \ldots
					\wedge R'(x_n, y)
				\bigr)
	\bigr]\!\!\!\:\bigr]_{x^X, y^X},
\]
where
\[
	R'
	= [\![
		R(x, y) \vee R(y, x)
	]\!]_{x^X, y^X}
\]
is the symmetric closure of \(R\).

Also, in an infinitary positive category
\(
	\mathbf C
\) there is a one-to-one correspondence between morphisms
\(
	f \colon X \to Y
\) and \textit{functional relations}
\(
	F \subseteq X \times Y
\), i.e. satisfying the axioms
\(
	\exists{y^Y} F(x, y)
\) and
\(
	F(x, y), F(x, y') \vdash y = y'
\). Namely, \(f\) corresponds to its \textit{graph}
\(
	[\![f(x) = y]\!]_{x^X, y^Y}
\).

Finally, we need the notion of pretopoi. Recall that in \(\Set\) (and in all presheaf categories) every equivalence relation
\(
	E \subseteq X \times X
\) is the kernel pair of a regular epimorphism
\(
	X \to X / E
\). A positive category
\(
	\mathbf C
\) is called a \textit{pretopos} if every equivalence relation is a kernel pair. In a pretopos every epimorphism is regular \cite[corollary A1.4.9]{elephant}, so a morphism \(f\) is invertible if and only if it is a monomorphism and an epimorphism. Also, any infinitary pretopos (i.e. a pretopos and an infinitary positive category) is co-complete \cite[lemma A1.4.19]{elephant},
\(
	\colim_{i \in \mathbf I}
		X_i
	= (
		\bigsqcup_{i \in \mathbf I}
		X_i
	) / E
\), where \(E\) is the equivalence relation generated by
\[
	\bigsqcup_{i, j \in \mathbf I}
		\bigl[\!\!\!\:\bigr[
			\bigvee_{
				i \xrightarrow u k \xleftarrow v k
			}
				X_u(x) = X_v(y)
		\bigr]\!\!\!\:\bigr]_{x^{X_i}, y^{X_j}}.
\]

If
\(
	\mathbf C
\) is an infinitary pretopos such as
\(
	\Set
\), then there is the evaluation functor
\[
	\ev
	\colon \Ind(\mathbf C)
	\to \mathbf C,\,
	X
	\mapsto \varinjlim_{i \in I}
		X_i
\]
from the universal property, it is left adjoint to the embedding. It is easy to check using geometric logic that this functor is geometric, i.e. direct limits in
\(
	\mathbf C
\) are coherent as functors
\(
	\Cat(\mathbf I, \mathbf C) \to \mathbf C
\) for all small filtered
\(
	\mathbf I
\).

\subsection{Exact completion}

For every regular category
\(
	\mathbf C
\) there is a universal exact functor
\(
	\mathbf C \to \Ex(\mathbf C)
\) to a regular category with all factor-objects by equivalence relation (a \textit{Barr exact} category), this functor is fully faithful. We need the following explicit construction \cite[definition 11]{carboni-vitale}.
\begin{itemize}

	\item
	Objects of
	\(
		\Ex(\mathbf C)
	\) are formal fractions
	\(
		X / {\equiv_X}
	\), where
	\(
		X \in \mathbf C
	\) and
	\(
		{\equiv_X} \subseteq X \times X
	\) is an equivalence relation.

	\item
	Morphisms
	\(
		X / {\equiv_X} \to Y / {\equiv_Y}
	\) are binary relations
	\(
		F \subseteq X \times Y
	\) such that
	\[
		x \equiv_X x', y \equiv_Y y', F(x, y)
		\vdash F(x', y'),
	\quad
		\exists y^Y F(x, y),
	\quad
		F(x, y), F(x, y')
		\vdash y \equiv_Y y'.
	\]

	\item
	The composition of
	\(
		F \colon X / {\equiv_X} \to Y / {\equiv_Y}
	\) and
	\(
		G \colon Y / {\equiv_Y} \to Z / {\equiv_Z}
	\) is given by
	\[
		[\![
			\exists y^Y (F(x, y) \wedge G(y, z))
		]\!]_{x^X, z^Z},
	\]
	the identity morphism
	\(
		X / {\equiv_X} \to X / {\equiv_X}
	\) is just
	\(
		\equiv_X
	\).

	\item
	The functor
	\(
		\mathbf C \to \Ex(\mathbf C)
	\) is given by
	\(
		X \mapsto X / {=_X}
	\) and
	\(
		(f \colon X \to Y)
		\mapsto [\![f(x) = y]\!]_{x^X, y^Y}
	\).

	\item
	A morphism
	\(
		F \colon X / {\equiv_X} \to Y / {\equiv_Y}
	\) is a monomorphism if and only if
	\(
		F(x, y), F(x', y) \vdash x = x'
	\). The subobject semi-lattice of
	\(
		X / {\equiv_X}
	\) is naturally isomorphic to the semi-lattice of
	\(
		A \subseteq X
	\) such that
	\(
		A(x), x \equiv_X x' \vdash A(x')
	\) (with the equivalence relation
	\(
		{\equiv_X} \cap (A \times A)
	\)).

	\item
	A morphism
	\(
		F \colon X / {\equiv_X} \to Y / {\equiv_Y}
	\) is a regular epimorphism if and only if
	\(
		\exists x^X\, F(x, y)
	\). The image of \(F\) is
	\(
		[\![\exists x^X\, F(x, y)]\!]_{y^Y}
	\) in the subobject semi-lattice of
	\(
		Y / {\equiv_Y}
	\).

	\item
	The product of
	\(
		X / {\equiv_X}
	\) and
	\(
		Y / {\equiv_Y}
	\) is
	\(
		(X \times Y) / ({\equiv_X} \times {\equiv_Y})
	\). The equalizer of
	\(
		F, G \colon X / {\equiv_X} \to Y / {\equiv_Y}
	\) is
	\(
		[\![
			\exists{y^Y} (
				F(x, y) \wedge G(x, y)
			)
		]\!]_{x^X}
	\) in the subobject lattice of
	\(
		X / {\equiv_X}
	\).

	\item
	A subobject
	\(
		E
		\subseteq (X \times X)
		/ ({\equiv_X} \times {\equiv_X})
	\) is an equivalence relation in
	\(
		\Ex(\mathbf C)
	\) if and only if \(E\) is an equivalence relation in
	\(
		\mathbf C
	\) and
	\(
		{\equiv_X} \subseteq E
	\). In this case
	\(
		X / E
	\) is the factor-object.

\end{itemize}
The finitary version of the following lemma is \cite[corollary A3.3.10]{elephant}
\begin{lemma} \label{ex-compl}
	For any infinitary positive category
	\(
		\mathbf C
	\) its exact completion
	\(
		\Ex(\mathbf C)
	\) is an infinitary pretopos and the fully faithful functor
	\(
		\mathbf C \to \Ex(\mathbf C)
	\) is geometric.
\end{lemma}
\begin{proof}
	It is easy to see that
	\[
		\bigsqcup_{i \in I}
			(X_i / {\equiv_i})
		\cong \bigl(\bigsqcup_{i \in I} X_i\bigr)
		/ \bigl(\bigsqcup_{i \in I} {\equiv_i}\bigr)
	\]
	The required sequents may be checked using the geometric logic.
\end{proof}

It follows that for any positive category
\(
	\mathbf C
\) the category
\(
	\Ex(\Ind(\mathbf C))
\) is an infinitary pretopos and the fully faithful functor
\(
	\mathbf C \to \Ex(\Ind(\mathbf C))
\) is coherent. Also, if
\(
	F \colon \mathbf C \to \mathbf D
\) is a geometric functor between infinitary positive categories, then
\(
	\Ex(\mathbf C) \to \Ex(\mathbf D)
\) is also geometric and it is fully faithful whenever \(F\) is fully faithful. If
\(
	\mathbf C
\) is already an infinitary pretopos, then
\(
	\mathbf C \to \Ex(\mathbf C)
\) is an equivalence of categories.

\section{Co-localization}

\subsection{Category \(\mathbf U_K\)}

Take a unital ring \(K\). The category
\(
	\mathbf P_K = \Cat(\Ring_K^\fp, \Set)
\) of presheaves on the category of affine schemes of finite presentation is an infinitary pretopos and contains all finitely presented affine schemes. It follows that the category
\(
	\Pro(\mathbf P_K)
\) is positive, the categories
\(
	\Ind(\Pro(\mathbf P_K))
\) and
\(
	\Ind(\mathbf P_K)
\) are infinitary positive, and the categories
\[
	\Ex(\Ind(\mathbf P_K)),
\quad
	\mathbf U_K
	= \Ex(\Ind(\Pro(\mathbf P_K)))
\]
are infinitary pretopoi by lemma \ref{ex-compl}. We mostly work within this category
\(
	\mathbf U_K
\), it differs from
\(
	\mathbf{IP}_{\Ring_K^\fp}
\) from \cite[\S 4]{isotropic-elem} only by the external exact completion necessary to define group objects by generators and relations. Recall that the category of unital algebras
\(
	\Ring_K
\) is the ind-completion of the category of finitely presented unital algebras
\(
	\Ring_K^\fp
\) by \cite[corollary 6.3.5]{kashiwara-schapira}, so
\(
	\mathbf P_K
\) may be identified with the category of all functors
\(
	\Ring_K \to \Set
\) preserving direct limits. For any unital algebra \(R\) over \(K\) consider the ``evaluation'' functor
\[
	\ev_R \colon \mathbf P_K \to \Set,\,
	F \mapsto \varinjlim_{R' \to R} F(R'),
\]
where \(R'\) runs over all finitely presented unital algebras with homomorphisms to \(R\) (if \(R\) is already finitely presented, then
\(
	\ev_R(F) = F(R)
\)). Composing the induced functors
\[
	\Ind(\mathbf P_K) \to \Ind(\Set),
\quad
	\Ex(\Ind(\mathbf P_K)) \to \Ex(\Ind(\Set))
\]
with the evaluation functors
\[
	\Ind(\Set) \to \Set,
\quad
	\Ex(\Ind(\Set)) \to \Set
\]
we obtain two functors
\[
	\ev_R \colon \Ind(\mathbf P_K) \to \Set,
\quad
	\ev_R \colon \Ex(\Ind(\mathbf P_K)) \to \Set.
\]

Recall that an object \(X\) in a regular category
\(
	\mathbf C
\) is called \textit{projective} \cite[definition before lemma A1.3.8]{elephant} if every regular epimorphism to \(X\) has a section. Equivalently, \(X\) is projective if and only if
\(
	f_* \colon \mathbf C(X, Y) \to \mathbf C(X, Z)
\) is surjective for all regular epimorphisms
\(
	f \colon Y \to Z
\). If \(X\) is projective in
\(
	\mathbf C
\), then it is also projective in
\(
	\Ind(\mathbf C)
\) since every regular epimorphism in the ind-completion is just an inverse system of regular epimorphisms in
\(
	\mathbf C
\) up to an isomorphism and a morphism from \(X\) to an ind-object is given by a morphism to some its member. Also, if \(X\) is projective in
\(
	\mathbf C
\), then it is also projective in
\(
	\Ex(\mathbf C)
\) since any regular epimorphism
\(
	Y / {\equiv_Y} \to X
\) is given by a regular epimorphism
\(
	Y \to X
\) satisfying the extra condition of
\(
	\equiv_Y
\)-invariance.

All affine schemes of finite presentation are projective objects in
\(
	\mathbf P_K
\) by Yoneda lemma. It follows that they are also projective in
\(
	\Ind(\mathbf P_K)
\) and
\(
	\Ex(\Ind(\mathbf P_K))
\). From this and the definitions it easily follows that
\[
	\ev_R
	\cong \Ex(\Ind(\mathbf P_K))\bigl(
		\Spec(R),
		{-}
	\bigr)
\]
if \(R\) is finitely presented. For arbitrary \(R\) we have
\[
	\ev_R(X) \cong \varinjlim_{R' \to R} \ev_{R'}(X),
\]
where \(R'\) runs over finitely presented algebras with homomorphisms to \(R\). In particular,
\[
	\ev_K(X) \cong \Ex(\Ind(\mathbf P_K))(\Spec(K), X)
\]
is the set of global elements of \(X\) since
\(
	\Spec(K)
\) is the terminal object.

We have the following commutative diagram of positive categories and coherent functors (up to natural isomorphisms, but they may be chosen to be the identities).
\[\xymatrix@R=45pt@C=120pt@M=6pt@!0{
&
	\mathbf P_K
	\ar_{\ev_R}[dl]
	\ar@{^{(}->}[d]
	\ar@{^{(}->}[r]
&
	\Pro(\mathbf P_K)
	\ar@{^{(}->}[d]
\\
	\Set
&
	\Ind(\mathbf P_K)
	\ar_{\ev_R}[l]
	\ar@{^{(}->}[d]
	\ar@{^{(}->}[r]
&
	\Ind(\Pro(\mathbf P_K))
	\ar@{^{(}->}[d]
\\ &
	\Ex(\Ind(\mathbf P_K))
	\ar_{\ev_R}[ul]
	\ar@{^{(}->}[r]
&
	\mathbf U_K
}\]
The functors denoted by \(\hookrightarrow\) are fully faithful. All functors in this diagram are geometric except
\[
	\Set
	\hookrightarrow \Pro(\mathbf P)
	\hookrightarrow \Ind(\Pro(\Set)),
\quad
	\Set \hookrightarrow \Ind(\Set).
\]

The diagram is functorial on \(K\). Namely, for every unital \(K\)-algebra \(R\) there exists a functor
\[
	({-})_R \colon \mathbf P_K \to \mathbf P_R
\]
such that
\(
	(F_R)(S) = \ev_S(F)
\) for all
\(
	S \in \Ring_R^\fp
\), in particular,
\(
	\Spec(T)_R \cong \Spec(T \otimes_K R)
\) for any
\(
	T \in \Ring_K^\fp
\). This functor is geometric and induces geometric base change functors
\(
	({-})_R
\) between the categories from the diagram above (except \(\Set\)), though
\(
	({-})_R
	\colon \Pro(\mathbf P_K)
	\to \Pro(\mathbf P_R)
\) is only coherent. If \(S\) is any unital \(R\)-algebra, then
\[
	{\ev_S} \circ ({-})_R
	\cong \ev_S
	\colon \mathbf U_K
	\to \Set,
\quad
	(({-})_R)_S
	\cong ({-})_S
	\colon \mathbf U_K
	\to \mathbf U_S.
\]

\begin{lemma} \label{fp-change}
	If \(R\) is a unital \(K\)-algebra of finite presentation, then the the category
	\(
		\mathbf U_R
	\) is naturally equivalent to the slice category
	\(
		\mathbf U_K / \Spec(R)
	\) and similarly for its subcategories from the diagram above. The base change functor is just
	\[
		({-})_R \cong ({-}) \times \Spec(R)
		\colon \mathbf U_K
		\to \mathbf U_K / \Spec(R)
	\]
	and it has a faithful right adjoint, namely, the forgetful functor
	\(
		\mathbf U_K / \Spec(R) \to \mathbf U_K
	\).
\end{lemma}
\begin{proof}
	Indeed, there are the following facts.
	\begin{itemize}

		\item
		If
		\(
			\mathbf C
		\) is a small category and
		\(
			X \in \mathbf C
		\) is its object, then
		\[
			\Cat((\mathbf C / X)^\op, \Set)
			\sim \Cat(\mathbf C^\op, \Set) / \mathbf y(X)
		\]
		by \cite[lemma 1.4.12]{kashiwara-schapira}, where
		\(
			\mathbf y
			\colon \mathbf C
			\to \Cat(\mathbf C^\op, \Set)
		\) is the Yoneda embedding. It follows that
		\[
			\mathbf P_R \sim \mathbf P_K / \Spec(R)
		\]
		since the Yoneda embedding in this case is just
		\(
			\Spec({-})
		\) (recall that we identify schemes and the corresponding presheaves).

		\item
		If
		\(
			\mathbf C
		\) is any category with an object \(X\), then the canonical functor
		\(
			\Pro(\mathbf C / X) \to \Pro(\mathbf C) / X
		\) is an equivalence. Indeed, any morphism in
		\(
			\Pro(\mathbf C)
		\) to \(X\) is isomorphic to the inverse limit of morphisms to \(X\) \cite[corollary 6.1.14]{kashiwara-schapira}, so the functor is essentially surjective. It is easy to check that the functor is also fully faithful. It follows that
		\[
			\Pro(\mathbf P_R)
			\sim \Pro(\mathbf P_K) / \Spec(R).
		\]

		\item
		Similarly, for any category
		\(
			\mathbf C
		\) and its object \(X\) the canonical functor
		\(
			\Ind(\mathbf C / X) \to \Ind(\mathbf C) / X
		\) is an equivalence, so
		\[
			\Ind(\mathbf P_R)
			\sim \Ind(\mathbf P_K) / \Spec(R),
		\quad
			\Ind(\Pro(\mathbf P_R))
			\sim \Ind(\Pro(\mathbf P_K)) / \Spec(R).
		\]

		\item
		The properties of a category
		\(
			\mathbf C
		\) to be regular, Barr exact, positive, infinitary positive, a pretopos, or an infinitary pretopos are preserved by passing to a slice category
		\(
			\mathbf C / X
		\) \cite[remarks after lemma A1.3.3, before example A1.3.7, and after corollary A1.4.4]{elephant}. Namely, equalizers and fiber products in
		\(
			\mathbf C / X
		\) are the same as in
		\(
			\mathbf C
		\), binary products are the fiber products over \(X\), the terminal object is just \(X\). Regular epimorphisms, images, subobject (semi-)lattices, and factor-object by equivalence relations are the same as in
		\(
			\mathbf C
		\). It is easy to see that if
		\(
			\mathbf C
		\) is regular and \(X\) is any object, then
		\(
			\Ex(\mathbf C / X) \sim \Ex(\mathbf C) / X
		\). Thus
		\[
			\Ex(\Ind(\mathbf P_R))
			\sim \Ex(\Ind(\mathbf P_K)) / \Spec(R),
		\quad
			\mathbf U_R
			\sim \mathbf U_K / \Spec(R). \qedhere
		\]

	\end{itemize}
\end{proof}

\subsection{Actions of groups and rings}

We also need several algebraic notions making sense in infinitary pretopoi.

A group object \(G\) \textit{acts} on an object \(X\) in a category
\(
	\mathbf C
\) with finite products if there is a morphism
\(
	\up{({-})}{({=})} \colon G \times X \to X
\) such that
\(
	\up {g g'} x = \up g {(\up {g'} x)}
\) and
\(
	\up 1 x = x
\). If \(X\) is also a group object, we usually impose the additional axiom
\(
	\up g {(x x')} = \up g x \up g {x'}
\). In this case the \textit{semi-direct product}
\(
	X \rtimes G
\) is also a group object, it is the product of \(X\) and \(G\) in
\(
	\mathbf C
\) with the multiplication
\[
	(x \rtimes g) (x' \rtimes g')
	= (x\, \up g {x'}) \rtimes g g'.
\]
A \textit{crossed module} is a homomorphism
\(
	\delta \colon X \to G
\) of group objects together with an action of \(G\) on \(X\) (by group automorphisms) such that
\(
	\delta(\up g x) = \up g {\delta(x)}
\) and
\(
	\up x y = \up {\delta(x)} y
\) for
\(
	x, y \colon X
\) and
\(
	g \colon G
\).

Similarly, a unital ring object \(R\) \textit{acts} on a ring object \(S\) in a category
\(
	\mathbf C
\) with finite products if there is a biadditive multiplication morphism
\(
	R \times S \to S
\) such that
\(
	r (s s') = (r s) s'
\),
\(
	1 s = s
\), and
\(
	(r r') s = r (r' s)
\). In this case the \textit{semi-direct product}
\(
	S \rtimes R
\) is also a unital ring object, it is the object
\(
	S \times R
\) of
\(
	\mathbf C
\) with the operations
\[
	(s \rtimes r) + (s' \rtimes r')
	= (s + s') \rtimes (r + r'),
\quad
	(s \rtimes r) (s' \rtimes r')
	= (s s' + r s' + r' s) \rtimes r r'.
\]
A homomorphism
\(
	\delta \colon S \to R
\) in
\(
	\mathbf C
\) from a ring object to a unital ring object is called a \textit{ring crossed module} if \(R\) acts on \(S\),
\(
	\delta(r s) = r \delta(s)
\), and
\(
	s s' = \delta(s) s'
\). Clearly, actions of \(R\) on a unital ring \(S\) are in a one-to-one correspondence with unital homomorphisms
\(
	f \colon R \to S
\), namely,
\(
	r s = f(r) s
\) and
\(
	f(r) = r 1_S
\) for
\(
	r \colon R
\) and
\(
	s \colon S
\).

\begin{lemma} \label{gen-rel}
	Let
	\(
		\mathbf C
	\) be an infinitary pretopos.
	\begin{enumerate}

		\item
		If \(X\) is any object, then
		\(
			\bigsqcup_{n = 0}^\infty
				X^{\times n}
		\) (the object of formal strings on \(X\)) is the free monoid on \(X\), i.e. every morphism
		\(
			X \to M
		\) to a monoid object in
		\(
			\mathbf C
		\) continues to a unique homomorphism
		\(
			\bigsqcup_{n = 0}^\infty
				X^{\times n}
			\to M
		\).

		\item
		Let \(X\) be an object,
		\(
			X^{-1}
		\) be its copy with an isomorphism
		\(
			({-})^{-1} \colon X \to X^{-1}
		\),
		\[
			M
			= \bigsqcup_{n = 0}^\infty
				(X \sqcup X^{-1})^{\times n}
			= \bigsqcup_{n = 0}^\infty
				\bigsqcup_{
					\varepsilon_1,
					\ldots,
					\varepsilon_n
					\in \{-1, 1\}
				}
					\prod_{i = 1}^n
						X^{\varepsilon_i}
		\]
		be the free monoid on
		\(
			X \sqcup X^{-1}
		\), and
		\(
			R \subseteq M \times M
		\) be a subobject with the components
		\[
			R_{\vec \varepsilon, \vec \eta}
			\subseteq \prod_{i = 1}^n
				X^{\varepsilon_i}
			\times \prod_{j = 1}^m
				X^{\eta_j}
		\]
		for all
		\(
			n, m \in \mathbb N_0
		\) and
		\(
			\varepsilon_i, \eta_j \in \{-1, 1\}
		\). Consider the equivalence relation \(E\) on \(M\) generated by
		\[
			\bigl[\!\!\!\:\bigl[
				\exists u^M, z^M, w^M, v^M
					\bigl(
						l = u z v
						\wedge r = u w v
						\wedge R(z, w)
					\bigr)
			\bigr]\!\!\!\:\bigr]_{l^M, r^M}
		\]
		and by
		\[
			\bigl[\!\!\!\:\bigl[
				\exists u^M, x^X, v^M
					\bigl(
						l
							= u
							x^\varepsilon
							x^{-\varepsilon}
							v
						\wedge r = u v
					\bigr)
			\bigr]\!\!\!\:\bigr]_{l^M, r^M}
		\]
		for
		\(
			\varepsilon \in \{-1, 1\}
		\). Then
		\(
			\langle X \mid R \rangle = M / E
		\) is the group object generated by \(X\) with relations \(R\), i.e. for any group object \(H\) and a morphism
		\(
			f \colon X \to H
		\) such that
		\[
			R_{\vec \varepsilon, \vec \eta}(
				x_1^{\varepsilon_1},
				\ldots,
				x_n^{\varepsilon_n};
				y_1^{\eta_1},
				\ldots,
				y_m^{\eta_m}
			)
			\vdash \prod_{i = 1}^n
				f(x_i^{\varepsilon_i})
			= \prod_{j = 1}^m
				f(y_j^{\varepsilon_j})
		\]
		holds for all
		\(
			\vec \varepsilon
		\) and
		\(
			\vec \eta
		\) there is a unique homomorphism
		\(
			M / E \to H
		\) continuing \(f\).

		\item
		Let
		\(
			G = \langle X \mid R \rangle
		\) be a group object given by generators and relations, \(H\) be a group object, and \(P\) be an object. Let also
		\(
			f \colon P \times X \to H
		\) be a morphism such that
		\[
			R_{\vec{\varepsilon}, \vec \eta}(
				x_1^{\varepsilon_1},
				\ldots,
				x_n^{\varepsilon_n};
				y_1^{\eta_1},
				\ldots,
				y_m^{\eta_m}
			)
			\vdash \prod_{i = 1}^n
				f(p, x_i)^{\varepsilon_i}
			= \prod_{j = 1}^n
				f(p, y_j)^{\eta_j}
		\] holds for all
		\(
			\vec \varepsilon
		\) and
		\(
			\vec \eta
		\). Then \(f\) continues to a unique morphism
		\(
			f' \colon P \times G \to H
		\) multiplicative on the second argument.

		\item
		Let
		\(
			G = \langle X \mid R \rangle
		\) be a group object given by generators and relations, \(P\) be any object, and
		\(
			f \colon X \times P \to P
		\) be a morphism with an inverse
		\(
			f \colon X^{-1} \times P \to P
		\) (necessarily unique), i.e.
		\[
			f(x, f(x^{-1}, p))
			= p
			= f(x^{-1}, f(x, p)),
		\]
		and such that
		\[
			R_{\vec \varepsilon}(
				x_1^{\varepsilon_1},
				\ldots,
				x_n^{\varepsilon_n};
				y_1^{\eta_1},
				\ldots
				y_m^{\eta_m}
			)
			\vdash f(
				x_1^{\varepsilon_1},
				\ldots f(x_n^{\varepsilon_n}, p) \ldots
			)
			= f(
				y_1^{\eta_1},
				\ldots f(y_m^{\eta_m}, p) \ldots
			)
		\]
		holds for all
		\(
			\vec \varepsilon
		\) and
		\(
			\vec \eta
		\). Then \(f\) continues to a unique action of \(G\) on \(P\). If \(P\) is a group object and
		\[
			f(x, p_1 p_2) = f(x, p_1)\, f(x, p_2),
		\]
		then \(G\) acts on \(P\) by group automorphisms.

	\end{enumerate}
\end{lemma}
\begin{proof}
	All these claims may be proved in the same way as for \(\Set\), but using the formal geometric logic. In (3) we may actually assume that \(P\) is the terminal object by passing to the slice category
	\(
		\mathbf C / P
	\).
\end{proof}

\subsection{Co-localization and formal localization}

Take a unital ring \(K\) and a reductive group scheme \(G\) over \(K\) of local isotropic rank
\(
	\geq 3
\). We are going to construct a group object
\(
	\stlin_G \in \mathbf U_K
\) by generators and relations and then show that it actually lies in
\(
	\Ex(\Ind(\mathbf P_K))
\), so
\(
	\stlin_G(K) = \ev_K(\stlin_G) \in \Group
\) is an ordinary group. The whole Steinberg group functor will be given by
\(
	\stlin_G({-}) = \ev_{(-)}(\stlin_G)
\).

Let
\(
	s \in K
\) and
\(
	(T, \Phi)
\) be an isotropic pinning of
\(
	G_{K_s}
\). Denote by \(\R\) the unital ring object
\(
	\mathbb A^1_K \in \mathbf P_K
\), i.e.
\(
	\R(R) = R
\) for any unital \(K\)-algebra \(R\). Following \cite[\S 4]{isotropic-elem} for any unital \(K\)-algebra \(R\) let
\[
	R^{(s)} = \{a^{(s)} \mid a \in R\}
\]
be the ring crossed module over \(R\) with
\[
	a^{(s)} + b^{(s)} = (a + b)^{(s)},
\quad
	a^{(s)} b^{(s)} = (a s b)^{(s)},
\quad
	r a^{(s)} = (r a)^{(s)},
\quad
	\delta(a^{(s)}) = s a
\]
(the \textit{\(s\)-homotope} of \(R\)). Clearly, they are the values of the corresponding crossed module
\(
	\R^{(s)} \in \mathbf P_K
\) over \(\R\) (also isomorphic to
\(
	\mathbb A^1_K
\) as a scheme). object The \textit{co-localization} of \(\R\) at \(s\) is the inverse system
\[
	\R^{(s^\infty)}
	= \bigl(
		\ldots \to \R^{(s^2)} \to \R^{(s)} \to \R
	\bigr)
	\in \Pro(\mathbf P_K),
\]
where the structure maps are given by
\(
	a^{(s^{k + 1})} \mapsto (s a)^{(s^k)}
\). Since the structure maps are homomorphisms of ring crossed modules, the co-localization is a ring crossed module over \(\R\) inside the category
\(
	\Pro(\mathbf P_K)
\).

On the other hand we have the formal localization
\[
	\textstyle
	\R_s^{\Ind}
	= (
		\R \to \frac \R s \to \frac \R {s^2} \to \ldots
	)
	\in \Ind(\mathbf P_K).
\]
Here
\(
	(\frac \R {s^n})(R)
	= \{\frac r {s^n} \mid r \in R\}
\) are the sets of formal fractions for any unital \(K\)-algebra \(R\), so
\(
	\frac \R {s^n}
\) are \(\R\)-modules with
\(
	\frac p {s^n} + \frac q {s^n} = \frac{p + q}{s^n}
\) and
\(
	r \frac p {s^n} = \frac{r p}{s^n}
\), and the structure maps are
\(
	\frac p {s^n} \mapsto \frac{s p}{s^{n + 1}}
\) (each of them may be considered as a definition of the morphism in
\(
	\mathbf P_K
\) by the corresponding term or as a family of maps
\(
	\bigl(\frac \R {s^n}\bigr)(R)
	\to \bigl(\frac \R {s^{n + 1}}\bigr)(R)
\) for all unital algebras \(R\) of finite presentation). The formal localization is a unital ring object with the multiplication
\(
	\frac p {s^n} \frac q {s^m}
	= \frac{p q}{s^{n + m}}
\) and
\(
	\R \to \R^{\Ind}_s
\) is a unital homomorphism. It is easy to check that \(
	\ev_R(\R^{\Ind}_s) \cong R_s
\) for any unital \(K\)-algebra \(R\).

It turns out that
\(
	\R^{(s^\infty)}
\) is a ring crossed module over
\(
	\R^{\Ind}_s
\) with the action
\(
	\frac p {s^n} a^{(s^m)} = (p a)^{(s^{m - n})}
\) and the composite structure homomorphism
\(
	\delta
	\colon \R^{(s^\infty)}
	\to \R
	\to \R_s^{\Ind}
\) inside
\(
	\Ind(\Pro(\mathbf P_K))
\). The action of \(\R\) on
\(
	\R^{(s^\infty)}
\) is the restriction of this new action.

The ring crossed module
\(
	\R^{(s^\infty)} \to \R_s^{\Ind}
\) is preserved under base changes
\(
	\mathbf U_K \to \mathbf U_R
\). Also,
\(
	\R^{(s^\infty)} \to \R_s^{\Ind}
\) is functorial on \(s\) in the following sense. Take
\(
	t \in K
\). We have homomorphisms
\[
	t^*
	\colon \R^{((t s)^\infty)}
	\to \R^{(s^\infty)},
\quad
	t_*
	\colon \R_s^{\Ind}
	\to \R_{t s}^{\Ind}
\]
of ring objects and unital ring objects given by
\(
	t^*(a^{((t s)^n)}) = (t^n a)^{(s^n)}
\) and
\(
	t_*(\frac p {s^n}) = \frac{t^n p}{(t s)^n}
\), they satisfy the identities
\[
	p\, t^*(a)
	= t^*\bigl(t_*(p)\, a\bigr),
\quad
	(t_* \circ \delta \circ t^*)(a) = \delta(a)
\]
and are also preserved under base changes. Of course,
\(
	1^*
\) and
\(
	1_*
\) are the identity morphisms,
\(
	(t u)^* = u^* \circ t^*
\), and
\(
	(t u)_* = t_* \circ u_*
\).

If
\(
	t s = t' s
\), then
\(
	t^* = {t'}^*
\) and
\(
	t_* = t'_*
\), so these morphisms depend only on the elements \(s\) and
\(
	t s
\). For any
\(
	k \in \mathbb N_+
\) the morphisms
\(
	(s^{k - 1})^*
	\colon K^{((s^k)^\infty)}
	\to K^{(s^\infty)}
\) and
\(
	(s^{k - 1})_*
	\colon K_s^{\Ind}
	\to K_{s^k}^{\Ind}
\) are invertible with the inverses given by the identity maps
\(
	a^{(s^{k n})} \mapsto a^{(s^{k n})}
\) and
\(
	\frac p {s^{k n}} \mapsto \frac p {s^{k n}}
\) respectively. It follows that up to canonical isomorphisms the objects
\(
	\R^{(s^\infty)}
\) and
\(
	\R_s^{\Ind}
\) depend only on the open subscheme
\(
	\mathcal D(s) \subseteq \Spec(K)
\) instead of the element \(s\).

\subsection{\(\mathbf U_K\)-points of affine schemes}

Recall that any pointed affine \(K_s\)-scheme \(X\) of finite presentation is isomorphic to
\[
	\Spec\bigl(
		K_s[x_1, \ldots, x_n]
		/ (f_1, \ldots, f_m)
	\bigr)
\]
for
\(
	f_j \in K_s[\vec x]
\) with
\(
	f_j(\vec 0) = 0
\). Indeed, \(X\) has such a presentation possibly without the last condition. Since it is pointed, there is a homomorphism of \(K\)-algebras
\(
	K_s[\vec x] / (\vec f) \to K_s
\) given by
\(
	x_i \mapsto a_i \in K_s
\). Clearly,
\(
	f_j(\vec a) = 0
\), so the change of variables
\(
	x'_i = x_i - a_i
\) gives the required presentation. Any morphism
\(
	\varphi \colon X \to Y
\) between such schemes
\[
	X
	\cong \Spec\bigl(
		K_s[x_1, \ldots, x_n]
		/ (f_1, \ldots, f_m)
	\bigr),
\quad
	Y
	\cong \Spec\bigl(
		K_s[y_1, \ldots, y_{n'}]
		/ (g_1, \ldots, g_{m'})
	\bigr)
\]
preserving the distinguished section is determined by the inverse images \(h_i\) of \(y_i\), i.e. by a family
\(
	h_i \in K_s[\vec x] / (\vec f)
\) such that
\(
	h_i(\vec 0) = 0
\) and
\(
	g_j(\vec h) = 0
\).

For any
\(
	\R_s^{\Ind}
\)-algebra \(\mathcal A\) in
\(
	\mathbf U_K
\) (i.e. a ring object \(\mathcal A\) with an action of
\(
	\R_s^{\Ind}
\)) and a pointed affine \(K_s\)-scheme \(X\) of finite presentation we may construct the object
\[
	X(\mathcal A)
	= [\![
		f_1(\vec a) = 0
		\wedge \ldots
		\wedge f_m(\vec a)
		= 0
	]\!]_{a_1^{\mathcal A}, \ldots, a_n^{\mathcal A}}
\]
of \(\mathcal A\)-points of \(X\). Clearly, it is functorial on \(\mathcal A\). On the other hand,
\(
	X(\mathcal A)
\) is functorial on \(X\), since any scheme morphism
\(
	\varphi \colon X \to Y
\) given by
\(
	h_i \in K_s[\vec x] / (\vec f)
\) as above induces a morphism
\[
	[\![
		(h_1(\vec a), \ldots, h_{n'}(\vec a))
	]\!]_{\vec a^{X(\mathcal A)}}
	\colon X(\mathcal A)
	\to Y(\mathcal A).
\]
This construction preserves finite limits on each argument, so it also preserves monomorphisms. Also, the construction is preserved under base changes, i.e. if \(R\) is a unital \(K\)-algebra, then
\[
	X(\mathcal A)_R
	\cong X_{R_s}(\mathcal A_R)
	\in \mathbf U_R
\]
in a natural way.

Now let \(X\) be a group scheme (e.g. the reductive group scheme
\(
	G_{K_s}
\) or one of
\(
	P_\alpha
\)). It is pointed with the distinguished unit section. Consider the diagram of ring objects
\[\xymatrix@R=45pt@C=120pt@M=6pt@!0{
	\R^{(s^\infty)}
	\ar@/_12pt/^(0.45){({-}) \rtimes 0}[r]
&
	\R^{(s^\infty)} \rtimes \R_s^{\Ind}
	\ar@/^12pt/^(0.55){({=})}[r]
	\ar@/_12pt/_(0.55){\delta({-}) + ({=})}[r]
&
	\R_s^{\Ind}.
	\ar_(0.45){0 \rtimes ({-})}[l]
}\]
Here
\(
	0 \rtimes ({-})
\) is a common section of two morphisms and
\(
	({-}) \rtimes 0
\) is the kernel of the projection
\(
	({=})
\) on the second factor. It follows that the image of this diagram under
\(
	X({-})
\) has the same properties, in particular,
\[
	X(\R^{(s^\infty)} \rtimes \R_s^{\Ind})
	\cong X(\R^{(s^\infty)}) \rtimes X(\R_s^{\Ind})
\]
for some action of
\(
	X(\R_s^{\Ind})
\) on
\(
	X(\R^{(s^\infty)})
\) (induced by the conjugation on
\(
	X(\R^{(s^\infty)} \rtimes \R_s^{\Ind})
\)) and
\[
	X\bigl(\delta({-}) + ({=})\bigr)(g \rtimes h)
	= X(\delta)(g)\, h
\]
after this identification. The kernels of
\(
	({=})
\) and
\(
	\delta({-}) + ({=})
\) have trivial product, so their embeddings to
\(
	\R^{(s^\infty)} \rtimes \R_s^{\Ind}
\) continue to a unique homomorphism
\[
	\Ker({=})
		\times \Ker\bigl(\delta({-}) + ({=})\bigr)
	\to \R^{(s^\infty)} \rtimes \R_s^{\Ind}.
\]
Since a similar homomorphism exists after applying
\(
	X({-})
\), the kernels of
\(
	X({=})
\) and
\(
	X\bigl(\delta({-}) + ({=})\bigr)
\) commute in the group sense. It easily follows that
\[
	\delta
	\colon X(\R^{(s^\infty)})
	\to X(\R_s^{\Ind})
\]
is a crossed module.

This crossed module is extra-natural on \(s\), i.e. if
\(
	t \in K
\) is any element, then
\[
	(t_* \circ \delta \circ t^*)(x) = \delta(x)
	\colon X(\R_{s t}^{\Ind}),
\quad
	t^*\bigl(\up {t_*(g)} x\bigr)
	= \up g {\bigl(t^*(x)\bigr)}
	\colon X(\R^{(s^\infty)})
\]
for
\(
	x \colon X(\R^{((s t)^\infty)})
\),
\(
	g \colon X(\R_s^{\Ind})
\). The last identity may be proved by considering the composition
\[
	X(\R^{((s t)^\infty)}) \rtimes X(\R_s^{\Ind})
	\cong X\bigl(
		\R^{((s t)^\infty)} \rtimes \R_s^{\Ind}
	\bigr)
	\to X\bigl(
		\R^{(s^\infty)} \rtimes \R_s^{\Ind}
	\bigr)
	\cong X(\R^{(s^\infty)}) \rtimes X(\R_s^{\Ind}),
\]
where the middle morphism is induced by the homomorphism
\(
	t^*({-}) \rtimes ({=})
\) of ring objects. Then the composition is given by the same formula and since it is product-preserving, the action is extra-natural.

\subsection{Co-local Steinberg groups}

Now fix a unital ring \(K\) and a reductive group scheme \(G\) over \(K\) of local isotropic rank at least \(3\). Take \(s\) such that
\(
	G_{K_s}
\) has an isotropic pinning
\(
	(T, \Phi)
\) of rank at least \(3\). For any
\(
	\R_s^{\Ind}
\)-algebra
\(
	\mathcal A
\) in
\(
	\mathbf U_K
\) (e.g.
\(
	\R_s^{\Ind}
\) itself or
\(
	\R^{(s^\infty)}
\)) let
\[
	\stlin_{G, T, \Phi}(\mathcal A)
	\in \mathbf U_K
\]
be the group object generated by
\(
	\bigsqcup_{\alpha \in \Phi}
		P_\alpha(\mathcal A)
\), the generating morphisms
\(
	P_\alpha(\mathcal A)
	\to \stlin_{G, T, \Phi}(\mathcal A)
\) are denoted by
\(
	x_\alpha
\). The relations are
\begin{itemize}

	\item
	\(
		x_\alpha(p)\, x_\alpha(q)
		= x_\alpha(p \dotplus q)
	\) for
	\(
		p, q \colon P_\alpha(\mathcal A)
	\);

	\item
	\(
		x_\alpha(p) = x_{2 \alpha}(p)
	\) for
	\(
		p \colon P_{2 \alpha}(\mathcal A)
	\) if
	\(
		\alpha \in \Phi
	\) is ultrashort in a component of type
	\(
		\mathsf{BC}_\ell
	\);

	\item
	\(
		[x_\alpha(p), x_\beta(q)]
		= \prod_{
			\substack{
				i \alpha + j \beta \in \Phi
			\\
				i, j \in \mathbb N_+
			}
		}
			x_{i \alpha + j \beta}\bigl(
				f_{\alpha, \beta}
					^{i \alpha + j \beta}
					(p, q)
			\bigr)
	\) for
	\(
		p \colon P_\alpha(\mathcal A)
	\) and
	\(
		q \colon P_\beta(\mathcal A)
	\), where \(\alpha\) and \(\beta\) are not anti-parallel.

\end{itemize}
Clearly, there is a canonical homomorphism
\(
	\stmap
	\colon \stlin_{G, T, \Phi}(\mathcal A)
	\to G(\mathcal A)
\) such that
\(
	\stmap(x_\alpha(p)) = t_\alpha(p)
\) for
\(
	p \colon P_\alpha(\mathcal A)
\), its image is precisely the co-local elementary group
\(
	\elem_{G, T, \Phi}(\mathcal A)
\) from \cite[\S 5]{isotropic-elem}. In particular, we have group objects
\(
	\stlin_{G, T, \Phi}(\R^{(s^\infty)})
\) and
\(
	\stlin_{G, T, \Phi}(\R_s^{\Ind})
\), they are functorial on \(s\) in various senses (in the same way as the underlying ring objects) and they are also preserved under base changes.

Clearly, for every unipotent subset
\(
	\Sigma \subseteq \Phi
\) the multiplication morphism
\[
	\bigl[\!\!\!\:\bigl[
		\prod_\alpha x_\alpha(p_\alpha)
	\bigr]\!\!\!\:\bigr]
	\colon \prod_{\alpha \in \Sigma \setminus 2 \Sigma}
		P_\alpha(\mathcal A)
	\to \stlin_{G, T, \Phi}(\mathcal A)
\]
is a monomorphism. Thus the relations of
\(
	\stlin_{G, T, \Phi}(\mathcal A)
\) imply that
\(
	U_\Sigma(\mathcal A)
\) may be considered as a subgroup of both
\(
	\stlin_{G, T, \Phi}(\mathcal A)
\) and
\(
	G(\mathcal A)
\).

The following lemma is a generalization of \cite[lemma 5]{isotropic-elem}. We are going to apply it to
\(
	K_s
\) itself,
\(
	P_\alpha(K_s)
\), and
\(
	P_\alpha(K_s) \times P_\beta(K_s)
\).
\begin{lemma} \label{power-idem}
	Let
	\(
		s \in K
	\) be any element and
	\(
		(M, M_0)
	\) be a \(2\)-step nilpotent
	\(
		K_s
	\)-module such that
	\(
		M_0
	\) and
	\(
		M / M_0
	\) are free \(K_s\)-modules of finite rank. Then the morphism
	\[
		\mu
		= ({-}) \cdot ({=})^k
		\colon \mathbb W(M)(\R^{(s^\infty)})
			\times \R^{(s^\infty)}
		\to \mathbb W(M)(\R^{(s^\infty)})
	\]
	is an epimorphism for every
	\(
		k > 0
	\). It generates the target group object with the only relations
	\[
		\mu(p \dotplus q, a)
		= \mu(p, a) \dotplus \mu(q, a)
	\]
	for
	\(
		p, q
		\colon \mathbb W(M)(\R^{(s^\infty)})
	\),
	\(
		a \colon \R^{(s^\infty)}
	\) and
	\[
		\mu(p \cdot a^k, b) = \mu(p, a b)
	\]
	for
	\(
		p \colon \mathbb W(M)(\R^{(s^\infty)})
	\) and
	\(
		a, b \colon \R^{(s^\infty)}
	\).
\end{lemma}
\begin{proof}
	Recall that
	\[
		\mathbb W(M)(\R^{(s^\infty)})
		\cong (\R^{(s^\infty)})^k
		\dotoplus (\R^{(s^\infty)})^l
	\]
	with the operations
	\[
		(x \dotoplus y) \cdot a = a x \dotoplus a^2 y,
	\quad
		a (0 \dotoplus y) = 0 \dotoplus a y,
	\quad
		(x_1 \dotoplus y_1)
			\dotplus (x_2 \dotoplus y_2)
		= (x_1 + x_2)
			\dotoplus (y_1 + c(x_1, x_2) + y_2)
	\]
	for some bilinear form
	\(
		c
		\colon \R_s^k \times \R_s^k
		\to \R_s^l
	\). The morphism \(\mu\) is a regular epimorphism in
	\(
		\Pro(\mathbf P_K)
	\) since it is the inverse limit of surjective morphisms
	\begin{align*}
		(\R^{(s^n)})^k
			\times (\R^{(s^n)})^l
			\times \R^{(s^n)}
		&\to (\R^{(s^{(k + 1) n})})^k
			\dotoplus (\R^{(s^{(2 k + 1) n})})^l,
	\\
		(\vec x^{(s^n)}, \vec y^{(s^n)}, a^{(s^n)})
		&\mapsto (a^k \vec x)^{(s^{(k + 1) n})}
			\dotoplus (a^{2 k} \vec y)
				^{(s^{(2 k + 1) n})}
	\end{align*}
	of presheaves. The kernel pair of \(\mu\) is the inverse limit of the binary relations
	\[
		\bigl[\!\!\!\:\bigl[
			\exists
				b_1, b_2, c \in \R^{(s^n)}
				\enskip \bigl(
					b_1^k \vec x_1 = b_2^k \vec x_2
					\wedge b_1^{2 k} \vec y_1
						= b_2^{2 k} \vec y_2
					\wedge a_1 = b_1 c
					\wedge a_2 = b_2 c
				\bigr)
		\bigr]\!\!\!\:\bigr]
	\]
	for
	\[
		(
			\vec x_1^{(s^n)},
			\vec y_1^{(s^n)},
			a_1^{(s^{2 n})}
		),
		(
			\vec x_2^{(s^n)},
			\vec y_2^{(s^n)},
			a_2^{(s^{2 n})}
		)
		\in (\R^{(s^n)})^k
			\times (\R^{(s^n)})^l
			\times \R^{(s^{2 n})},
	\]
	this is obvious if we consider \(\mu\) as the projective limit of the morphisms
	\[
		(\mathcal R^{(s^n)})^l
			\times (\mathcal R^{(s^n)})^m
			\times \mathcal R^{(s^{2 n})}
		\to (\mathcal R^{(s^{(2 k + 1) n})})^l
			\times (\mathcal R^{(s^{(4 k + 1) n})})^m
	\]
	and take
	\(
		c = 1^{(s^n)}
	\). In other words, the kernel pair of \(\mu\) is just
	\(
		E \circ E^\op
	\), where \(E\) is the image of
	\[
		\bigl[\!\!\!\:\bigl[
			\bigl((p \cdot a^k, b), (p, a b)\bigr)
		\bigr]\!\!\!\:\bigr]_{
			p \colon \mathbf W(M)(\R^{(s^\infty)}),\,
			a \colon \R^{(s^\infty)},\,
			b \colon \R^{(s^\infty)}
		}.
	\]

	It remains to show that the identity morphism
	\(
		\id
		\colon \mathbb W(M)(\R^{(s^\infty)})
		\to \mathbb W(M)(\R^{(s^\infty)})
	\) generates this group object and satisfies the only relation
	\[
		\id(p \cdot a^k \dotplus q \cdot a^k)
		= \id(p \cdot a^k) \dotplus \id(q \cdot a^k)
	\]
	for
	\(
		p, q \colon \mathbb W(M)(\R^{(s^\infty)})
	\) and
	\(
		a \colon \R^{(s^\infty)}
	\). Indeed, since \(\mu\) is a regular epimorphism both for \(M\) and for
	\(
		M \times M
	\), this relation is equivalent to
	\(
		\id(p \dotplus q) = \id(p) \dotplus \id(q)
	\). The claim now follows from the trivial fact that every group object \(G\) in an infinitary pretopos
	\(
		\mathbf C
	\) is generated by
	\(
		\id \colon G \to G
	\) with the only relation
	\(
		\id(g h) = \id(g)\, \id(h)
	\).
\end{proof}

\section{Root elimination}

Let us say that a subspace
\(
	V \leq \mathbb R \Phi
\) is \textit{\(k\)-small} if for every component
\(
	\Psi \subseteq \Phi
\) the codimension of
\(
	V \cap \mathbb R \Psi
\) in
\(
	\mathbb R \Psi
\) is at least \(k\). Every subspace \(V\) is \(0\)-small, and it is \(1\)-small if and only if it contains no components of \(\Phi\). If
\(
	V \leq V' \leq \mathbb R \Phi
\) are such that
\(
	\dim(V') \leq \dim(V) + 1
\) and \(V\) is
\(
	(k + 1)
\)-small, then \(V'\) is \(k\)-small.

\begin{lemma} \label{small-subsp}
	Let \(\Phi\) be a root system and
	\(
		V \leq \mathbb R \Phi
	\) be a subspace.
	\begin{enumerate}

		\item
		If \(V\) is \(1\)-small and
		\(
			\alpha \in \Phi \cap V
		\), then there is an irreducible saturated root subsystem
		\(
			\alpha \in \Psi \subseteq \Phi
		\) of rank \(2\) such that
		\(
			\mathbb R \Psi \cap V = \mathbb R \alpha
		\).

		\item
		If \(V\) is \(2\)-small and
		\(
			\alpha, \beta \in \Phi
		\) are linearly independent roots such that
		\(
			(\mathbb R \alpha + \mathbb R \beta) \cap V
		\) is one-dimensional, then there is an irreducible saturated root subsystem
		\(
			\alpha \in \Psi \subseteq \Phi
		\) of rank \(2\) such that
		\(
			\beta \notin \Psi
		\) and
		\(
			(\mathbb R \Psi + \mathbb R \beta) \cap V
		\) is still one-dimensional.

		\item
		If \(V\) is \(2\)-small and
		\(
			\alpha, \beta \in \Phi \cap V
		\) are linearly independent, then there is a rank \(2\) irreducible saturated root subsystem
		\(
			\alpha \in \Psi \subseteq \Phi
		\) such that
		\(
			\beta \notin \Psi
		\) and
		\(
			\mathbb R \Psi \cap V = \mathbb R \alpha
		\).

		\item
		If \(V\) is \(2\)-small,
		\(
			\alpha \in \Phi \cap V
		\) and
		\(
			\beta, \gamma \in \Phi \setminus V
		\) are linearly independent, and
		\(
			(
				\mathbb R \alpha
				+ \mathbb R \beta
				+ \mathbb R \gamma
			)
			\cap V
		\) is two-dimensional, then there is a rank \(2\) irreducible saturated root subsystem
		\(
			\alpha \in \Psi \subseteq \Phi
		\) such that
		\(
			\beta, \gamma \notin \Psi
		\) and both intersections
		\(
			(\mathbb R \Psi + \mathbb R \beta) \cap V
		\),
		\(
			(\mathbb R \Psi + \mathbb R \gamma) \cap V
		\) are one-dimensional (i.e. coincide with
		\(
			\mathbb R \alpha
		\)).

		\item
		If \(V\) is \(2\)-small, the roots
		\(
			\alpha \in \Phi \cap V
		\) and
		\(
			\beta, \gamma \in \Phi \setminus V
		\) are linearly independent, and
		\(
			(
				\mathbb R \alpha
				+ \mathbb R \beta
				+ \mathbb R \gamma
			)
			\cap V
		\) is two-dimensional, then there exists an irredicuble saturated root subsystem
		\(
			\gamma \in \Psi \subseteq \Phi
		\) of rank \(2\) such that
		\(
			\mathbb R \alpha
			+ \mathbb R \beta
			+ \mathbb R \Psi
		\) is four-dimensional and
		\(
			(
				\mathbb R \Psi
				+ \mathbb R \alpha
				+ \mathbb R \beta
			)
			\cap V
		\) is still two-dimensional.

	\end{enumerate}
\end{lemma}
\begin{proof}
	Note that for every root
	\(
		\alpha \in \Phi
	\) the sum of spans of all rank \(2\) irreducible saturated root subsystems
	\(
		\alpha \in \Psi \subseteq \Phi
	\) contains the component
	\(
		\Phi_0 \subseteq \Phi
	\) containing \(\alpha\). Indeed, otherwise there is a hyperplane
	\(
		\alpha \in H \leq \mathbb R \Phi_0
	\) such that \(\alpha\) is orthogonal to
	\(
		\Phi_0 \setminus H
	\). This is impossible e.g. by \cite[lemma 1]{isotropic-elem}.
	\begin{enumerate}

		\item
		By the \(1\)-smallness of \(V\), not all rank \(2\) irreducible saturated root subsystems
		\(
			\alpha \in \Psi \subseteq \Phi
		\) are contained in \(V\).

		\item
		Suppose that for every rank \(2\) irreducible saturated root subsystem
		\(
			\alpha \in \Psi \subseteq \Phi
		\) either
		\(
			\beta \in \Psi
		\) or
		\(
			(\mathbb R \Psi + \mathbb R \beta) \cap V
		\) is two-dimensional. If
		\(
			\alpha \notin V
		\), then
		\(
			\mathbb R \Psi \cap V
		\) is non-zero for every such \(\Psi\), so
		\(
			\mathbb R \alpha + V
		\) contains the component
		\(
			\Phi_0 \subseteq \Phi
		\) containing \(\alpha\). If
		\(
			\alpha \in V
		\), then
		\(
			\mathbb R \beta + V
		\) contains all such \(\Psi\) and the whole component \(\Phi_0\). Both cases contradict the \(2\)-smallness of \(V\).

		\item
		This follows from the claim 1.

		\item
		Assume the contrary. The subspace
		\(
			\mathbb R \beta + \mathbb R \gamma + V
		\) has dimension
		\(
			\dim(V) + 1
		\) and contains all rank \(2\) irreducible saturated root subsystems
		\(
			\alpha \in \Psi \subseteq \Phi
		\), this contradicts the \(2\)-smallness of \(V\).

		\item
		If this is false, then the subspace
		\(
			V + \mathbb R \gamma
		\) of dimension
		\(
			\dim(V) + 1
		\) contains all irreducible saturated root subsystems
		\(
			\gamma \in \Psi \subseteq \Phi
		\) of rank \(2\), a contradiction.
		\qedhere

	\end{enumerate}
\end{proof}

Let
\(
	V \leq \mathbb R \Phi
\) be a subspace. Consider the Steinberg group
\[
	\stlin_{G, T, \Phi \setminus V}(
		\R^{(s^\infty)}
	)
	\in \mathbf U_K
\]
with ``eliminated'' \(V\). Namely, it is the group object generated by
\[
	x_\alpha
	\colon P_\beta(\R^{(\infty, s)})
	\to \stlin_{G, T, \Phi \setminus V}(
		\R^{(s^\infty)}
	)
\]
for
\(
	\alpha \in \Phi \setminus V
\) with the only relations
\begin{itemize}

	\item
	\(
		x_\alpha(p)\, x_\alpha(q)
		= x_\alpha(p \dotplus q)
	\) for
	\(
		p, q \colon P_\alpha(\R^{(s^\infty)})
	\) and
	\(
		\alpha \in V
	\);

	\item
	\(
		x_\alpha(p) = x_{2 \alpha}(p)
	\) for
	\(
		p \colon P_{2 \alpha}(\R^{(s^\infty)})
	\) if
	\(
		\alpha \in \Phi \setminus V
	\) is ultrashort in a component of type
	\(
		\mathsf{BC}_\ell
	\);

	\item
	\(
		[x_\alpha(p), x_\beta(q)]
		= \prod_\gamma
			x_\gamma(
				f_{\alpha, \beta}^\gamma(p, q)
			)
	\) for
	\(
		p \colon P_\alpha(\R^{(s^\infty)})
	\) and
	\(
		q \colon P_\beta(\R^{(s^\infty)})
	\), where \(\alpha\) and \(\beta\) are not anti-parallel and
	\(
		(
			\mathbb R_{\geq 0} \alpha
			+ \mathbb R_{\geq 0} \beta
		) \cap V
		= \{0\}
	\).

\end{itemize}

For any subspaces
\(
	V' \leq V \leq \mathbb R \Phi
\) let also
\[
	F_V^{V'}
	\colon \stlin_{G, T, \Phi \setminus V}(
		\R^{(s^\infty)}
	)
	\to \stlin_{G, T, \Phi \setminus V'}(
		\R^{(s^\infty)}
	)
\]
be the canonical homomorphism, i.e.
\(
	F_V^{V'}(x_\alpha(p)) = x_\alpha(p)
\) for
\(
	p \colon P_\alpha(\R^{(s^\infty)})
\) and
\(
	\alpha \in \Phi \setminus V
\). These group objects and homomorphisms are also preserved under base changes and functorial on \(s\).

We say that roots
\(
	\alpha, \beta \in \Phi
\) are \textit{neighbors} if they are linearly independent, non-orthogonal, and
\(
	(\mathbb R_{> 0} \alpha + \mathbb R_{> 0} \beta)
	\cap \Phi
	= \varnothing
\). If \(\Phi\) has no components of type
\(
	\mathsf G_2
\), then this is equivalent to
\(
	0 < \angle(\alpha, \beta) < \frac \pi 2
\).

\begin{lemma} \label{elim-sur}
	Let
	\(
		V' \leq V \leq \mathbb R \Phi
	\) be \(1\)-small subspaces. Then
	\[
		F_V^{V'}
		\colon \stlin_{G, T, \Phi / V}(\R^{(s^\infty)})
		\to \stlin_{G, T, \Phi / V'}(\R^{(s^\infty)})
	\]
	is an epimorphism. In particular,
	\(
		\stlin_{G, T, \Phi}(\R^{(s^\infty)})
	\) is generated by the images of
	\(
		x_\alpha
	\) for
	\(
		\alpha \in \Phi \setminus V
	\).
\end{lemma}
\begin{proof}
	We have to check that
	\(
		\Image(x_\alpha)
	\) lies in the subgroup from the statement for every
	\(
		\alpha \in \Phi \cap V
	\). Choose a neighbor
	\(
		\beta \in \Phi \setminus V
	\) of \(\alpha\) using lemma \ref{small-subsp}(1). The morphism
	\[
		\R^{(s^\infty)}
		\times P_{s_\beta(\alpha)}(\R^{(s^\infty)})
		\to P_\alpha(\R^{(s^\infty)})
	\]
	given by the term
	\(
		f_{\beta, s_\beta(\alpha)}^\alpha(
			e \cdot_\beta p,
			q
		)
	\) for a Weyl element
	\(
		e \in P_\beta(K_s)
	\) is a regular epimorphism by lemmas \ref{root-units}(5) and \ref{power-idem}. Then
	\(
		\Image(x_\alpha)
	\) is contained in the subgroup generated by
	\(
		[x_\beta(p), x_{s_\beta(\alpha)}(q)]
	\) and
	\(
		x_\gamma(r)
	\) for
	\(
		\gamma
		\in (
			\mathbb N_+ \beta
			+ \mathbb N_+ s_\beta(\alpha)
		)
		\setminus V
	\).
\end{proof}

\begin{lemma} \label{vacuous-rel}
	Let
	\(
		V \leq \mathbb R \Phi
	\) be \(2\)-small subspace and
	\(
		\alpha, \beta \in \Phi
	\) be linearly independent roots such that
	\(
		(
			\mathbb R_{\geq 0} \alpha
			+ \mathbb R_{\geq 0} \beta
		)
		\cap \Phi
		\cap V
		= \varnothing
	\). Then
	\[
		[x_\alpha(p), x_\beta(q)]
		= \prod_\gamma
			x_\gamma(
				f_{\alpha, \beta}^\gamma(p, q)
			)
	\]
	holds in
	\(
		\stlin_{G, T, \Phi \setminus V}(
			\R^{(s^\infty)}
		)
	\) for
	\(
		p \colon P_\alpha(\R^{(s^\infty)})
	\) and
	\(
		q \colon P_\beta(\R^{(s^\infty)})
	\).
\end{lemma}
\begin{proof}
	If
	\(
		(
			\mathbb R_{\geq 0} \alpha
			+ \mathbb R_{\geq \beta}
		)
		\cap V
		= 0
	\), then there is nothing to prove, so from now on we assume that this intersection is a ray \(\rho\) not containing any roots. Using induction by
	\(
		\angle(\alpha, \beta)
	\), we may assume that the claim holds for all pairs of roots with smaller angle. Choose a root
	\(
		\gamma
		\in \Phi
		\setminus (
			\mathbb R \alpha
			+ \mathbb R \beta
		)
	\) neighbor to \(\alpha\) such that
	\(
		(
			\mathbb R \alpha
			+ \mathbb R \beta
			+ \mathbb R \gamma
		)
		\cap V
	\) is one-dimensional, it exists by lemma \ref{small-subsp}(2). Evaluate
	\[
		\up{
			[x_\gamma(p), x_{s_\gamma(\alpha)}(q)]
		}{x_\beta(r)}\,
		x_\beta(\dotminus r)
	\]
	in two ways using the commutator formula, either by first simplifying the commutator (such that
	\(
		x_\alpha\bigl(
			f_{\gamma, s_\gamma(\alpha)}
				^\alpha
				(p, q)
		\bigr)
	\) becomes the first factor in the case
	\(
		\angle(\alpha, \gamma) = \frac \pi 4
	\)) or just by applying four conjugations in a row. In the first case we obtain
	\[
		\bigl[
			x_\alpha\bigl(
				f_{\gamma, s_\gamma(\alpha)}
					^\alpha
					(p, q)
			\bigr),
			x_\beta(r)
		\bigr]\,
		\prod_{
			\substack{
				i \alpha + j \beta + k \gamma \in \Phi
			\\
				i, j, -k \in \mathbb N_+
			}
		}
			x_{
				i \alpha + j \beta + k \gamma
			}(t_{ijk})
	\]
	for some terms
	\(
		t_{ijk}
		\colon P_{
			i \alpha + j \beta + k \gamma
		}(\R^{(s^\infty)})
	\). In the second case and using the induction hypothesis we obtain a product
	\[
		\prod_{
			\substack{
				i \alpha + j \beta + k \gamma \in \Phi
			\\
				i, j \in \mathbb N_+,\,
				-k \in \mathbb N_{\geq 0}
			}
		}
			x_{i \alpha + j \beta + k \gamma}(t'_{ijk})
	\]
	for some terms
	\(
		t'_{ijk}
		\colon P_{
			i \alpha + j \beta + k \gamma
		}(\R^{(s^\infty)})
	\) unless there are two linearly independent roots
	\[
		\delta, \varepsilon
		\in (
			\mathbb R_{\geq 0} s_\gamma(\alpha)
			+ \mathbb R_{> 0} \beta
			+ \mathbb R_{\geq 0} \gamma
		)
		\cap \Phi
	\]
	such that
	\(
		\delta, \varepsilon
		\notin \mathbb R \alpha + \mathbb R \beta
	\),
	\(
		\rho
		\subseteq \mathbb R_{\geq 0} \delta
			+ \mathbb R_{\geq 0} \varepsilon
	\), and
	\(
		\angle(\delta, \varepsilon)
		\geq \angle(\alpha, \beta)
	\). If the ``bad'' case does not occur, we easily obtain the claim via lemmas \ref{root-units}(5) and \ref{power-idem}. In particular, the claim holds if
	\(
		\Psi
		= (
			\mathbb R \alpha
			+ \mathbb R \beta
			+ \mathbb R \gamma
		)
		\cap \Phi
	\) is reducible, i.e. if \(\beta\) is orthogonal to \(\alpha\) and \(\gamma\).

	From now on we consider only irreducible \(\Psi\). Since its rank is \(3\), it is of type
	\(
		\mathsf A_3
	\),
	\(
		\mathsf B_3
	\),
	\(
		\mathsf C_3
	\), or
	\[
		\mathsf{BC}_3
		= \{
			\pm \e_i \pm \e_j
			\mid
			1 \leq i < j \leq 3
		\}
		\cup \{
			\pm \e_i
			\mid
			1 \leq i \leq 3
		\}
		\cup \{
			\pm 2 \e_i
			\mid
			1 \leq i \leq 3
		\}.
	\]
	It is convenient to realize
	\(
		\mathsf A_3
	\) as
	\(
		\mathsf D_3 \subseteq \mathsf{BC}_3
	\), so in any case
	\(
		\Psi \subseteq \mathsf{BC}_3
	\). We are going to possibly swap \(\alpha\) with \(\beta\) and to change \(\gamma\) by another neighbor of \(\alpha\) in
	\(
		\Psi
		\setminus (\mathbb R \alpha + \mathbb R \beta)
	\) in such a way that the above argument is valid. Almost all cases (namely, such that the condition from the last column holds) up to the action of the Weyl group of
	\(
		\mathsf{BC}_3
	\) are covered by the following table.
	\begin{center}
		\begin{tabular}{|c|c|c|c|c|c|}
		\hline
			\(\Psi\)
			& \(\alpha\)
			& \(\beta\)
			& \(\angle(\alpha, \beta)\)
			& \(\gamma\)
			& condition
		\\ \hline \hline
			not \(\mathsf A_3\)
			& \(\e_1 + \e_3\)
			& \(\e_1\), \(2 \e_1\)
			& \(\pi / 4\)
			& \(\e_1 - \e_2\)
			& none
		\\ \hline
			\(\mathsf A_3\)
			& \(\e_1 + \e_3\)
			& \(\e_2 + \e_3\)
			& \(\pi / 3\)
			& \(\e_1 + \e_2\)
			& none
		\\ \hline
			not \(\mathsf A_3\)
			& \(\e_1 + \e_2\)
			& \(\e_2 + \e_3\)
			& \(\pi / 3\)
			& \(\e_1\), \(2 \e_1\)
			& none
		\\ \hline
			\(\mathsf A_3\)
			& \(\e_1 + \e_2\)
			& \(\e_1 - \e_2\)
			& \(\pi / 2\)
			& \(\e_1 - \e_3\)
			& \(\rho \neq \mathbb R_{\geq 0} \e_1\)
		\\ \hline
			not \(\mathsf A_3\)
			& \(\e_3 + \e_2\)
			& \(\e_3 - \e_2\)
			& \(\pi / 2\)
			& \(\e_3 - \e_1\)
			& none
		\\ \hline
			not \(\mathsf A_3\)
			& \(\e_1\), \(2 \e_1\)
			& \(\e_2\), \(2 \e_2\)
			& \(\pi / 2\)
			& \(\e_1 + \e_3\)
			& none
		\\ \hline
			not \(\mathsf A_3\)
			& \(\e_1 + \e_2\)
			& \(\e_3\), \(2 \e_3\)
			& \(\pi / 2\)
			& \(\e_1 + \e_3\)
			& \(
				\rho
				\neq \mathbb R_{\geq 0}
					(\e_1 + \e_2 + \e_3)
			\)
		\\ \hline
			any
			& \(\e_1 + \e_2\)
			& \(\e_3 - \e_2\)
			& \(2 \pi / 3\)
			& \(\e_2 + \e_3\)
			& none
		\\ \hline
			not \(\mathsf A_3\)
			& \(\e_1 - \e_2\)
			& \(\e_2\)
			& \(3 \pi / 4\)
			& \(\e_1 - \e_3\)
			& none
		\\ \hline
		\end{tabular}
	\end{center}

	On the other hand, suppose that
	\(
		\Phi
		\cap (
			\mathbb R_{> 0} \alpha
			+ \mathbb R_{> 0} \beta
		)
		= \varnothing
	\), this condition holds in the two cases not covered by the table. There exists a common neighbor \(\gamma\) of \(\alpha\) and \(\beta\), it is also given by the above table ignoring the last column (unless \(\Psi\) is not of type
	\(
		\mathsf A_3
	\) and
	\(
		\angle(\alpha, \beta) = \pi / 3
	\), in this case we take \(\gamma\) from the corresponding row for
	\(
		\Psi = \mathsf A_3
	\)). Evaluate
	\[
		\bigl[
			\up{x_\gamma(p)}{x_{s_\gamma(\alpha)}(q)},
			\up{x_\gamma(p)}{x_{s_\gamma(\beta)}(r)}
		\bigr]
	\] in two ways, either by moving the conjugation outside of the commutator (so the commutator becomes trivial) or by applying the commutator identities and the expanding the commutator using
	\[
		[a b, c] = \up a {[b, c]}\, [a, c],
	\enskip
		[a, b c] = [a, b]\, \up b {[a, c]}
	\]
	In the second case we order the factors in the commutator such that
	\(
		u
		= x_\alpha\bigl(
			f_{\gamma, s_\gamma(\alpha)}^\alpha(p, q)
		\bigr)
	\) and
	\(
		v
		= x_\beta\bigl(
			f_{\gamma, s_\gamma(\beta)}^\beta(p, r)
		\bigr)
	\) are the first ones. We obtain an identity
	\(
		w\, [u, v]\, w' = 1
	\) for some \(w\) and \(w'\) of the type
	\[
		\prod_{
			\substack{
				i \alpha + j \beta + k \gamma \in \Phi
			\\
				i, j, -k \in \mathbb N_+
			}
		}
			x_{i \alpha + j \beta + k \gamma}
				(t_{i j k}).
	\]
	Since \(w\) and \(w'\) are necessary inverses of each other (this may be checked in
	\(
		\stlin_{G, T, \Phi}(\R^{(s^\infty)})
	\)), we get
	\(
		[x_\alpha(p), x_\beta(q)] = 1
	\) by lemmas \ref{root-units}(5) and \ref{power-idem}.
\end{proof}

\begin{prop} \label{elim-bij}
	Let
	\(
		V' \leq V \leq \mathbb R \Phi
	\) be \(2\)-small subspaces. Then
	\[
		F_V^{V'}
		\colon \stlin_{G, T, \Phi \setminus V}(
			\R^{(s^\infty)}
		)
		\to \stlin_{G, T, \Phi \setminus V'}(
			\R^{(s^\infty)}
		)
	\]
	is an isomorphism.
\end{prop}
\begin{proof}
	Without loss of generality
	\(
		V' = 0
	\). By lemma \ref{elim-sur} we already know that
	\(
		F_V^0
	\) is an epimorphism. In order to find the inverse homomorphism we are going to construct the missing root morphisms
	\[
		x_\alpha
		\colon P_\alpha(\R^{(s^\infty)})
		\to \stlin_{G, T, \Phi \setminus V}(
			\R^{(s^\infty)}
		)
	\]
	for
	\(
		\alpha \in \Phi \cap V
	\). For any root
	\(
		\beta \in \Phi \setminus V
	\) neighbor to \(\alpha\) (it exists by lemma \ref{small-subsp}(1)) let
	\[
		x_\alpha^\beta(p, q)
		= [
			x_\beta(p),
			x_{s_\beta(\alpha)}(q)
		]\,
		\prod_{
				\gamma \notin V
		}
			x_\gamma\bigl(
				\dotminus f_{\beta, s_\beta(\alpha)}
					^\gamma
					(p, q)
			\bigr)
		\colon \stlin_{G, T, \Phi \setminus V}(
			\R^{(s^\infty)}
		)
	\]
	for
	\(
		p \colon P_\beta(\R^{(s^\infty)})
	\) and
	\(
		q \colon P_{s_\beta(\alpha)}(\R^{(s^\infty)})
	\), so
	\[
		F_V^0(x_\alpha^\beta(p, q))
		= x_\alpha\bigl(
			f_{\beta, s_\beta(\alpha)}
				^\alpha
				(p, q)
		\bigr).
	\]

	Let us check that
	\[
		[x_\alpha^\beta(p, q), x_\gamma(r)]
		= \prod_\delta
			x_\delta\bigl(
				f_{\alpha, \gamma}^\delta\bigl(
					f_{\beta, s_\beta(\alpha)}
						^\alpha
						(p, q),
					r
				\bigr)
			\bigr)
	\]
	if
	\(
		\delta \in \Phi \setminus V
	\). This easily follows by evaluating
	\[
		\up{
			[x_\beta(p), x_{s_\beta(\alpha)}(q)]
		}{x_\gamma(r)}\,
		x_\gamma(\dotminus r)
	\]
	in two ways as in the proof of lemma \ref{vacuous-rel} if
	\(
		V
		\cap (
			\mathbb R \alpha
			+ \mathbb R \beta
			+ \mathbb R \gamma
		)
	\) is one-dimensional and the roots \(\alpha\), \(\beta\), \(\gamma\) are linearly dependent. Otherwise choose by lemma \ref{small-subsp}(2) or \ref{small-subsp}(5) a neighbor \(\delta\) of \(\gamma\) such that
	\(
		\delta
		\notin \mathbb R \alpha
		+ \mathbb R \beta
		+ \mathbb R \gamma
	\) and
	\[
		(
			\mathbb R \alpha
			+ \mathbb R \beta
			+ \mathbb R \gamma
			+ \mathbb R \delta
		)
		\cap V
		= (
			\mathbb R \alpha
			+ \mathbb R \beta
			+ \mathbb R \gamma
		)
		\cap V.
	\]
	Evaluating
	\[
		\up{
			[
				x_\delta(r),
				x_{s_\delta(\gamma)}(s)
			]
		}{x_\alpha^\beta(p, q)}\,
		x_\alpha^\beta(p, q)^{-1}
	\]
	in two ways using the known commutator relations for roots from
	\(
		\mathbb R_{\geq 0} \alpha
		+ \mathbb R_{\geq 0} \gamma
		+ \mathbb R_{\geq 0} \delta
	\) we get the required identity by lemmas \ref{root-units}(5) and \ref{power-idem}.

	We are ready to prove that
	\(
		x_\alpha^\beta(p, q \dotplus q')
		= x_\alpha^\beta(p, q)\,
		x_\alpha^\beta(p, q')
	\). Indeed, this easily follows from expanding both sides of
	\[
		\bigl[
			x_\beta(p),
			x_{s_\beta(\alpha)}(q \dotplus q')
		\bigr]
		= \bigl[
			x_\beta(p),
			x_{s_\beta(\alpha)}(q)
		\bigr]\,
		\up{x_{s_\beta(\alpha)}(q)}{
			\bigl[
				x_\beta(p),
				x_{s_\beta(\alpha)}(q')
			\bigr]
		}
	\]
	using the known commutator formulae.

	Now choose using lemma \ref{small-subsp}(1, 2) a rank \(3\) irreducible root subsystem
	\(
		\alpha \in \Psi \subseteq \Phi
	\) such that
	\(
		\mathbb R \Psi \cap V = \mathbb R \alpha
	\) (such \(\Psi\) may be chosen to contain any given
	\(
		\beta \in \Phi \setminus V
	\) neighbor to \(\alpha\)). Take two neighbors
	\(
		\beta, \gamma \in \Psi
	\) of \(\alpha\) such that they are also neighbors of each other from the following table up to the action of the Weyl group (namely, if \(\Psi\) is not of the type
	\(
		\mathsf A_3
	\) we require that \(\alpha\) or \(\gamma\) is not in
	\(
		\mathsf A_3 \subseteq \Psi
	\)).
	\begin{center}
		\begin{tabular}{|c|c|c|c|}
		\hline
			\(\Psi\)
			& \(\alpha\)
			& \(\beta\)
			& \(\gamma\)
		\\ \hline \hline
			\(\mathsf A_3\)
			& \(\e_1 + \e_2\)
			& \(\e_1 + \e_3\)
			& \(\e_2 + \e_3\)
		\\ \hline
			not \(\mathsf A_3\)
			& \(\e_1\)
			& \(\e_1 + \e_2\)
			& \(\e_1 + \e_3\)
		\\ \hline
			not \(\mathsf A_3\)
			& \(\e_1 + \e_2\)
			& \(\e_1 + \e_3\)
			& \(\e_1\)
		\\ \hline
		\end{tabular}
	\end{center}
	Note that in all cases \(\gamma\) and
	\(
		s_\beta(\alpha)
	\) are orthogonal and they span a reducible saturated root subsystem (of type
	\(
		2 \mathsf A_1
	\) or
	\(
		\mathsf A_1 + \mathsf{BC}_1
	\)), so the corresponding root morphisms commute. Also,
	\(
		4 \frac{\alpha \cdot \beta}
			{\beta \cdot \beta}\,
		\frac{\beta \cdot \gamma}
			{\gamma \cdot \gamma}
		= 2 \frac{\alpha \cdot \gamma}
			{\gamma \cdot \gamma}
	\). Evaluating the expression
	\[
		\bigl[
			\up{x_\gamma(p)}{x_{s_\gamma(\beta)}(q)},
			\up{x_\gamma(p)}{x_{s_\beta(\alpha)}(r)}
		\bigr]
	\]
	in two ways as in the proof of lemma \ref{vacuous-rel} using the known commutator relations we obtain
	\[
		x_\alpha^\gamma\bigl(
			p,
			f_{s_\gamma(\beta), s_\beta(\alpha)}
				^{s_\gamma(\alpha)}
				(q, r)
		\bigr)
		= x_\alpha^\beta\bigl(
			f_{\gamma, s_\gamma(\beta)}^\beta(p, q),
			r
		\bigr).
	\]
	The image of this identity in
	\(
		\stlin_{G, T, \Phi}(\R^{(s^\infty)})
	\) also gives
	\[
		f_{\gamma, s_\gamma(\alpha)}^\alpha\bigl(
			p,
			f_{s_\gamma(\beta), s_\beta(\alpha)}
				^{s_\gamma(\alpha)}
				(q, r)
		\bigr)
		= f_{\beta, s_\beta(\alpha)}^\alpha\bigl(
			f_{\gamma, s_\gamma(\beta)}
				^\beta
				(p, q),
			r
		\bigr).
	\]
	It follows that
	\begin{align*}
		x_\alpha^\gamma\bigl(
			e \cdot_\gamma a b,
			f_{s_\gamma(\beta), s_\beta(\alpha)}
				^{s_\gamma(\alpha)}
				(q, r)
		\bigr)
		&= x_\alpha^\beta\bigl(
			f_{\gamma, s_\gamma(\beta)}
				^\beta
				(e \cdot_\gamma a b, q),
			r
		\bigr)
		= x_\alpha^\beta\bigl(
			f_{\gamma, s_\gamma(\beta)}
				^\beta
				(
					e \cdot_\gamma a,
					q \cdot_{s_\gamma(\beta)} b^{
						2 \frac{\beta \cdot \gamma}
							{\gamma \cdot \gamma}
					}
				),
			r
		\bigr)
	\\
		&= x_\alpha^\gamma\bigl(
			e \cdot_\gamma a,
			f_{s_\gamma(\beta), s_\beta(\alpha)}
				^{s_\gamma(\alpha)}
				(q, r) \cdot_{s_\gamma(\alpha)} b^{
					2 \frac{\alpha \cdot \gamma}
					{\gamma \cdot \gamma}
				}
		\bigr),
	\end{align*}
	for any
	\(
		e \in P_\gamma(K_s)^*
	\), so
	\[
		x_\alpha^\gamma(e \cdot_\gamma a b, p)
		= x_\alpha^\gamma\bigl(
			e \cdot_\gamma a,
			p \cdot_{s_\gamma(\alpha)} b^{
				2 \frac{\alpha \cdot \gamma}
				{\gamma \cdot \gamma}
			}
		\bigr)
	\]
	by lemmas \ref{root-units}(5) and \ref{power-idem}. Finally, the same lemmas imply that there is a unique homomorphism
	\[
		x_\alpha
		\colon P_\alpha(\R^{(s^\infty)})
		\to \stlin_{G, T, \Psi \setminus V}(
			\R^{(s^\infty)}
		)
	\]
	such that
	\(
		x_\alpha^\gamma(e \cdot_\gamma a, q)
		= x_\alpha\bigl(
			f_{\gamma, s_\gamma(\alpha)}
				^\alpha
				(e \cdot_\gamma a, q)
		\bigr)
	\).

	We still have to check that
	\(
		x_\alpha
	\) is independent on the choices of \(\Psi\), \(\gamma\), and \(e\); and that
	\(
		x_\alpha^\beta(p, q)
		= x_\alpha\bigl(
			f_{\beta, s_\beta(\alpha)}^\alpha(p, q)
		\bigr)
	\) for all neighbors \(\beta\) of \(\alpha\) (along with the remaining commutator formulae). For the roots
	\(
		\beta, \gamma \in \Psi
	\) from above we have
	\[
		x_\alpha\bigl(
			f_{\beta, s_\beta(\alpha)}^\alpha\bigl(
				f_{\gamma, s_\gamma(\beta)}
					^\beta
					(e \cdot_\gamma a, q),
				r
			\bigr)
		\bigr)
		= x_\alpha\bigl(
			f_{\gamma, s_\gamma(\alpha)}
				^\alpha
				\bigl(
					e \cdot_\gamma a,
					f_{
						s_\gamma(\beta),
						s_\beta(\alpha)
					}^{s_\gamma(\alpha)}(q, r)
				\bigr)
		\bigr)
		= x_\alpha^\beta\bigl(
			f_{\gamma, s_\gamma(\beta)}
				^\beta
				(e \cdot_\gamma a, q),
			r
		\bigr),
	\]
	so lemmas \ref{root-units}(5) and \ref{power-idem} implies
	\(
		x_\alpha^\beta(p, q)
		= x_\alpha\bigl(
			f_{\beta, s_\beta(\alpha)}^\alpha(p, q)
		\bigr)
	\). Using similar calculations in another direction, we have
	\[
		x_\alpha^\gamma\bigl(
			p,
			f_{s_\gamma(\beta), s_\beta(\alpha)}
				^{s_\gamma(\alpha)}
				(q, r)
		\bigr)
		= x_\alpha\bigl(
			f_{\beta, s_\beta(\alpha)}
				^\alpha
				\bigl(
					f_{\gamma, s_\gamma(\beta)}
						^\beta
						(p, q),
					r
				\bigr)
		\bigr)
		= x_\alpha\bigl(
			f_{\gamma, s_\gamma(\alpha)}^\alpha\bigl(
				p,
				f_{s_\gamma(\beta), s_\beta(\alpha)}
					^{s_\gamma(\alpha)}
					(q, r)
			\bigr)
		\bigr),
	\]
	so
	\(
		x_\alpha^\gamma(p, q)
		= x_\alpha\bigl(
			f_{\gamma, s_\gamma(\alpha)}^\alpha(p, q)
		\bigr)
	\) by the same lemmas. By considering all \(4\) or \(6\) neighbors of \(\alpha\) in \(\Psi\) and using the skew-commutativity
	\(
		f_{\mu, \nu}^\lambda(p, q)
		= \dotminus f_{\nu, \mu}^\lambda(q, p)
	\) we obtain the same identity
	\(
		x_\alpha^\delta(p, q)
		= x_\alpha\bigl(
			f_{\delta, s_\delta(\alpha)}^\alpha(p, q)
		\bigr)
	\) for all neighbors \(\delta\) of \(\alpha\) in \(\Psi\). Now it is clear that
	\(
		x_\alpha
	\) is independent on \(\gamma\) and \(e\). It also does not depend on \(\Psi\) by lemma \ref{small-subsp}(4). Clearly, if \(\alpha\) is ultrashort in a component of \(\Phi\) of type
	\(
		\mathsf{BC}_\ell
	\), then
	\(
		x_{2 \alpha}(p) = x_\alpha(p)
	\) for
	\(
		p \colon P_\alpha(\R^{(s^\infty)})
	\).

	Now we have all root morphisms. Let us finish the remaining cases of the commutator formula for
	\(
		[x_\alpha(p), x_\beta(q)]
	\) if
	\(
		\alpha, \beta \in \Phi \setminus V
	\) and
	\(
		(
			\mathbb R_{> 0} \alpha
			+ \mathbb R_{> 0} \beta
		)
		\cap \Phi
		\cap V
	\) is non-trivial. The only remaining case is when
	\(
		\alpha \perp \beta
	\). Choose a rank \(3\) saturated root subsystem
	\(
		\alpha, \beta \in \Psi \subseteq \Phi
	\) such that
	\(
		\mathbb R \Psi \cap V
	\) is one-dimensional using lemma \ref{small-subsp}(2). Take a root \(\gamma\) neighbor to \(\alpha\) as in the table from the proof of lemma \ref{vacuous-rel}. Evaluating
	\[
		\bigl[
			\up{x_\gamma(p)}{x_{s_\gamma(\alpha)}(q)},
			\up{x_\gamma(p)}{x_{s_\gamma(\beta)}(r)}
		\bigr]
	\]
	in two ways and applying lemmas \ref{root-units}(5) and \ref{power-idem} we get the required identity.

	By the above and lemma \ref{vacuous-rel} we have all commutator relations except the ones for
	\(
		[x_\alpha(p), x_\beta(q)]
	\), where
	\(
		\alpha, \beta \in \Phi \cap V
	\) are linearly independent. In this case choose
	\(
		\gamma \in \Phi \setminus V
	\) neighbor to \(\alpha\) by lemma \ref{small-subsp}(3). The required identity follows by induction on
	\(
		\angle(\alpha, \beta)
	\) evaluating
	\[
		\up{[x_\gamma(p), x_{s_\gamma(\alpha)}(q)]}
			{x_\beta(r)}\,
		x_\beta(\dotminus r)
	\]
	in two ways and applying lemmas \ref{root-units}(5) and \ref{power-idem}.

	To sum up, we constructed a homomorphism
	\(
		\stlin_{G, T, \Phi}(\R^{(s^\infty)})
		\to \stlin_{G, T, \Phi \setminus V}(
			\R^{(s^\infty)}
		)
	\) left inverse to
	\(
		F_V^0
	\). Since
	\(
		F_V^0
	\) is also an epimorphism, it is invertible.
\end{proof}

\section{Crossed module structure}

In this section we are going to prove that the composite homomorphism
\(
	\stlin_{G, T, \Phi}(\R^{(s^\infty)})
	\to G(\R_s^{\Ind})
\) is a crossed module in a unique way.

\subsection{Perfect group objects}

Recall that a group object \(X\) in an infinitary pretopos is \textit{perfect} if it is generated by the commutator morphism
\(
	[{-}, {=}] \colon X \times X \to X
\).

\begin{lemma} \label{perf-x-mod}
	Let
	\(
		\mathbf C
	\) be an infinitary pretopos.
	\begin{enumerate}

		\item
		For any crossed module
		\(
			\delta \colon X \to G
		\) in
		\(
			\mathbf C
		\) the kernel
		\(
			\Ker(\delta) \leq X
		\) is central, i.e. the commutator morphism
		\(
			X \times \Ker(\delta) \to X
		\) is trivial, and the image
		\(
			\Image(\delta) \leq G
		\) is normal, i.e. the conjugation morphism
		\(
			G \times \Image(\delta) \to G
		\) takes values in
		\(
			\Image(\delta)
		\).

		\item
		If
		\(
			f \colon X \to G
		\) is a perfect central extension, then there is a unique morphism
		\(
			\langle {-}, {=} \rangle
			\colon G \times G
			\to X
		\) such that
		\(
			\bigl\langle f(x), f(y) \bigr\rangle
			= [x, y]
		\) and
		\(
			f\bigl(\langle g, h \rangle\bigr) = [g, h]
		\).

		\item
		Let
		\(
			\delta \colon X \to G
		\) and
		\(
			\delta' \colon X' \to G'
		\) be crossed modules, \(P\) be an object in
		\(
			\mathbf C
		\). Suppose that
		\(
			u \colon P \times X \to X'
		\) and
		\(
			v \colon P \times G \to G'
		\) are morphisms multiplicative on their second arguments,
		\(
			\delta'(u(p, x)) = v(p, \delta(x))
		\), and \(X\) is perfect. Then
		\(
			u(p, \up g x) = \up{v(p, g)}{u(p, x)}
		\). In particular, the action of \(G\) on \(X\) making \(\delta\) a crossed module is unique.

		\item
		Let
		\(
			\delta \colon X \to G
		\) and
		\(
			\delta' \colon X' \to G
		\) be crossed modules such that \(X\) is perfect. Then there is at most one homomorphism
		\(
			f \colon X \to X'
		\) such that
		\(
			\delta' \circ f = \delta
		\), it is necessarily \(G\)-equivariant.

	\end{enumerate}
\end{lemma}
\begin{proof}
	The first two claims are easy. To prove (3) let
	\[
		\varepsilon(p, g, x)
		= \up{v(p, g)}{u(p, x)}\, u(p, \up g x)^{-1}.
	\]
	Then
	\[
		\delta'(\varepsilon(p, g, x))
		= \up{v(p, g)}{v(p, \delta(x))}\,
			v(p, \up g {\delta(x)})^{-1}
		= 1,
	\]
	so
	\(
		\varepsilon(p, g, x)
	\) is central,
	\[
		\varepsilon(p, g, x y)
		= \up{v(p, g)}{u(p, x)}\,
			\varepsilon(p, g, y)\,
			u(p, \up g x)^{-1}
		= \varepsilon(p, g, x)\,
			\varepsilon(p, g, y),
	\]
	and
	\[
		\varepsilon(p, g, [x, y])
		= [\varepsilon(p, g, x), \varepsilon(p, g, y)]
		= 1.
	\]
	Since \(X\) is perfect, \(\varepsilon\) is trivial by lemma \ref{gen-rel}(3) applied to the generating morphism
	\(
		[{-}, {=}] \colon X \times X \to X
	\).

	Finally, let us prove (4). Take such a homomorphism
	\(
		f \colon X \to X'
	\), then
	\(
		\Image(\delta) \leq \Image(\delta')
	\). We have
	\[
		f([x, y])
		= [f(x), f(y)]
		= \bigl\langle
			\delta'(f(x)),
			\delta'(f(y))
		\bigr \rangle
		= \bigl\langle
			\delta(x),
			\delta(y)
		\bigr\rangle
	\]
	using (2), so \(f\) is unique by lemma \ref{gen-rel}(3). The \(G\)-equivariance follows from (3).
\end{proof}

\begin{lemma} \label{st-perf}
	The Steinberg group
	\(
		\stlin_{G, T, \Phi}(\R^{(s^\infty)})
	\) is perfect.
\end{lemma}
\begin{proof}
	Take a root
	\(
		\alpha \in \Phi
	\). If \(\alpha\) lies in a saturated root subsystem
	\(
		\alpha \in \Psi \subseteq \Phi
	\) of type
	\(
		\mathsf A_2
	\), then take its neighbor
	\(
		\beta \in \Psi
	\) and note that
	\[
		[{-}, {=}]
		\colon U_\beta(\R^{(s^\infty)})
			\times U_{s_\beta(\alpha)}(\R^{(s^\infty)})
		\to U_\alpha(\R^{(s^\infty)})
	\] is an epimorphism by lemmas \ref{root-units}(5) and \ref{power-idem}.

	Now suppose that such \(\Psi\) does not exist, i.e. \(\alpha\) lies in a component of type
	\(
		\mathsf B_\ell
	\),
	\(
		\mathsf C_\ell
	\), or
	\(
		\mathsf{BC}_\ell
	\) and not in its root subsystem of type
	\(
		\mathsf D_\ell
	\). Take its neighbor \(\beta\) lying the root subsystem of type
	\(
		\mathsf D_\ell
	\) and let
	\(
		\gamma = s_\beta(\alpha)
	\). Since
	\(
		s_\gamma(\beta)
	\) also lies in the root subsystem of type
	\(
		\mathsf D_\ell
	\), the root subgroup
	\(
		U_{s_\gamma(\beta)}(\R^{(s^\infty)})
	\) is contained in the derived subgroup of
	\(
		\stlin_{G, T, \Phi}(\R^{(s^\infty)})
	\) by the proved case. Also, the morphism
	\[
		[\![
			[x, y] z
		]\!]
		\colon U_\beta(\R^{(s^\infty)})
			\times U_\gamma(\R^{(s^\infty)})
			\times U_{s_\gamma(\beta)}(\R^{(s^\infty)})
		\to U_{\{\alpha, s_\gamma(\beta)\}}(
			\R^{(s^\infty)}
		)
	\]
	is an epimorphism again by lemmas \ref{root-units}(5) and \ref{power-idem}.
\end{proof}

\subsection{Extended Steinberg groups}

Consider the \textit{extended Steinberg group}
\(
	\gstlin_{G, T, \Phi}(\R_s^{\Ind})
\), it is the group object in
\(
	\mathbf U_K
\) with generators
\(
	x_\alpha(p)
\) for
\(
	p \colon P_\alpha(\R_s^{\Ind})
\),
\(
	\alpha \in \Phi
\) and
\(
	d(g)
\) for
\(
	g \colon L(\R_s^{\Ind})
\). The relations are
\begin{itemize}

	\item
	\(
		x_\alpha(p)\, x_\alpha(q)
		= x_\alpha(p \dotplus q)
	\);

	\item
	\(
		d(g)\, d(h) = d(g h)
	\);

	\item
	\(
		x_{2 \alpha}(p) = x_\alpha(p)
	\) for ultrashort \(\alpha\) in a component of type
	\(
		\mathsf{BC}_\ell
	\) and
	\(
		p \colon P_{2 \alpha}(\R_s^{\Ind})
	\);

	\item
	\(
		\up{d(g)}{x_\alpha(p)} = x_\alpha(\up g p)
	\), where
	\(
		\up g p \colon P_\alpha(\R_s^{\Ind})
	\) is given by the scheme morphism such that
	\(
		\up g {t_\alpha(p)} = t_\alpha(\up g p)
	\) for
	\(
		g \in L(R)
	\),
	\(
		p \in P_\alpha(R)
	\),
	\(
		R \in \Ring_{K_s}
	\);

	\item
	\(
		[x_\alpha(p), x_\beta(q)]
		= \prod_\gamma
			x_\gamma(f_{\alpha, \beta}^\gamma(p, q))
	\) for linearly independent \(\alpha\) and \(\beta\);

	\item
	\(
		t_\alpha(p)\,
			t_{-\alpha}(q)\,
			t_\alpha(r)
		= t_{-\alpha}(p')\,
			t_\alpha(q')\,
			t_{-\alpha}(r')\,
			d(g)
		\vdash x_\alpha(p)\,
			x_{-\alpha}(q)\,
			x_\alpha(r)
		= x_{-\alpha}(p')\,
			x_\alpha(q')\,
			x_{-\alpha}(r')\,
			d(g)
	\).

\end{itemize}
This group object is preserved under base changes and functorial on \(s\) in the same way as
\(
	\R_s^{\Ind}
\). Note that there are well-defined morphisms
\[
	\up{(-)}{(=)}
	\colon L(\R_s^{\Ind})
		\times P_\alpha(\R^{(s^\infty)})
	\to P_\alpha(\R^{(s^\infty)}),
\quad
	f_{\alpha \beta}^\gamma
	\colon P_\alpha(\R_s^{\Ind})
		\times P_\beta(\R^{(s^\infty)})
	\to P_\gamma(\R^{(s^\infty)})
\]
induced by the actions of
\(
	(P_\alpha \rtimes L)(\R_s^{\Ind})
\) on
\(
	(P_\alpha \rtimes L)(\R^{(s^\infty)})
\) and of
\(
	U_\Sigma(\R_s^{\Ind})
\) on
\(
	U_\Sigma(\R^{(s^\infty)})
\) for
\(
	\Sigma
	= (
		\mathbb R_{\geq 0} \alpha
		+ \mathbb R_{\geq 0} \beta
	)
	\cap \Phi
\), where \(\alpha\) and \(\beta\) are not anti-parallel. These morphisms are also preserved under base changes and extra-natural on \(s\).

Informally, we should look at
\(
	G(\R_s^{\Ind})
\) as a ``sheafification'' of
\(
	\gstlin_{G, T, \Phi}(\R_s^{\Ind})
\) in Zariski topology. Namely, consider the abstract group
\(
	\gstlin_{G, T, \Phi}(R_s)
	= \ev_R(\gstlin_{G, T, \Phi}(\R_s^{\Ind}))
\) given by the same generators and relations in \(\Set\) for a unital algebra \(R\). We have the following result.

\begin{lemma} \label{gst-local}
	Suppose that \(K\) is a local ring, \(G\) is a reductive group scheme over \(K\), and
	\(
		(T, \Phi)
	\) is an isotropic pinning on \(G\) of rank at least \(3\). Then the canonical homomorphism
	\(
		\stmap \colon \gstlin_{G, T, \Phi}(K) \to G(K)
	\) is one-to-one.
\end{lemma}
\begin{proof}
	Fix a maximal isotropic pinning
	\(
		(T, \Phi) \subseteq (T', \Phi')
	\) with corresponding subgroups
	\(
		U'_\alpha, L' \leq G
	\), so
	\[
		U_\alpha(K)
		= \prod_{\Image(\beta) = \alpha}
			U'_\beta(K),
	\quad
		\langle
			L'(K), U'_\beta(K)
			\mid
			\Image(\beta) = 0
		\rangle
		\leq L(K).
	\]
	By the Gauss decomposition of \(G(K)\) with respect to \((T', \Phi')\) we have
	\[
		G(K)
		= U'_{\Pi'}(K)\,
			U'_{-\Pi'}(K)\,
			U'_{\Pi'}(K)\,
			L'(K)
		= U_\Pi(K)\, U_{-\Pi}(K)\, U_\Pi(K)\, L(K),
	\]
	where
	\(
		\Pi \subseteq \Phi
	\) and
	\(
		\Pi' \subseteq \Phi'
	\) are subsets of positive roots such that
	\(
		\Image(\Pi') \subseteq \Pi \cup \{0\}
	\), so \(\stmap\) is surjective.

	Let us prove that also
	\[
		\gstlin_{G, T, \Phi}(K)
		= U_\Pi(K)\, U_{-\Pi}(K)\, U_\Pi(K)\, d(L(K)).
	\]
	The right hand side \(X\) of this equality is closed under multiplications by
	\(
		x_\alpha(a)
	\) and
	\(
		d(g)
	\) for positive \(\alpha\). Let us check that
	\[
		X
		= U_{s_\alpha(\Pi)}(K)\,
			U_{-s_\alpha(\Pi)}(K)\,
			U_{s_\alpha(\Pi)}(K)\,
			d(L(K))
	\]
	for any basic root \(\alpha\). Indeed, by the commutator formulae we know that
	\[
		X
		= U_{\Pi \setminus \mathbb R \alpha}(K)\,
			U_{\Pi \setminus \mathbb R \alpha}(K)\,
			U_{\Pi \setminus \mathbb R \alpha}(K)\,
			U_\alpha(K)\,
			U_{-\alpha}(K)\,
			U_\alpha(K)\,
			d(L(K)).
	\]
	Recall that the fppf sheaf
	\(
		\langle U_{-\alpha}, L, U_\alpha \rangle
		\leq G
	\) is a reductive subgroup with an isotropic pinning
	\(
		(T, \Phi \cap \mathbb Z \alpha)
	\) by \cite[XXII, proposition 5.10.1]{sga3}. The last relation in
	\(
		\gstlin_{G, T, \Phi}(K)
	\) and the Gauss decomposition applied to
	\(
		\langle U_{-\alpha}, L, U_\alpha \rangle
	\) imply that
	\[
		U_\alpha(K)\,
			U_{-\alpha}(K)\,
			U_\alpha(K)\,
			d(L(K))
		= U_{-\alpha}(K)\,
			U_\alpha(K)\,
			U_{-\alpha}(K)\,
			d(L(K)).
	\]
	It follows that \(X\) is independent on the choice of \(\Pi\), so it is closed under multiplications by all root subgroups from the left and
	\(
		X = \gstlin_{G, T, \Phi}(K)
	\).

	Now take any
	\(
		g \in \Ker(\stmap)
	\) and write it as
	\(
		g = g_1 g_2 d g_3
	\) for
	\(
		g_1, g_3 \in U_\Pi(K)
	\),
	\(
		g_2 \in U_{-\Pi}(K)
	\), and
	\(
		d \in d(L(K))
	\). Since
	\(
		\up {g_3} g = g_3 g_1 g_2 d
	\) also lies in the kernel and the multiplication map
	\(
		U_\Pi(K) \times U_{-\Pi}(K) \times L(K)
		\to G(K)
	\) is injective, we have
	\(
		g_3 = g_1^{-1}
	\),
	\(
		g_2 = 1
	\), and
	\(
		d = 1
	\), i.e.
	\(
		g = 1
	\).
\end{proof}

There is the following commutative diagram of group objects and homomorphisms denoted by solid arrows.
\[\xymatrix@R=45pt@C=120pt@M=6pt@!0{
	\stlin_{G, T, \Phi}(\R^{(s^\infty)})
	\ar^\delta[r]
	\ar^{\stmap}[d]
&
	\gstlin_{G, T, \Phi}(\R_s^{\Ind})
	\ar^{\stmap}[d]
\\
	G(\R^{(s^\infty)})
	\ar_\delta[r]
&
	G(\R_s^{\Ind})
	\ar@/_12pt/@{.>}[ul]
	\ar@/_12pt/@{.>}[l]
}\]

The dotted arrows denote natural actions by automorphisms, the diagonal one is constructed below in proposition \ref{local-action}. They have the following properties.
\begin{itemize}

	\item
	The homomorphisms
	\begin{align*}
		\delta
		&\colon \stlin_{G, T, \Phi}(\R^{(s^\infty)})
		\to \gstlin_{G, T, \Phi}(\R_s^{\Ind}),
	\\
		\delta \circ \stmap
		&\colon \stlin_{G, T, \Phi}(\R^{(s^\infty)})
		\to G(\R_s^{\Ind}),
	\\
		\delta
		&\colon G(\R^{(s^\infty)}) \to G(\R_s^{\Ind})
	\end{align*}
	are crossed modules, where the action of
	\(
		\gstlin_{G, T, \Phi}(\R_s^{\Ind})
	\) on
	\(
		\stlin_{G, T, \Phi}(\R^{(s^\infty)})
	\) is induced from the action of
	\(
		G(\R_s^{\Ind})
	\).

	\item
	The homomorphism
	\(
		\stmap
		\colon \stlin_{G, T, \Phi}(\R^{(s^\infty)})
		\to G(\R^{(s^\infty)})
	\) is
	\(
		G(\R_s^{\Ind})
	\)-equivariant.

	\item
	Both actions are preserved under base changes. They are also extra-natural on \(s\), i.e. satisfy the identities
	\[
		\up g {\bigl(t^*(x)\bigr)}
		= t^*\bigl(\up {t_*(g)} x\bigr).
	\]

\end{itemize}
The claim for
\(
	\delta \colon G(\R^{(s^\infty)}) \to G(\R_s^{\Ind})
\) is already known. The action of
\(
	G(\R_s^{\Ind})
\) on
\(
	\stlin_{G, T, \Phi}(\R^{(s^\infty)})
\) is constructed in proposition \ref{local-action} below, where we also prove the claimed identities.

\begin{lemma} \label{gst-action}
	The homomorphism
	\(
		\delta
		\colon \stlin_{G, T, \Phi}(\R^{(s^\infty)})
		\to \gstlin_{G, T, \Phi}(\R_s^{\Ind})
	\) is a crossed module in a unique way. The action is given by
	\(
		\up{d(g)}{x_\alpha(p)} = x_\alpha(\up g p)
	\) for
	\(
		g \colon L(\R_s^{\Ind})
	\),
	\(
		p \colon P_\alpha(\R^{(s^\infty)})
	\) and
	\[
		\up{x_\alpha(p)}{x_\beta(q)}\,
			x_\beta(\dotminus q)
		= \prod_\gamma
			x_\gamma(f_{\alpha, \beta}^\gamma(p, q))
	\]
	for
	\(
		p \colon P_\alpha(\R_s^{\Ind})
	\),
	\(
		q \colon P_\beta(\R^{(s^\infty)})
	\) if \(\alpha\) and \(\beta\) are not anti-parallel. The action is preserved under base changes, it is extra-natural on \(s\). Moreover,
	\(
		\stmap(\up g x) = \up{\stmap(g)}{\stmap(x)}
	\) for
	\(
		x \colon \stlin_{G, T, \Phi}(\R^{(s^\infty)})
	\) and
	\(
		g \colon \gstlin_{G, T, \Phi}(\R_s^{\Ind})
	\), i.e. the vertical arrows from the diagram above are a homomorphism of crossed modules.
\end{lemma}
\begin{proof}
	Lemmas \ref{perf-x-mod} and \ref{st-perf} imply that the action is unique, is preserved under base changes, and is preserved under \(\stmap\) if the action exists at all.

	The formulae from the statement induce unique actions of the objects
	\(
		U_\alpha(\R_s^{\Ind})
	\) and
	\(
		L(\R_s^{\Ind})
	\) on
	\(
		\stlin_{G, T, \Phi}(\R^{(s^\infty)})
	\) by endomorphisms by lemma \ref{gen-rel}(3) and proposition \ref{elim-bij} applied to
	\(
		V = \mathbb R \alpha
	\). Indeed, the necessary identities involving
	\(
		f_{\beta \gamma}^\delta
	\) and the morphisms
	\(
		\up{(-)}{(=)}
		\colon L \times P_\beta
		\to P_\beta
	\) trivially hold inside the group scheme
	\(
		G_{K_s}
	\). By the same argument,
	\[
		\up {x_\alpha(p \dotplus q)} g
		= \up{x_\alpha(p)}{(\up {x_\alpha(q)} g)},
	\quad
		\up {d(h_1 h_2)} g
		= \up {d_1(h_1)} {(\up {d(h_2)} g)},
	\]
	so in particular the objects act by automorphisms.

	Now the required action of
	\(
		\stlin_{G, T, \Phi}(\R_s^{\Ind})
	\) on
	\(
		\stlin_{G, T, \Phi}(\R^{(s^\infty)})
	\) may be constructed by lemmas \ref{gen-rel}(4) and \ref{elim-sur} applied to
	\(
		V = 0
	\),
	\(
		V = \mathbb R \alpha
	\), or
	\(
		V = \mathbb R \alpha + \mathbb R \beta
	\) for linearly independent
	\(
		\alpha, \beta \in \Phi
	\). It is easy to see using lemmas \ref{gen-rel} and \ref{elim-sur} that this action makes \(\delta\) a crossed module.

	Finally, the extra-naturality easily follows from the explicit formulae describing the actions and lemma \ref{elim-sur} since it suffice to check the claim for \(g\) from generating objects.
\end{proof}

\subsection{Cosheaf properties}

\begin{lemma} \label{ring-cosheaf}
	Let
	\(
		s \in K
	\) be any element and
	\(
		(M, M_0)
	\) be a \(2\)-step nilpotent \(K_s\)-module such that \(M_0\) and
	\(
		M / M_0
	\) are free \(K_s\)-modules of finite rank. Let also
	\(
		t_1, \ldots, t_n \in K
	\) be such that
	\(
		\sum_{i = 1}^n K_s t_i = K_s
	\). Then the morphisms
	\[
		t_i^*
		\colon \mathbb W(M)(\R^{((t_i s)^\infty)})
		\to \mathbb W(M)(\R^{(s^\infty)})
	\]
	generate
	\(
		\mathbb W(M)(\R^{(s^\infty)})
	\) and the only relations between them are
	\begin{itemize}

		\item
		\(
			t_i^*(p \dotplus q)
			= t_i^*(p) \dotplus t^*(q)
		\),

		\item
		\(
			t_i^* \circ t_j^*
			= t_j^* \circ t_i^*
			\colon \mathbb W(M)(
				\R^{((t_i t_j s)^\infty)}
			)
			\to \mathbb W(M)(\R^{(s^\infty)})
		\) for
		\(
			i < j
		\) (so
		\(
			(t_i t_j)^*
		\) is well-defined as any of these compositions),

		\item
		\(
			[t_i^*(p), t_j^*(q)]
			= (t_i t_j)^*([p, q]^\cdot)
		\) for
		\(
			i < j
		\), where the commutator morphism in the right hand side is induced by the commutator morphism of \(M\) as a bilinear map
		\(
			(M / M_0) \times (M / M_0) \to M_0
		\) and by the multiplication morphism
		\[
			\R^{((t_i s)^\infty)}
				\times \R^{((t_j s)^\infty)}
			\to \R^{((t_i t_j s)^\infty)}
		\]
		given by the maps
		\[
			\bigl(
				p^{((t_i s)^m)},
				q^{((t_j s)^m)}
			\bigr)
			\mapsto (s^m p q)^{((t_i t_j s)^m)}.
		\]

	\end{itemize}
\end{lemma}
\begin{proof}
	Since
	\(
		M \cong K_s^k \dotoplus K_s^l
	\), the generating morphisms are additive and satisfy commutator relations, the claim easily reduces to the case
	\(
		M = K_s
	\). Recall that every
	\(
		t_i^*
	\) is given by the inverse system of maps
	\[
		\R^{(t_i s)^m} \to \R^{(s^m)},\,
		a^{(t_i s)^m} \mapsto (t_i^m a)^{(s^m)}.
	\]
	Consider the abelian group presheaves
	\(
		L_m \in \mathbf P_K
	\) generated by the abstract morphisms
	\(
		\tau_i \colon \R^{(t_i s)^m} \to L_m
	\) with the only relations
	\[
		\tau_i\bigl((p + q)^{((t_i s)^m)}\bigr)
		= \tau_i\bigl(p^{((t_i s)^m)}\bigr)
		+ \tau_i\bigl(q^{((t_i s)^m)}\bigr),
	\quad
			\tau_i\bigl(
				(t_j^m p)^{(t_i s)^m}
			\bigr)
			= \tau_j\bigl(
				(t_i^m p)^{(t_j s)^m}
			\bigr).
	\]
	We are going to prove that the homomorphism
	\(
		c
		\colon \varprojlim_m L_m
		\to \R^{(s^\infty)}
	\) in
	\(
		\Pro(\mathbf P_K)
	\) (given by the inverse limit of the canonical maps
	\(
		c_m \colon L_m \to \R^{(s^m)},\,
		\tau_i\bigl(p^{((t_i s)^m)}\bigr)
		\mapsto (t_i^m p)^{(s^m)}
	\)) is invertible.

	Since \(t_i\) generate the unit ideal
	\(
		K_s
	\), for any
	\(
		m \geq 0
	\) there are
	\(
		N \geq 0
	\) and
	\(
		a_{1m}, \ldots, a_{nm} \in K
	\) such that
	\(
		\sum_{i = 1}^n a_{im} t_i^m = s^N
	\). Consider homomorphisms
	\[
		d_m
		\colon \R^{(s^{m + N})}
		\to L_m,\,
		p^{(s^{m + N})}
		\mapsto \sum_{i = 1}^n
			\tau_i\bigl((a_{im} p)^{((t_i s)^m)}\bigr).
	\]
	Clearly,
	\(
		(c_m \circ d_m)(p^{(s^{m + N})})
		= (s^N p)^{s^m}
	\). On the other hand,
	\[
		d_m\bigl(
			c_{m + N}\bigl(
				\tau_i\bigl(p^{((t_i s)^{m + N})}\bigr)
			\bigr)
		\bigr)
	=
		\sum_{j = 1}^n
			\tau_j\bigl(
				(a_{jm} t_i^{m + N} p)^{((t_j s)^m)}
			\bigr)
	=
		\sum_{j = 1}^n
			\tau_i\bigl(
				(a_{jm} t_j^m t_i^N p)^{((t_i s)^m)}
			\bigr)
	=
		\tau_i\bigl((t_i^N s^N p)^{((t_i s)^m)}\bigr).
	\]
	In other words,
	\(
		c^{-1}
	\) is the inverse limit of \(d_m\).
\end{proof}

The next result is an analogue of \cite[theorems 3, 4]{cosheaves}, i.e. Steinberg groups satisfy a cosheaf property.
\begin{prop} \label{st-cosheaf}
	Let
	\(
		t_1, \ldots, t_n \in K
	\) be such that
	\(
		\sum_{i = 1}^n K_s t_i = K_s
	\). Then the Steinberg group
	\(
		\stlin_{G, T, \Phi}(\R^{(s^\infty)})
	\) is the colimit of the diagram
	\[
		\stlin_{G, T, \Phi}(\R^{((t_i t_j s)^\infty)})
		\rightrightarrows \stlin_{G, T, \Phi}(
			\R^{((t_i s)^\infty)}
		)
	\]
	of crossed modules over
	\(
		\gstlin_{G, T, \Phi}(\R_s^{\Ind})
	\), i.e. it is the universal crossed module with homomorphisms
	\[
		t_i^*
		\colon \stlin_{G, T, \Phi}(
			\R^{((t_i s)^\infty)}
		)
		\to \stlin_{G, T, \Phi}(\R^{(s^\infty)})
	\]
	satisfying
	\(
		t_i^* \circ t_j^* = t_j^* \circ t_i^*
	\). Moreover,
	\(
		\stlin_{G, T, \Phi}(\R^{(s^\infty)})
	\) as a group object is generated by the morphisms
	\(
		t_i^*
	\) satisfying the only relations
	\begin{itemize}

		\item
		\(
			t_i^*(g h) = t_i^*(g)\, t_i^*(h)
		\);

		\item
		\(
			t_i^* \circ t_j^*
			= t_j^* \circ t_i^*
			\colon \stlin_{G, T, \Phi}\bigl(
				\R^{((t_i t_j s)^\infty)}
			\bigr)
			\to \stlin_{G, T, \Phi}(\R^{(s^\infty)})
		\) for
		\(
			i < j
		\);

		\item
		\(
			\up{t_i^*(g)}{t_j^*(h)}
			= t_j^*(\up {\widetilde g} h)
		\) for
		\(
			i \neq j
		\), where
		\(
			\widetilde g
			= \bigl(
				(t_j)_*
				\circ \delta
				\circ t_i^*
			\bigr)(g)
			\colon \gstlin_{G, T, \Phi}(
				\R_{t_j s}^{\Ind}
			)
		\).

	\end{itemize}
\end{prop}
\begin{proof}
	Let us prove the second claim. Indeed, consider the morphisms
	\[
		x_\alpha^i
		= x_\alpha \circ t_i^*
		\colon P_\alpha(\R^{((t_i s)^\infty)})
		\to \stlin_{G, T, \Phi}(\R^{(s^\infty)})
	\]
	and the relations
	\begin{itemize}

		\item
		\(
			x_\alpha^i(p \dotplus q)
			= x_\alpha^i(p)\, x_\alpha^i(q)
		\);

		\item
		\(
			x_{2 \alpha}^i(p) = x_\alpha^i(p)
		\) for
		\(
			p
			\colon P_{2 \alpha}(\R^{((t_i s)^\infty)})
		\) if \(\alpha\) is ultrashort in a component of type
		\(
			\mathsf{BC}_\ell
		\);

		\item
		\(
			x_\alpha^i \circ t_j^*
			= x_\alpha^j \circ t_i^*
			\colon P_\alpha(\R^{((t_i t_j s)^\infty)})
			\to \stlin_{G, T, \Phi}(\R^{(s^\infty)})
		\) for
		\(
			i < j
		\);

		\item
		\(
			[x_\alpha^i(p), x_\beta^j(q)]
			= \prod_\gamma
				x_\gamma^j(
					f_{\alpha \beta}^\gamma(
						\widetilde p,
						q
					)
				)
		\) for
		\(
			\widetilde p
			= \bigl(
				(t_j)_*
				\circ \delta
				\circ t_i^*
			\bigr)(p)
			\colon P_\alpha(\R_{t_j s}^{\Ind})
		\) if \(\alpha\) and \(\beta\) are not anti-parallel.

	\end{itemize}
	These morphisms glue together to
	\(
		t_i^*
		\colon \stlin_{G, T, \Phi}(
			\R^{((t_i s)^\infty)}
		)
		\to \stlin_{G, T, \Phi}(\R^{(s^\infty)})
	\) giving the presentation from the statement. The last relation from the statement follows from the listed relations between
	\(
		x_\alpha^i
	\) by lemmas \ref{gen-rel}(3, 4) and \ref{elim-sur}. On the other hand, the morphisms
	\(
		x_\alpha^i
	\) glue together to the homomorphisms
	\(
		x_\alpha
		\colon P_\alpha(\R^{(s^\infty)})
		\to \stlin_{G, T, \Phi}(\R^{(s^\infty)})
	\) by lemma \ref{ring-cosheaf}. These homomorphisms satisfy the commutator formula since every commutator term
	\[
		\bigl[
			x_\alpha\bigl(
				\sum_i^\cdot
					t_i^*(p_i)
			\bigr),
			x_\beta\bigl(
				\sum_j^\cdot
					t_j^*(q_j)
			\bigr)
		\bigr]
	\]
	may be easily rewritten as a product of root elements with the roots from
	\(
		\{i \alpha + j \beta \in \Phi \mid i, j > 0\}
	\) if \(\alpha\) and \(\beta\) are not anti-parallel.

	Now we prove the first claim. Let
	\(
		\delta_X
		\colon X
		\to \gstlin_{G, T, \Phi}(\R_s^{\Ind})
	\) be a crossed module and
	\(
		f_i
		\colon \stlin_{G, T, \Phi}(
			\R^{((t_i s)^\infty)}
		)
		\to X
	\) be homomorphisms satisfying
	\(
		f_i \circ t_j^*
		= f_j \circ t_i^*
	\). By the (already proved) second claim there exists a unique homomorphism
	\(
		g
		\colon \stlin_{G, T, \Phi}(\R^{(s^\infty)})
		\to X
	\) such that
	\(
		f_i = g \circ t_i^*
	\) and
	\(
		\delta_X \circ g = \delta
	\). It remains to check that \(g\) is
	\(
		\gstlin_{G, T, \Phi}(\R_s^{\Ind})
	\)-equivariant. This follows from lemma \ref{gen-rel}(3) applied to the generating family
	\(
		t_i^*
	\) of
	\(
		\stlin_{G, T, \Phi}(\R^{(s^\infty)})
	\).
\end{proof}

\subsection{Construction of the action}

Recall that the abstract group
\(
	G(K_s)
\) is canonically isomorphic to
\[
	\ev_K(G(\R_s^{\Ind}))
	\cong \Ind(\mathbf P_K)\bigl(
		\Spec(K),
		G(\R_s^{\Ind})
	\bigr).
\]
We consider it as a \textit{discrete} group object in
\(
	\Ind(\mathbf P_K) \subseteq \mathbf U_K
\), i.e. as the coproduct of its one-element subsets. Let also
\(
	t_* \colon G(K_s) \to G(K_{t s})
\) be the canonical homomorphisms. The same remarks are applicable to
\(
	\gstlin_{G, T, \Phi}(K_s)
\) inside
\(
	\Ex(\Ind(\mathbf P_K))
\).

Note that the groups
\[
	G(K_{\mathfrak p})
	= \ev_{K_{\mathfrak p}}\bigl(
		G(\R_s^{\Ind})
	\bigr)
	= \varinjlim_t
		G(K_{t s}),
\quad
	\gstlin_{G, T, \Phi}(K_{\mathfrak p})
	= \ev_{K_{\mathfrak p}}\bigl(
		\gstlin_{G, T, \Phi}(\R_s^{\Ind})
	\bigr)
	= \varinjlim_t
		\gstlin_{G, T, \Phi}(K_{t s})
\]
are isomorphic by lemma \ref{gst-local} for every prime ideal
\(
	s \notin \mathfrak p \leqt K
\).

\begin{lemma} \label{k2-central}
	The kernel of
	\(
		\stmap
		\colon \gstlin_{G, T, \Phi}(\R_s^{\Ind})
		\to G(\R_s^{\Ind})
	\) trivially acts on
	\(
		\stlin_{G, T, \Phi}(\R^{(s^\infty)})
	\), i.e. the action of
	\(
		\gstlin_{G, T, \Phi}(\R_s^{\Ind})
	\) on the co-local Steinberg group from lemma \ref{gst-action} factors through the unique action of
	\(
		\Image(\stmap)
	\).
\end{lemma}
\begin{proof}
	Note that
	\[
		\up g {[x, y]}
		= [\up g x, \up g y]
		= \langle
			\up{\stmap(g)}{\stmap(x)},
			\up{\stmap(g)}{\stmap(y)}
		\rangle
		= \langle \stmap(x), \stmap(y) \rangle
		= [x, y]
	\]
	for
	\(
		x, y
		\colon \stlin_{G, T, \Phi}(\R^{(s^\infty)})
	\) and
	\(
		g
		\colon \Ker(
			\stmap
			\colon \gstlin_{G, T, \Phi}(\R_s^{\Ind})
			\to G(\R_s^{\Ind})
		)
	\) by lemmas \ref{perf-x-mod}(2) and \ref{gst-action}. This implies the claim by lemmas \ref{gen-rel}(3) and \ref{st-perf}.
\end{proof}

\begin{lemma} \label{point-action}
	There is a unique action of
	\(
		G(K_s)
	\) on
	\(
		\stlin_{G, T, \Phi}(\R^{(s^\infty)})
	\) by automorphisms such that this action is extra-natural on \(s\) and induces the action of
	\(
		\gstlin_{G, T, \Phi}(K_s)
	\) from lemma \ref{gst-action}. Moreover, this action is preserved under base changes and
	\(
		\stmap
		\colon \stlin_{G, T, \Phi}(\R^{(s^\infty)})
		\to G(\R^{(s^\infty)})
	\) is
	\(
		G(K_s)
	\)-equivariant.
\end{lemma}
\begin{proof}
	Let us construct the endomorphism
	\(
		A_g
	\) of
	\(
		\stlin_{G, T, \Psi}(\R^{(s^\infty)})
	\) corresponding to a fixed
	\(
		g \in G(K_s)
	\). Choose elements
	\(
		t_1, \ldots, t_n \in K
	\) generating the unit ideal of
	\(
		K_s
	\) and elements
	\[
		\widetilde g_i
		\in \gstlin_{G, T, \Phi}(K_{t_i s})
		= \ev_K(\gstlin_{G, T, \Phi}(
			\R_{t_i s}^{\Ind})
		)
	\]
	such that
	\(
		\stmap(\widetilde g_i)
		= (t_i)_*(g)
		\in G(K_{t_i s})
	\), where
	\(
		(t_i)_* \colon G(K_s) \to G(K_{t_i s})
	\) is the localization homomorphism. To find such elements recall that for every prime ideal
	\(
		s \notin \mathfrak p
	\) the image of \(g\) in
	\(
		G(K_{\mathfrak p})
		= \varinjlim_{t \notin \mathfrak p}
			G(K_{t s})
	\) may be lifted to an element
	\[
		\widetilde g_{\mathfrak p}
		\in \gstlin_{G, T, \Phi}(K_{\mathfrak p})
		= \varinjlim_{t \notin \mathfrak p}
			\gstlin_{G, T, \Phi}(K_{t s})
	\]
	by the Gauss decomposition. Next, the direct limit formulae show that such
	\(
		\widetilde g_t
	\) also exists for a principal open subset
	\(
		\mathfrak p
		\in \mathcal D(t s)
		\subseteq \mathcal D(s)
	\) for some
	\(
		t \notin \mathfrak p
	\). Finitely many such subsets cover
	\(
		\mathcal D(s)
	\) since the latter is quasi-compact.

	We define
	\(
		A_g
	\) on the generators of
	\(
		\stlin_{G, T, \Phi}(\R^{(s^\infty)})
	\) from proposition \ref{st-cosheaf} by the formula
	\[
		A_g\bigl(t_i^*(h)\bigr)
		= t_i^*(\up {\widetilde g_i} h)
	\]
	for
	\(
		h
		\colon \stlin_{G, T, \Phi}(
			\R^{((t_i s)^\infty)}
		)
	\), where
	\(
		\widetilde g_i
	\) acts as a global element of
	\(
		\gstlin_{G, T, \Phi}(K_s^{\Ind})
	\) by lemma \ref{gst-action}. Clearly, such
	\(
		A_g
	\) preserves the first relation
	\(
		t_i^*(h h') = t_i^*(h)\, t_i^*(h')
	\).

	For all \(i\), \(j\) we may further choose elements
	\(
		r_1, \ldots, r_m \in K
	\) generating the unit ideal of
	\(
		K_{t_i t_j s}
	\) such that
	\[
		(t_j r_k)_*(\widetilde g_i)
		= (t_i r_k)_*(\widetilde g_j)
		\in \gstlin_{G, T, \Phi}(K_{r_k t_i t_j s}),
	\]
	they exist by lemma \ref{gst-local} applied to
	\(
		K_{\mathfrak p}
	\) for prime ideals
	\(
		t s \notin \mathfrak p
	\) and by the direct limit argument as above. It follows that
	\[
		A_g \circ t_i^* \circ (t_j r_k)^*
		= A_g \circ t_j^* \circ (t_i r_k)^*,
	\]
	so by proposition \ref{st-cosheaf}
	\(
		A_g
	\) preserves the second relation.

	In order to show that
	\(
		A_g
	\) preserves the last relation it suffices to check that
	\[
		\bigl(
			{\stmap}
			\circ (t_j)_*
			\circ \delta
			\circ t_i^*
		\bigr)(\up {\widetilde g_i} h)
		= \stmap\bigl
			(\up{\widetilde g_j}{\bigl(
				\bigl(
					(t_j)_*
					\circ \delta
					\circ t_i^*
				\bigr)(h)
			\bigr)}
		\bigr)
		\colon G(K_{s t_j}^{\Ind})
	\]
	for
	\(
		h
		\colon \stlin_{G, T, \Phi}(
			K^{((s t_i)^\infty)}
		)
	\) by lemmas \ref{gst-action} and \ref{k2-central}. We have
	\begin{align*}
		\stmap\bigl(
			\up{\widetilde g_j}{\bigl(
				\bigl(
					(t_j)_*
					\circ \delta
					\circ t_i^*
				\bigr)(h)
			\bigr)}
		\bigr)
		&= (t_j)_*(g)\,
		\bigl(
			{\stmap}
			\circ (t_j)_*
			\circ \delta
			\circ t_i^*
		\bigr)(h)\,
		(t_j)_*(g)^{-1}
	\\
		&= (t_j)_*(g)\,
		\bigl(
			(t_j)_*
			\circ \delta
			\circ t_i^*
			\circ {\stmap}
		\bigr)(h)\,
		(t_j)_*(g)^{-1}
	\\
		&= (t_j)_*\bigl(
			\up g {\bigl(
				(\delta \circ t_i^* \circ {\stmap})(h)
			\bigr)}
		\bigr)
	\\
		&= \bigl(
			(t_j)_*
			\circ \delta
			\circ t_i^*
		\bigr)\bigl(\up{(t_i)_*(g)}{\stmap(h)}\bigr)
	\\
		&= \bigl(
			{\stmap}
			\circ (t_j)_*
			\circ \delta
			\circ t_i^*
		\bigr)(\up {\widetilde g_i} h).
	\end{align*}

	The constructed
	\(
		A_g
	\) is uniquely determined by the requirements from the statement. It is independent on the choices since for any
	\(
		t' \in K
	\) and
	\(
		\widetilde g'
		\in \gstlin_{G, T, \Phi}(K_{t' s})
	\) lifting
	\(
		t'_*(g)
	\) we may consider the extended family
	\(
		t_1, \ldots, t_n, t'
	\) and by the proved results
	\[
		A_g\bigl({t'}^*(h)\bigr)
		= {t'}^*(\up {\widetilde g'} h).
	\]
	It is easy to see that
	\(
		A_g
	\) is preserved under base changes and is extra-natural on \(s\). Moreover,
	\(
		\stmap(A_g(h)) = \up g {\stmap(h)}
	\) by lemma \ref{gen-rel}(3) and since for any \(i\) we have
	\[
		\stmap\bigl(
			A_g\bigl(t_i^*(h)\bigr)
		\bigr)
		= \stmap\bigl(
			t_i^*(\up {\widetilde g_i} h)
		\bigr)
		= t_i^*\bigl(
			\up{(t_i)_*(g)}{\stmap(h)}
		\bigr)
		= \up g {
			\stmap\bigl(t_i^*(h)\bigr)
		}.
	\]

	Finally, by the uniqueness
	\(
		A_{g g'} = A_g A_{g'}
	\) and
	\(
		A_1 = \id
	\). It follows that
	\(
		A_g
	\) are automorphisms.
\end{proof}

\begin{prop} \label{local-action}
	The homomorphism
	\[
		\delta \circ \stmap
		\colon \stlin_{G, T, \Phi}(\R^{(s^\infty)})
		\to G(\R_s^{\Ind})
	\]
	is a crossed product in a unique way. The action of
	\(
		G(\R_s^{\Ind})
	\) on
	\(
		\stlin_{G, T, \Phi}(\R^{(s^\infty)})
	\) is preserved under base changes and is extra-natural on \(s\). Moreover, this action induces the action of
	\(
		\gstlin_{G, T, \Phi}(\R_s^{\Ind})
	\) from lemma \ref{gst-action} and the action of
	\(
		G(K_s)
	\) from lemma \ref{point-action}. The homomorphism
	\[
		\stmap
		\colon \stlin_{G, T, \Phi}(\R^{(s^\infty)})
		\to G(\R^{(s^\infty)})
	\]
	is
	\(
		G(\R_s^{\Ind})
	\)-equivariant.
\end{prop}
\begin{proof}
	By lemmas \ref{fp-change} and \ref{point-action} there are morphisms
	\[
		A_{s, u}
		\colon \Spec(R)
			\times \stlin_{G, T, \Phi}(\R^{(s^\infty)})
		\to \stlin_{G, T, \Phi}(\R^{(s^\infty)})
	\]
	for all
	\(
		R \in \Ring_K^\fp
	\) and morphisms
	\(
		u \colon \Spec(R) \to G(\R_s^{\Ind})
	\) such that
	\begin{itemize}

		\item
		\(
			A_{s, u \circ f}(p, x) = A_{s, u}(f(p), x)
		\) for any morphisms
		\(
			\Spec(S)
			\xrightarrow f \Spec(R)
			\xrightarrow u G(\R_s^{\Ind})
		\);

		\item
		\(
			A_{s, u}(p, x y)
			= A_{s, u}(p, x)\, A_{s, u}(p, y)
		\);

		\item
		\(
			A_{s, e}(*, x) = x
		\), where
		\(
			e \colon \Spec(K) \to G(\R_s^{\Ind})
		\) is the unit global element;

		\item
		\(
			A_{s, u}(p, A_{s, v}(q, x))
			= A_{s, m \circ (u \times v)}((p, q), x)
		\), where
		\(
			m
			\colon G(\R_s^{\Ind})
				\times G(\R_s^{\Ind})
			\to G(\R_s^{\Ind})
		\) is the multiplication;

		\item
		\(
			A_{s, {\stmap} \circ u}(p, x)
			= \up {u(p)} x
		\) for
		\(
			u
			\colon \Spec(R)
			\to \gstlin_{G, T, \Phi}(\R_s^{\Ind})
		\);

		\item
		\(
			\stmap(A_{s, u}(p, x))
			= \up{u(p)}{\stmap(x)}
		\);

		\item
		\(
			A_{s, u}\bigl(p, t^*(x)\bigr)
			= t^*\bigl(
				A_{t s, t_* \circ u}(p, x)
			\bigr)
		\).

	\end{itemize}

	Recall that
	\[
		G
		\cong \Spec\bigl(
			K[x_1, \ldots, x_n]
			/ (f_1, \ldots, f_m)
		\bigr)
	\]
	is a affine group scheme of finite presentation. For any
	\(
		N \geq 0
	\) take sufficiently large
	\(
		M \gg N
	\) such that
	\(
		f_j
		\colon (\frac \R {s^N})^n
		\to \frac \R {s^M}
	\) are well-defined for all \(s\) simultaneously. Consider objects
	\(
		G(\frac \R {s^N}) \in \mathbf P_K
	\), they are subobjects of
	\(
		(\frac \R {s^N})^n
	\) defined by the equations
	\(
		f_j(\vec x) = 0
	\). These objects actually depend on the choice of \(M\), but we only need that they are functorial on \(s\) and
	\(
		G(\R_s^{\Ind})
		= \varinjlim_N
			G(\frac \R {s^N})
	\). All of them are representable, i.e.
	\(
		G(\frac \R {s^N}) \cong \Spec(R_{s, N})
	\) for some
	\(
		R_{s, N} \in \Ring_K^\fp
	\). Also, the group operation on \(G\) induces morphisms
	\(
		G(\frac \R {s^N})
			\times G(\frac \R {s^N})
		\to G(\frac \R {s^L})
	\) for some
	\(
		L \gg N
	\) depending on \(N\), as well as the identity elements
	\(
		1 \in G(\frac \R {s^N})
	\). Similarly, we may construct the objects
	\[\textstyle
		d
		\colon L(\frac \R {s^N})
		\to G(\frac \R {s^T}),
	\quad
		t_\alpha
		\colon P_\alpha(\frac \R {s^N})
		\to G(\frac \R {s^T})
	\]
	together with the canonical morphisms for
	\(
		T \gg N
	\).

	Now we have the morphisms
	\[
		A_{s, N}
		= A_{s, \Spec(R_{s, N}) \to G(\R_s^{\Ind})}
		\colon \Spec(R_{s, N})
			\times \stlin_{G, T, \Phi}(\R^{(s^\infty)})
		\to \stlin_{G, T, \Phi}(\R^{(s^\infty)})
	\]
	such that
	\begin{itemize}

		\item
		\(
			A_{s, N}(g, x)
			= A_{s, N + 1}(\iota_N(g), x)
		\), where
		\(
			\iota_N
			\colon \Spec(R_{s, N})
			\to \Spec(R_{s, N + 1})
		\) are the canonical morphisms;

		\item
		\(
			A_{s, N}(g, x y)
			= A_{s, N}(g, x)\, A_{s, N}(g, y)
		\);

		\item
		\(
			A_{s, N}(1, x) = x
		\);

		\item
		\(
			A_{s, N}(g, A_{s, N}(h, x))
			= A_{s, L}(g h, x)
		\);

		\item
		\(
			A_{s, T}\bigl(
				d(g)\,
				\prod_{i = 1}^n
					t_{\alpha_i}(p_i),
				x
			\bigr)
			= \up {
				d(g)\,
				\prod_{i = 1}^n
					x_{\alpha_i}(p_i)
			} x
		\) for all
		\(
			n \geq 0
		\) and roots
		\(
			\alpha_1, \ldots, \alpha_n \in \Phi
		\), where
		\(
			g \colon L(\frac \R {s^N})
		\) and
		\(
			p_i \colon P_{\alpha_i}(\frac \R {s^N})
		\);

		\item
		\(
			\stmap(A_{s, N}(g, x)) = \up g {\stmap(x)}
		\);

		\item
		\(
			A_{s, N}
		\) are extra-natural on \(s\).

	\end{itemize}

	Taking direct limits, we finally get an action of
	\(
		G(\R_s^{\Ind})
	\) on
	\(
		\stlin_{G, T, \Phi}(\R^{(s^\infty)})
	\) by automorphisms such that it induces the action of
	\(
		\gstlin_{G, T, \Phi}(\R_s^{\Ind})
	\) from lemma \ref{gst-action}, it is extra-natural on \(s\), and
	\(
		\stmap
		\colon \stlin_{G, T, \Phi}(\R^{(s^\infty)})
		\to G(\R^{(s^\infty)})
	\) is
	\(
		G(\R_s^{\Ind})
	\)-equivariant. Since any morphism
	\(
		\Spec(K) \to G(\R_s^{\Ind})
	\) factors through some
	\(
		\Spec(R_{s, N})
	\), the constructed action also induces the action from lemma \ref{point-action}.

	Let us check that
	\(
		\delta \circ {\stmap}
		\colon \stlin_{G, T, \Phi}(\R^{(s^\infty)})
		\to G(\R_s^{\Ind})
	\) is a crossed module. Indeed, both \(\delta\) and \(\stmap\) are
	\(
		G(\R_s^{\Ind})
	\)-equivariant. Also,
	\[
		\up x y
		= \up {\delta(x)} y
		= \up {\stmap(\delta(x))} y
		= \up {\delta(\stmap(x))} y
	\]
	for
	\(
		x, y
		\colon \stlin_{G, T, \Phi}(\R^{(s^\infty)})
	\).

	Finally, the action is unique and is preserved under base changes by lemma \ref{perf-x-mod}(3).
\end{proof}

\section{Locally isotropic Steinberg groups}

Recall that
\(
	\delta \colon G(\R^{(s^\infty)}) \to G(\R_s^{\Ind})
\) is a crossed module. Since it is extra-natural on \(s\), the homomorphism
\[
	s^*
	\colon G(\R^{(s^\infty)})
	\to G(\R)
\]
is also a crossed module, where the action of
\(
	G(\R)
\) on
\(
	G(\R^{(s^\infty)})
\) is the pullback of the action of
\(
	G(\R_s^{\Ind})
\). The properties of the diagram from the previous section (i.e. proposition \ref{local-action}) imply that both
\[
	\stmap
	\colon \stlin_{G, T, \Phi}(\R^{(s^\infty)})
	\to G(\R^{(s^\infty)}),
\quad
	s^* \circ {\stmap}
	\colon \stlin_{G, T, \Phi}(\R^{(s^\infty)})
	\to G(\R)
\]
are also crossed modules, where the actions are induced from the action of
\(
	G(\R_s^{\Ind})
\).

\begin{lemma} \label{st-unique}
	Let
	\(
		(T, \Phi)
	\) and
	\(
		(T', \Phi')
	\) be two isotropic pinnings of
	\(
		G_{K_s}
	\) of rank at least \(3\). Then there is a unique isomorphism
	\[
		F
		\colon \stlin_{G, T', \Phi'}(\R^{(s^\infty)})
		\to \stlin_{G, T, \Phi}(\R^{(s^\infty)})
	\]
	of crossed modules over
	\(
		G(\R_s^{\Ind})
	\). This isomorphism commutes with the homomorphisms
	\(
		t^*
	\), is preserved under base changes, and satisfies
	\(
		{\stmap} \circ F
		= \stmap'
		\colon \stlin_{G, T', \Phi'}(\R^{(s^\infty)})
		\to G(\R^{(s^\infty)})
	\).
\end{lemma}
\begin{proof}
	The required isomorphism is unique by lemma \ref{perf-x-mod}(4), and by the same argument is is preserved under base changes. Such an isomorphism clearly exists if
	\(
		(T, \Phi) = \up g {(T', \Phi')}
	\) for some
	\(
		g \in G(K_s)
	\), namely, it is given by
	\[
		F(x'_\alpha(p)) = x_{\up g \alpha}(\up g p)^g,
	\]
	where
	\(
		\up g {({-})}
		\colon P'_\alpha
		\to P_{\up g \alpha}
	\) is the homomorphism such that
	\(
		\up g {t'_\alpha(p)}
		= t_{\up g \alpha}(\up g p)
	\), and the outer automorphism
	\(
		({-})^g = \up{g^{-1}}({-})
	\) is taken from lemma \ref{point-action}. Note that \(F\) preserves the Steinberg relations since they are preserved by
	\(
		{\stmap} \circ F = \stmap'
	\), and \(F\) preserves the action of
	\(
		G(\R_s^{\Ind})
	\) by lemma \ref{perf-x-mod}(3). Here we denote the root morphisms associated with
	\(
		(T', \Phi')
	\) by
	\[
		t'_\alpha \colon P'_\alpha \to G_{K_s},
	\quad
		x'_\alpha
		\colon P'_\alpha(\R^{(s^\infty)})
		\to \stlin_{G, T, \Phi}(\R^{(s^\infty)})
	\]
	to distinguish them from the corresponding morphisms associated with
	\(
		(T, \Phi)
	\).

	Now consider the case
	\(
		(T, \Phi) \subseteq (T', \Phi')
	\). There are ``inclusion'' homomorphisms
	\(
		\iota_\alpha
		\colon P'_\alpha
		\to P_{\Image(\alpha)}
	\) such that
	\[
		t'_\alpha(p)
		= t_{\Image(\alpha)}(\iota_\alpha(p))
		\colon P_\alpha
		\to G_{K_s}
	\]
	for
	\(
		\alpha \in \Phi'
	\) with non-zero image in \(\Phi\). By proposition \ref{elim-bij} there is a unique isomorphism \(F\) such that
	\(
		F(x'_\alpha(p))
		= x_{\Image(\alpha)}(\iota_\alpha(p))
	\) if the image of \(\alpha\) is non-zero, it is again
	\(
		G(\R_s^{\Ind})
	\)-equivariant by lemma \ref{perf-x-mod}(3).

	In these two cases
	\[
		F \circ t^*
		= t^* \circ F
		\colon \stlin_{G, T', \Phi'}(
			\R^{((t s)^\infty)}
		)
		\to \stlin_{G, T, \Phi}(\R^{(s^\infty)}).
	\]
	for all
	\(
		t \in K
	\) by construction.

	Finally, consider the general case. Choose elements
	\(
		t_1, \ldots, t_n \in K
	\) generating the unit ideal of
	\(
		K_s
	\), isotropic pinnings
	\(
		(T_i, \Phi_i)
	\) of
	\(
		G_{K_{t_i s}}
	\), and elements
	\(
		g_i, g'_i \in G(K_{t_i s})
	\) such that
	\[
		(T_{t_i}, \Phi)
		\subseteq \up{g_i}{(T_i, \Phi_i)},
	\quad
		(T'_{t_i}, \Phi')
		\subseteq \up{g'_i}{(T_i, \Phi_i)}.
	\]
	These object may be constructed as follows. For every
	\(
		s \notin \mathfrak p \in \Spec(K)
	\) choose a maximal isotropic pinning of
	\(
		G_{\mathfrak p}
	\), its conjugates contain both
	\(
		(T_{\mathfrak p}, \Phi)
	\) and
	\(
		(T'_{\mathfrak p}, \Phi')
	\). Then extend these isotropic pinnings and conjugating elements to an open neighborhood
	\(
		\mathfrak p
		\in \mathcal D(t_i)
		\subseteq \Spec(K)
	\) by a direct limit argument. Finally, take finitely many
	\(
		t_i
	\) covering
	\(
		\mathcal D(s)
	\) by its quasi-compactness.

	We already know that
	\(
		\stlin_{G, T, \Phi}(\R^{(r t_i s)})
		\cong \stlin_{G, T_i, \Phi_i}(\R^{(r t_i s)})
		\cong \stlin_{G, T', \Phi'}(\R^{(r t_i s)})
	\) as crossed modules over
	\(
		G(\R_{r t_i s}^{\Ind})
	\) for all \(r\) and these isomorphisms commute with
	\(
		r^*
	\). By proposition \ref{st-cosheaf} we get the required isomorphism
	\[
		F
		\colon \stlin_{G, T', \Phi'}(\R^{(r s)})
		\to \stlin_{G, T, \Phi}(\R^{(r s)})
	\]
	of group objects for all \(r\) such that
	\(
		{\stmap} \circ F = \stmap'
	\) and
	\(
		F \circ {r'}^* = {r'}^* \circ F
	\). Again \(F\) is
	\(
		G(\R_s^{\Ind})
	\)-equivariant by lemma \ref{perf-x-mod}(3).
\end{proof}

We are ready to construct the Steinberg group
\(
	\stlin_G(\R)
\) as an group object in
\(
	\mathbf U_K
\).

\begin{theorem} \label{global-st}
	Let \(K\) be a unital ring and \(G\) be a reductive group scheme over \(K\) of local isotropic rank at least \(3\). Then there exists a functor
	\(
		s \mapsto \stlin_G(\R^{(s^\infty)})
	\) from the partially ordered set of principal open subsets of
	\(
		\Spec(K)
	\) to the category of group objects in
	\(
		\mathbf U_K
	\) together with a homomorphism
	\[
		\stmap
		\colon \stlin_G(\R^{(s^\infty)})
		\to G(\R^{(s^\infty)})
	\]
	natural on \(s\) and with a crossed
	\(
		G(\R_s^{\Ind})
	\)-module structure of on
	\(
		\stlin_G(\R^{(s^\infty)})
	\) extra-natural on \(s\). These data have the following properties.
	\begin{itemize}

		\item
		\(\stmap\) is
		\(
			G(\R_s^{\Ind})
		\)-equivariant.

		\item
		\(
			\stlin_G(\R^{(s^\infty)})
		\) is perfect.

		\item
		The groups
		\(
			\stlin_G(\R^{(s^\infty)})
		\) have the cosheaf property, i.e. the canonical homomorphism
		\[
			\colim\bigl(
				\stlin_G(
					\R^{((r_i r_j s)^\infty)}
				)
				\rightrightarrows \stlin_G(
					\R^{((r_i s)^\infty)}
				)
			\bigr)
			\to \stlin_G(\R^{(s^\infty)})
		\]
		of crossed modules over
		\(
			G(\R_s^{\Ind})
		\) (or just over
		\(
			G(\R)
		\)) is an isomorphism for any sequence
		\(
			t_1, \ldots, t_n \in K
		\) such that
		\(
			\sum_i K_s t = K_s
		\). Moreover,
		\[
			[t_i]^*
			\colon \stlin_G(\R^{((t_i s)^\infty)})
			\to \stlin_G(\R^{(s^\infty)})
		\]
		generate the latter group object and the only relations between them are listed in proposition \ref{st-cosheaf}.

		\item
		If
		\(
			(T, \Phi)
		\) is an isotropic pinning of
		\(
			G_{K_s}
		\) of rank at least \(3\), then there exists an isomorphism
		\[
			F
			\colon \stlin_{G, T, \Phi}(\R^{(s^\infty)})
			\cong \stlin_G(\R^{(s^\infty)})
		\]
		of group objects such that
		\(
			{\stmap} \circ F = \stmap
		\). Such isomorphism is unique and
		\(
			G(\R_s^{\Ind})
		\)-equivariant (by lemma \ref{perf-x-mod}(4)), it is also natural on \(s\).

		\item
		The functor
		\(
			s \to \stlin_G(\R^{(s^\infty)})
		\) with a natural homomorphism \(\stmap\) and an extra-natural crossed module structure satisfying all the above properties is unique up to a unique natural isomorphism.

		\item
		For any unital \(K\)-algebra \(R\) there is a unique isomorphism
		\[
			\stlin_G(\R^{(s^\infty)})_R
			\to \stlin_{G_R}(\R_R^{(s^\infty)})
		\]
		of group objects over
		\(
			G_R(\R_R^{(s^\infty)})
		\), it is
		\(
			G_R((\R_R)_s^{\Ind})
		\)-equivariant and natural on \(s\). It also commutes with the isomorphisms \(F\).

	\end{itemize}
\end{theorem}
\begin{proof}
	The uniqueness and the last property follow from lemma \ref{perf-x-mod}(4). We already have the groups
	\(
		\stlin_G(\R^{(s^\infty)})
		\cong \stlin_{G, T, \Phi}(\R^{(s^\infty)})
	\) for such \(s\) that there exists an isotropic pinning
	\(
		(T, \Phi)
	\) of
	\(
		G_{K_s}
	\) of rank at least \(3\), they are independent on the choice of
	\(
		(T, \Phi)
	\) by lemma \ref{st-unique}.

	Now choose
	\(
		t_1, \ldots, t_n \in K
	\) generating the unit ideal such that
	\(
		G_{K_{t_i}}
	\) have isotropic pinnings of rank at least \(3\). Let
	\(
		\stlin_G(\R^{(s^\infty)})
	\) be the colimit
	\[
		\colim\bigl(
			\stlin_G(
				\R^{((t_i t_j s)^\infty)}
			)
			\rightrightarrows \stlin_G(
				\R^{((t_i s)^\infty)}
			)
		\bigr)
	\]
	in the category of crossed modules over
	\(
		G(\R_s^{\Ind})
	\). Clearly, this object is functorial on \(s\), there exists the required
	\(
		G(\R_s^{\Ind})
	\)-equivariant homomorphism
	\[
		\stmap
		\colon \stlin_G(\R^{(s^\infty)})
		\to G(\R^{(s^\infty)}),
	\]
	and the cosheaf property holds by proposition \ref{st-cosheaf}. The same proposition implies that this construction is a continuation of the already defined
	\(
		\stlin_G(\R^{(s^\infty)})
	\) for \(s\) with isotropic pinnings on
	\(
		G_{K_s}
	\) of rank at least \(3\). Finally,
	\(
		\stlin_G(\R^{(s^\infty)})
	\) is perfect since it is generated as a group object by homomorphisms
	\[
		[t_i]^*
		\colon \stlin_G(\R^{((t_i s)^\infty)})
		\to \stlin_G(\R^{(s^\infty)})
	\]
	from perfect group objects.
\end{proof}

Now suppose that
\(
	\stlin_G(\R) \in \Ex(\Ind(\mathbf P_K))
\) up to an isomorphism. This condition clearly holds if \(G\) has an isotropic pinning of rank at least \(3\) since in this case
\(
	\stlin_G(\R) = \stlin_{G, T, \Phi}(\R)
\) is given just by generators and relations. Also, this condition is preserved under base changes. Under this condition the abstract Steinberg group functor
\[
	\stlin_G \colon \Ring_K \to \Group,\,
	R \mapsto \ev_R(\stlin_G(\R))
\]
is well-defined.

\begin{theorem} \label{abstr-xmod}
	Let \(K\) be a unital ring and \(G\) be a reductive group scheme over \(K\) of local isotropic rank at least \(3\). Suppose that
	\(
		\stlin_G(\R) \in \Ex(\Ind(\Set))
	\) up to an isomorphism, e.g. that \(G\) has an isotropic pinning of rank at least \(3\). Then the Steinberg group functor
	\(
		\stlin_G
	\) commutes with direct limits. The groups
	\(
		\stlin_G(R)
	\) are perfect and the homomorphisms
	\(
		\stmap \colon \stlin_G(R) \to G(R)
	\) are crossed modules in a unique way, so the action of
	\(
		G(R)
	\) on
	\(
		\stlin_G(R)
	\) is natural on \(R\).
\end{theorem}
\begin{proof}
	This easily follows from the properties of
	\(
		\ev_R
	\) and theorem \ref{global-st}.
\end{proof}

\bibliographystyle{plain}
\bibliography{references}

\end{document}